\documentclass[11pt]{article}
\usepackage[a4paper,margin=26mm]{geometry}
\usepackage[T1]{fontenc}
\usepackage{lmodern}
\usepackage{microtype}
\usepackage{xspace}
\usepackage{amsmath,amssymb,amsthm,mathtools,bm}
\usepackage{booktabs,tabularx,array,multirow}
\usepackage{graphicx}
\usepackage{tikz}
\usetikzlibrary{arrows.meta}
\usepackage{pgfplots}
\pgfplotsset{compat=1.18}
\usepackage{float}
\usepackage{xcolor}
\usepackage{enumitem}
\usepackage{placeins}
\usepackage[numbers,sort&compress]{natbib}
\usepackage[colorlinks=true,linkcolor=blue!45!black,citecolor=blue!45!black,
  urlcolor=blue!45!black,
  pdftitle={RUPA: Nonlinear volume consistency, constraint geometry and singular penalty limits in finite elements},
  pdfauthor={Yanlin Liu, Chao Huang, Kaixiang Yao, Yao Shen},
  pdfsubject={Quadrature consistency and nonlinear cell-volume constraints in finite elasticity},
  pdfkeywords={finite elasticity, numerical integration, nonlinear constraints, Holder error bounds, singular penalty limits, tensor-product elements}]{hyperref}
\usepackage[nameinlink,capitalise]{cleveref}

\newcommand{\R}{\mathbb{R}}
\newcommand{\GLp}{\mathrm{GL}^{+}(3)}
\newcommand{\cof}{\operatorname{cof}}
\newcommand{\tr}{\operatorname{tr}}
\newcommand{\dev}{\operatorname{dev}}
\newcommand{\sym}{\operatorname{sym}}

\newcommand{\rank}{\operatorname{rank}}
\newcommand{\spann}{\operatorname{span}}

\newcommand{\Qfour}{Q_4}
\newcommand{\That}{\widehat T}
\newcommand{\Pone}{\mathbb P_1}
\newcommand{\Ptwo}{\mathbb P_2}

\newtheorem{theorem}{Theorem}
\newtheorem{proposition}[theorem]{Proposition}
\newtheorem{lemma}[theorem]{Lemma}
\newtheorem{corollary}[theorem]{Corollary}
\theoremstyle{definition}

\theoremstyle{remark}

\AddToHook{env/theorem/begin}{\crefalias{theorem}{theorem}}
\AddToHook{env/proposition/begin}{\crefalias{theorem}{proposition}}
\AddToHook{env/lemma/begin}{\crefalias{theorem}{lemma}}
\AddToHook{env/corollary/begin}{\crefalias{theorem}{corollary}}
\AddToHook{env/definition/begin}{\crefalias{theorem}{definition}}
\AddToHook{env/remark/begin}{\crefalias{theorem}{remark}}

\AddToHook{cmd/appendix/after}{%
  \crefalias{section}{appendix}%
  \crefalias{subsection}{subappendix}%
  \crefalias{subsubsection}{subsubappendix}}

\title{\textbf{RUPA: Nonlinear volume consistency,\\
constraint geometry and singular penalty\\
limits in finite elements}}
\author{Yanlin Liu$^1$ \and Chao Huang$^{2,\dagger}$ \and
Kaixiang Yao$^{3,\dagger}$ \and Yao Shen$^{2,*}$}
\date{\small $^1$The University of Melbourne, Melbourne, Australia\\
$^2$Shanghai Jiao Tong University, Shanghai, China\\
$^3$Southern University of Science and Technology, Shenzhen, China\\
$^\dagger$Chao Huang and Kaixiang Yao share second authorship.\\
$^*$Corresponding author: \texttt{yaoshen@sjtu.edu.cn}}

\begin{document}
\maketitle

\begin{abstract}
Volume quadrature can change nonlinear finite-element constraints
while preserving their reference-state derivatives.
We connect an explicit determinant defect to feasible-set geometry
and singular mechanical response.
For affine tensor elements of coordinate degree $p\ge3$ with
$n\ge p+1$ Gauss points per coordinate, determinant volume is exact
precisely when $2n\ge3p$. Below that threshold we construct a
boundary-fixed cubic defect at every order.
The same directions yield a full-space cube-root residual--distance
bound under explicit cell-support, coefficient and physical-norm
assumptions, with mesh-uniform upper constants at fixed order.
With all cell-pressure equations retained, the volume Jacobian
gains rank at nearby feasible states despite agreement through second
derivatives at rest; the cube-root exponent is sharp on each fixed mesh.
A general localized-minimum theorem shows that the first reduced
compatibility term contributes its weighted square to the leading
energy in a joint small-load, large-bulk limit.
Cubic and quadratic defects therefore produce sextic and quartic terms.
The full cubic-element interior space has an exact normal form and
sharp local error exponents. Curved quadratic tetrahedra supply the
second-order contrast, a sparse rational witness and an exact
four-Jacobian volume formula.
Finite-strain tensor calculations illustrate normalized response
separation, with explicit stationary-point, extreme-bulk and
pressure-recovery qualifications.
The constructive correction preserves exact cell volumes, so
quadrature feasibility remains distinct from physical volume preservation.
\end{abstract}


\section{Introduction}
\label{sec:introduction}

The volume of a deformed cell is determined by its boundary values.
For a fully clamped body, the exact cell-volume changes therefore sum
to zero. A quadrature rule can lose this identity while reproducing
the reference volume and several derivatives. When every cell-pressure
equation is retained, the lost identity can become an additional
nonlinear constraint. RUPA---\emph{Regularity under Underintegration
and Penalty Asymptotics}---traces this change from the element
construction through feasible-set geometry to mechanical response.
The comparison keeps the material energy fixed and changes only
the integration of one scalar volume equation per cell.

A three-amplitude restriction makes the mechanism concrete.
For the tensor elements constructed below, choose displacement fields
\(U_1,U_2,U_3\), each zero on every face of one cell, and set
\(u(a)=a_1U_1+a_2U_2+a_3U_3\), where \(a=(a_1,a_2,a_3)\in\R^3\).
The exact and quadrature cell-volume changes, denoted here by
\(q^I\) and \(q^Q\), satisfy
\[
 q^I(a)=0,\qquad q^Q(a)=\gamma a_1a_2a_3,\qquad \gamma\ne0.
\]
At \(a=0\), their values and first two derivatives agree.
At \(a=(0,\tau,\tau)\), both volumes are still zero, while the
quadrature derivative in the missing direction is
\(\partial_{a_1}q^Q=\gamma\tau^2\).
The quadrature constraint excludes motions allowed by exact volume.
Figure~\ref{fig:cubic-volume-mechanism} shows this restriction;
the assembled theory retains the complete displacement space.

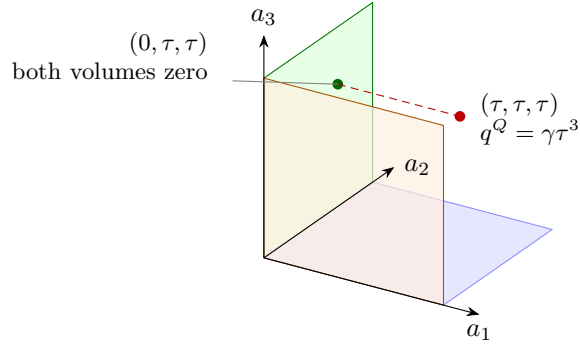
\begin{figure}[H]
\centering
\begin{tikzpicture}[x={(1.9cm,-.50cm)},y={(1.15cm,.80cm)},z={(0cm,1.9cm)},
  line join=round,>=Stealth,font=\small]
  \filldraw[fill=blue!10,draw=blue!45] (0,0,0)--(1.25,0,0)--(1.25,1.25,0)--(0,1.25,0)--cycle;
  \filldraw[fill=green!10,draw=green!45!black] (0,0,0)--(0,1.25,0)--(0,1.25,1.25)--(0,0,1.25)--cycle;
  \filldraw[fill=orange!12,draw=orange!65!black,fill opacity=.7]
    (0,0,0)--(1.25,0,0)--(1.25,0,1.25)--(0,0,1.25)--cycle;
  \draw[->] (0,0,0)--(1.5,0,0) node[below] {$a_1$};
  \draw[->] (0,0,0)--(0,1.5,0) node[right] {$a_2$};
  \draw[->] (0,0,0)--(0,0,1.55) node[above] {$a_3$};
  \fill[green!40!black] (0,.85,.85) circle[radius=2pt];
  \fill[red!75!black] (.85,.85,.85) circle[radius=2pt];
  \draw[densely dashed,red!70!black] (0,.85,.85)--(.85,.85,.85);
  \draw[gray] (0,.85,.85)--(-.4,.3,1.0);
  \node[anchor=east,align=right,font=\footnotesize] at (-.5,.3,1.15)
    {$(0,\tau,\tau)$\\both volumes zero};
  \node[anchor=west,align=left,font=\footnotesize] at (.92,.85,.85)
    {$(\tau,\tau,\tau)$\\$q^Q=\gamma\tau^3$};
\end{tikzpicture}
\caption{The cubic defect on the constructed three-direction subspace.
The three coordinate planes, drawn as finite patches, form the
quadrature-feasible set; exact volume
is unchanged throughout this amplitude space. The diagonal ray has cubic
residual and linear distance to those planes. At $(0,\tau,\tau)$ the
quadrature derivative in the $a_1$ direction is nonzero for $\tau\ne0$.
This is an algebraic restriction of the tensor construction; physical
admissibility uses small amplitudes, and the later distance theorem
concerns the full displacement space.}
\label{fig:cubic-volume-mechanism}
\end{figure}

The first part of the paper determines exactly when such a
boundary-fixed defect exists. Let \(\mathbb Q_p\) denote polynomials
of degree at most \(p\) in each coordinate, and let \(n\) be the Gauss
count per coordinate. For affine tensor elements with \(p\ge3\) and
\(n\ge p+1\), determinant volume is exact precisely when \(2n\ge3p\).
Every remaining pair admits an explicit direction triple with a
positive first-missed-moment coefficient. The anisotropic classification
also identifies cancellation caused by low transverse order.
This converse construction goes beyond degree counting.

The same directions give a constructive cube-root bound on distance
to the quadrature-defined feasible set. Four sign-selected evaluations
produce a scalar zero; independent cell support combines the corrections.
Under explicit coefficient and physical-norm bounds, the estimate in
\(H^1\), which measures displacement and its gradient, is uniform under
refinement at fixed order. Separate maximum-strain neighborhoods ensure
admissibility. The correction preserves all cell-face traces and hence
all exact cell volumes: it can remove quadrature residuals while
retaining an existing physical volume error.

On a connected, fully clamped tensor mesh containing a cell with this
defect, the volume maps agree
through second derivatives at rest, yet the quadrature Jacobian gains
rank at nearby common feasible states. The surviving aggregate cubic
proves fixed-mesh sharpness of the distance exponent and supplies the
mechanical consequence. More generally, let \(k\) be the degree of the
first nonzero compatibility term after the independent constraint
equations have been solved locally. With load amplitude \(\varepsilon\)
and bulk coefficient proportional to \(\varepsilon^{-2k+2}\), a
singular-penalty theorem gives an additional leading energy equal to
the weighted square of that term. The tensor defect is cubic and
produces a sextic term. At fixed finite bulk, the two models have the
same leading small-load response.

Two exact models resolve the geometry further. The full \(\mathbb Q_3\)
interior space has a matrix normal form and sharp local linear,
square-root and cube-root error exponents. Curved ten-node quadratic
tetrahedra (Tet10) supply a second-order contrast, an explicit
three-node rational obstruction, a quartic penalty consequence and
exact volume recovery from four existing Jacobian matrices.
The numerical evidence follows the same construction--geometry--response
chain. It retains the distinction between computed stationary points
and the localized minima of the penalty theorem, including extreme
bulk ratios and the nearly flat fourth-order response branch.

The determinant identities are classical null-Lagrangian results
\citep{ball1976convexity,olver1988null}. The companion study MORTIS
uses the same polarized defect for compatible cofactor reference forms
\citep{liu2026twominor}. Scalar harmonic-polynomial error bounds
\citep[Theorem~3.1]{badger2017harmonic}, higher-order regularity and
perturbation theory
\citep{izmailov2001error,bonnans1998perturbations,dontchev2014implicit}
supply established foundations. The contributions here are the explicit
finite-element defect construction, quantitative cellwise assembly,
feasible rank lifting and linked singular mechanical limit.

Integration-dependent locking and stability are established concerns
in finite elasticity \citep{maas2016tet10,foulk2021extending};
high-order analysis treats overintegration and bulk amplification
\citep[Sections~4.6.4 and 5.3]{heisserer2008highorder}.
Nonlinear Poisson stiffening concerns shell kinematics unable to satisfy
nonlinear isochoric conditions at integration points
\citep{willmann2023poisson}. Here one integral equation is retained per
cell, and cubature destroys a boundary-induced relation among those
equations despite the stated derivative agreement.
Pressure robustness and compatible mixed formulations offer related
perspectives \citep{linke2014helmholtz,fu2025fourfield}.

After the setting, \cref{sec:tensor-admission,sec:tri-affine,sec:assembled-geometry,sec:singular-penalty}
develop the construction, feasible geometry and mechanical consequence.
The complete cubic-cell model and tensor observations follow in
\cref{sec:sharp-models,sec:evidence}.
\Cref{sec:cubature} develops the Tet10 contrast through exact restoration.
The appendices retain supporting tangent theory, rational certificates,
numerical protocols and constitutive derivations.

\section{Discrete volume maps and nonlinear feasibility}
\label{sec:discrete-volume-setting}

\subsection{Geometry, displacement spaces, and norms}

Let $\Omega\subset\R^3$ be a bounded Lipschitz reference domain, partitioned
into finitely many conforming cells $K_e$, $e=1,\ldots,m$. Each cell is the
image of a fixed parent cell $\widehat K_e$ under a one-to-one map
$X_e:\widehat K_e\to K_e$. We assume positive Jacobian determinants and
bounded first derivatives and inverse derivatives on each cell. Geometry
maps and quadrature rules remain fixed as the displacement is varied.

The space $\mathcal U_h$ consists of continuous finite-element
displacements with zero trace on the complete outer boundary. Its members
are piecewise differentiable and belong to
$H^1_0(\Omega;\R^3)\cap W^{1,\infty}(\Omega;\R^3)$.
Here $H^1$ consists of square-integrable fields with square-integrable first
weak derivatives; the subscript zero prescribes zero boundary trace.
$W^{1,\infty}$ additionally requires the field and its first derivatives
to be essentially bounded. The subscript $h$ identifies the mesh.
Free displacement coefficients identify $\mathcal U_h$ with $\R^N$, where
$N$ is their number. In the Lagrange elements below these are nodal
coefficients. We state explicitly when a
result concerns that coefficient norm or a physical Sobolev norm.
For symmetric forms or matrices, $A\preceq B$ means
$v^TAv\le v^TBv$ for every vector $v$; $A\succ0$ denotes positive
definiteness.

For a displacement $u\in\mathcal U_h$, write
$\widehat u_e=u\circ X_e$ for its pullback and
$y_e=X_e+\widehat u_e$ for the current parent-cell map. The parent and
physical gradients are related by
\begin{equation}
 J_{X,e}=D_\xi X_e,\qquad J_{y,e}=D_\xi y_e,\qquad
 F_e(u)=\nabla_X(X+u)=J_{y,e}J_{X,e}^{-1}.
 \label{eq:general-kinematics}
\end{equation}
The row index of a gradient is the vector component and the column index
is the differentiation coordinate. The deformation gradient $F_e$ acts on
physical material coordinates. The matrix $J_{X,e}$ belongs to the fixed
reference geometry.

For matrices $H,L\in\mathbb M:=\R^{3\times3}$, define the Frobenius product
and norm by $H:L=\operatorname{tr}(H^TL)$ and
$\|H\|_F^2=H:H$. Write $I_3$ for the identity matrix,
$\sym H=(H+H^T)/2$, and $\dev H=H-(\tr H)I_3/3$.
Matrix norms with subscript $2$ are induced Euclidean norms; vector
$2$-norms are Euclidean. Write
$\GLp=\{F\in\mathbb M:\det F>0\}$ for the orientation-preserving
invertible matrices. The physical norms used below are
\[
 \|u\|_{H^1}^2=\int_\Omega(|u|^2+\|\nabla_Xu\|_F^2)\,dX,
 \qquad
 \|\nabla_Xu\|_\infty=\operatorname*{ess\,sup}_{X\in\Omega}
                         \|\nabla_Xu(X)\|_F.
\]
A fixed Poincar\'e inequality has the form
$\|u\|_{L^2}\le C_P\|\nabla_Xu\|_{L^2}$ on the clamped space
\citep{brezis2011functional}. Uniform statements retain a common constant
$C_P$ across the stated mesh family. Constants in $O(\cdot)$ and
$\Theta(\cdot)$ estimates are independent of the indicated limiting
parameter; other dependencies are specified with each result.

\subsection{Exact and quadrature volume changes}

For a parent-cell integrand $f$, let
\[
 \mathcal I_e[f]=\int_{\widehat K_e}f(\xi)\,d\xi,\qquad
 \mathcal Q_e[f]=\sum_{q=1}^{n_e}w_{eq}f(\xi_{eq}),\qquad
 \mathcal L_e=\mathcal Q_e-\mathcal I_e.
\]
The points $\xi_{eq}$ and weights $w_{eq}$ define the prescribed cubature,
that is, multidimensional numerical integration. The cell-local spaces,
points, and weights are independent of the current nodal values. Positivity
of the weights is imposed in the results that use it. The exact reference
volume is $V_e=\mathcal I_e[\det J_{X,e}]>0$.

For $r\in\{Q,I\}$, put $\mathcal R_e^Q=\mathcal Q_e$ and
$\mathcal R_e^I=\mathcal I_e$. Define the two cell-volume changes and their
assembled vectors by
\begin{equation}
 q_{h,e}^r(u)=\mathcal R_e^r[\det J_{y,e}-\det J_{X,e}],
 \qquad
 q_h^r(u)=(q_{h,1}^r(u),\ldots,q_{h,m}^r(u))^T.
 \label{eq:global-volume-maps}
\end{equation}
We call $q_h^Q$ the raw volume map and $q_h^I$ the exact-volume comparison.
Both measure changes from the same reference configuration. Their feasible
sets impose cell-average volume preservation, which is distinct from the
pointwise equation $\det F_e=1$.

Determinants are classical null Lagrangians: their exact integrals are
determined by boundary traces \citep{ball1976convexity,olver1988null}.
For completeness, the cofactor is the matrix of signed second-order minors,
defined polynomially on all of $\mathbb M$. For invertible $F$,
$\cof F=(\det F)F^{-T}$. The row-wise Piola identity and the determinant's
homogeneity give
\[
 \det D_\xi y_e
 =\tfrac13\operatorname{div}_\xi\bigl((\cof D_\xi y_e)^Ty_e\bigr).
\]
The normal cofactor flux depends on tangential derivatives of the boundary
trace. Opposite orientations cancel internal-face contributions, and the
complete outer clamp therefore gives
\begin{equation}
 \mathbf1_m^Tq_h^I(u)=0,
 \label{eq:general-exact-volume-relation}
\end{equation}
where $\mathbf1_m$ is the vector of $m$ ones. The same argument shows that a
cell-interior displacement with zero trace on every cell face leaves each
exact cell volume unchanged. These signed polynomial identities also hold
outside the orientation-preserving neighborhood whenever the fixed
finite-element functions make the integrals well defined.

\subsection{The common polarized defect}

For $H,L\in\mathbb M$, define the cofactor polarization
$H\times L=\cof(H+L)-\cof H-\cof L$; it is bilinear and symmetric, and
$H\times H=2\cof H$. For three vector fields on a parent cell, define
\begin{equation}
 \mathcal T_e(z;u,v)
 =\mathcal L_e\bigl[D_\xi z:(D_\xi u\times D_\xi v)\bigr].
 \label{eq:common-polarized-defect}
\end{equation}
This trilinear form is symmetric in its three fields. If
$\Delta_e(y)=\mathcal L_e[\det D_\xi y]$, determinant polarization gives
\[
 D^2\Delta_e(y)[u,v]=\mathcal T_e(y;u,v),\qquad
 D^3\Delta_e(y)[z,u,v]=\mathcal T_e(z;u,v),\qquad
 \Delta_e(y)=\tfrac16\mathcal T_e(y;y,y).
\]
These are the coefficient identities of a cubic polynomial. They hold
independently of the spatial polynomial degree: increasing interpolation
order changes spatial moments, while the determinant remains cubic in
three-dimensional displacement coefficients.

The companion study MORTIS uses the same pairing with a fixed auxiliary
potential in its first slot \citep{liu2026twominor}. Indeed, the cofactor
functional's cubature defect $\mathcal L_e[D_\xi z:\cof D_\xi y]$ has Hessian
$\mathcal T_e(z;u,v)$. Here the current deformation occupies that slot,
and the question concerns the nonlinear constraints themselves.

\subsection{Feasibility and the complete pressure energy}

Define the algebraic feasible sets
$\mathcal F_h^r=\{u\in\mathcal U_h:q_h^r(u)=0\}$.
For any norm $\|\cdot\|_*$, write
$\operatorname{dist}_*(u,\mathcal F)=\inf_{v\in\mathcal F}\|u-v\|_*$.
At a feasible state, the kernel of the volume Jacobian is the linearized
feasible space. Tangent directions to actual feasible configurations are
limits of $(u_j-u_0)/t_j$ with $q_h^r(u_j)=0$ and $t_j\downarrow0$.
The two sets of directions agree under appropriate regularity assumptions
\citep{dontchev2014implicit}. A cubature-induced failure of that agreement
is one of the mechanisms examined below.

For physical interpretation, we use a small strain neighborhood
$\|\nabla_Xu\|_\infty<\rho<1$. A zero-trace field extends by zero to
$\R^3$ with the same gradient bound. Thus identity plus that extension is
globally one-to-one, with lower Lipschitz constant at least $1-\rho$.
The homotopy $I_3+t\nabla_Xu$ stays nonsingular for $0\le t\le1$, proving
positive orientation. The later correction theorem uses a smaller input
neighborhood and a specified larger output neighborhood.

Let $\mathsf V=\operatorname{diag}(V_1,\ldots,V_m)$ and define
$\|q\|_{\mathsf V^{-1}}=(q^T\mathsf V^{-1}q)^{1/2}$.
For bulk modulus $\kappa>0$, stationary elimination of one pressure
$p_e$ per cell from
\[
 E_h(u)+p^Tq_h^r(u)-\frac1{2\kappa}p^T\mathsf Vp
\]
gives $p=\kappa\mathsf V^{-1}q_h^r(u)$ and the complete reduced energy
\begin{equation}
 E_{h,\kappa}^r(u)
 =E_h(u)+\frac\kappa2 q_h^r(u)^T\mathsf V^{-1}q_h^r(u).
 \label{eq:general-penalty-energy}
\end{equation}
The zero-bulk case is defined by omitting the penalty. All $m$ pressure
equations are retained, including the mean. The principal
examples have exact reference volumes under both rules, so their pressure
weights coincide. More general fixed positive weights are admitted by the
abstract penalty theorem.

The background energy $E_h$ and its material quadrature remain identical
between the raw and exact-volume comparisons. The nonlinear FE response uses
the isochoric neo-Hookean density
\[
 W_{\rm iso}(F)=\frac\mu2\bigl((\det F)^{-2/3}F:F-3\bigr),\qquad \mu>0,
\]
where $\mu$ is the shear modulus and isochoric means invariant under positive
uniform dilation. Its volume--distortion separation follows the classical
Flory construction \citep{flory1961thermodynamic}. The quadratic
driving-energy examples are identified separately. Every finite-state
Hessian differentiates the complete energy \eqref{eq:general-penalty-energy},
including the pressure-geometric term derived later.

\section{Sharp tensor volume defects}
\label{sec:tensor-admission}

We construct an interior displacement that changes quadrature volume
while leaving every face fixed. Degree counting gives an exactness
threshold; the explicit triple below proves its converse and quantifies
the surviving cubic. These same directions will give the feasible-distance
bound and the additional assembled constraint.

Use the parent cube \(\widehat K=[-1,1]^3\), whose volume is eight.
For \(\boldsymbol p=(p_1,p_2,p_3)\), let
\(\mathbb Q_{\boldsymbol p}\) be the scalar polynomials of degree at most
\(p_i\) in coordinate \(\xi_i\). Write \(\mathbb Q_p\) for equal orders.
The rule \(\mathcal G_{\boldsymbol n}\) has \(n_i\) Gauss--Legendre points
in coordinate \(i\), with integral weights, so
\(\mathcal G_{\boldsymbol n}[1]=8\).
It integrates every tensor polynomial of coordinate degree at most
\(2n_i-1\) exactly. The isotropic rule is denoted \(\mathcal G_n\).

The vector cell-interior space is
\[
\mathcal B_{\boldsymbol p}
=\bigl[\mathbb Q_{\boldsymbol p}\cap H^1_0(\widehat K)\bigr]^3.
\]
Its fields vanish on all six faces. For \(\widehat u\in\mathcal B_{\boldsymbol p}\),
define the scalar cubic
\[
P_{\boldsymbol p,\boldsymbol n}(\widehat u)
=\mathcal G_{\boldsymbol n}[\det D_\xi\widehat u].
\]
Its exact-integral counterpart is zero by the fixed-boundary volume identity.
Throughout this section assume \(p_i\ge1\) and \(n_i\ge p_i+1\).
The latter condition integrates all terms with at most two displacement
gradients in \(\det(I_3+D_\xi\widehat u)\).
Consequently
\begin{equation}
\mathcal G_{\boldsymbol n}[\det(I_3+D_\xi\widehat u)]-8
=P_{\boldsymbol p,\boldsymbol n}(\widehat u).
\label{eq:bubble-pure-volume-cubic}
\end{equation}

\begin{theorem}[Supported cubic admission]
\label{thm:tensor-volume-admission}
Under the stated Gauss-order assumptions,
\(P_{\boldsymbol p,\boldsymbol n}\) is not identically zero on
\(\mathcal B_{\boldsymbol p}\) if and only if
\begin{equation}
p_i\ge3\quad(i=1,2,3),
\qquad
2n_i<3p_i\quad\hbox{for at least one coordinate }i.
\label{eq:tensor-bubble-admission}
\end{equation}
For an admitted coordinate, there are three component-separated
\(W_1,W_2,W_3\in\mathcal B_{\boldsymbol p}\) such that
\begin{equation}
P_{\boldsymbol p,\boldsymbol n}
   (a_1W_1+a_2W_2+a_3W_3)
=\chi_{n_i}a_1a_2a_3,
\label{eq:tensor-witness-polynomial}
\end{equation}
where
\begin{equation}
\chi_n=
6\left(\frac{32}{105}\right)^2
\left(\frac{2^n}{\binom{2n}{n}}\right)^2>0.
\label{eq:gauss-witness-coefficient}
\end{equation}
\end{theorem}

\begin{proof}
We first prove the vanishing cases. A scalar polynomial zero on both faces
\(\xi_i=\pm1\) is divisible by \(1-\xi_i^2\). Thus if \(p_i=1\), the
entire cell-interior space is zero.
If \(p_i=2\), every vector field has the form
\[
\widehat u(\xi)=A(\xi_i)v(\xi_j,\xi_k),
\qquad A(t)=1-t^2,
\]
where \(\{i,j,k\}=\{1,2,3\}\).
The determinant of its gradient has a factor \(A(\xi_i)^2A'(\xi_i)\);
the remaining factor is independent of \(\xi_i\).
That factor is odd, so the symmetric Gauss rule gives zero.

Each term of \(\det D_\xi\widehat u\) differentiates once in every
coordinate across its three factors. Its degree in coordinate \(i\) is
therefore at most \(3p_i-1\).
If \(2n_i\ge3p_i\) in all coordinates, the rule is exact for this
determinant, whose integral is zero. These observations prove necessity.

For sufficiency, a simultaneous permutation of coordinate axes and vector
components lets us take the admitted coordinate to be the first.
Put \(p=p_1\), \(n=n_1\), and define
\begin{equation}
\begin{gathered}
A(t)=1-t^2,\qquad B(t)=t(1-t^2),\qquad C=\frac{32}{105},\\
d=2n-2p+1,\qquad
f(x)=x^{d-2}(1-x^2),\qquad
g(x)=x^{p-2}(1-x^2).
\end{gathered}
\label{eq:tensor-witness-factors}
\end{equation}
The inequalities \(n\ge p+1\) and \(2n<3p\) give \(3\le d\le p\).
For \(\xi=(x,y,z)\), choose
\begin{equation}
\begin{aligned}
W_1&=(f(x)A(y)A(z),0,0),\\
W_2&=(0,g(x)B(y)A(z),0),\\
W_3&=(0,0,g(x)A(y)B(z)).
\end{aligned}
\label{eq:tensor-supported-directions}
\end{equation}
The axial degrees are admitted by \(d\le p\), and the transverse degrees
are at most three. Every factor vanishes at its relevant endpoints.

Let \(w=W_1+W_2+W_3\).
The two transverse moments needed for its determinant are
\[
\int_{-1}^1A^2B'=2C,\qquad
\int_{-1}^1AA'B=-C.
\]
Expanding the determinant and integrating in \(y,z\) gives
\begin{align*}
\int_{-1}^1\!\!\int_{-1}^1\det D_\xi w\,dy\,dz
&=\bigl(4C^2-C^2\bigr)f'g^2+6C^2fgg'\\
&=3C^2(fg^2)'.
\end{align*}
All transverse polynomials here have degree at most six. The transverse
orders are at least four, so their Gauss rules give the same expression.
Moreover,
\[
fg^2=x^{2n-5}(1-x^2)^3.
\]
Its derivative has degree \(2n\) and leading coefficient \(-(2n+1)\).

Let \(\pi_n\) denote the monic Legendre polynomial, and let
\(\mathcal G_n^{(1)}\) be the one-dimensional \(n\)-point Gauss rule.
The polynomial \(x^{2n}-\pi_n^2\) has degree at most \(2n-1\), while
\(\pi_n\) vanishes at each Gauss node. Hence
\[
\left(\mathcal G_n^{(1)}-\int_{-1}^1\right)x^{2n}
=-\int_{-1}^1\pi_n^2.
\]
Since \(fg^2\) vanishes at both endpoints,
\[
\mathcal G_{\boldsymbol n}[\det D_\xi w]
=3C^2(2n+1)\int_{-1}^1\pi_n^2.
\]
The monic Legendre normalization is
\[
\int_{-1}^1\pi_n^2
=\frac2{2n+1}
  \left(\frac{2^n}{\binom{2n}{n}}\right)^2.
\]
This proves \eqref{eq:gauss-witness-coefficient}.
Each \(W_i\) occupies one component row, so multiplication by the three
amplitudes multiplies the determinant by \(a_1a_2a_3\).
This gives \eqref{eq:tensor-witness-polynomial} and sufficiency.
Finally, the linear and quadratic displacement terms have coordinate
degree at most \(2p_i\). Their exactness, together with the fixed boundary
trace, proves \eqref{eq:bubble-pure-volume-cubic}.
\end{proof}

For isotropic \(p\ge3\) and \(n\ge p+1\), this gives the sharp supported
threshold \(2n<3p\). When \(2n\ge3p\), degree counting integrates the
determinant of every \(\mathbb Q_p^3\) map exactly.
In particular Gauss5 removes the Q3 determinant-volume defect.
For anisotropic spaces, the transverse orders matter: a fully clamped
\(\mathbb Q_{(3,2,2)}\) cell with Gauss orders \((4,3,3)\) has zero
supported cubic despite \(2n_1<3p_1\).
The theorem classifies the cell-interior space. A failure of this admission
condition gives no conclusion about other motions of an assembled mesh.

\subsection{Physical directions and their volume coefficient}

Let \(X_e(\xi)=b_e+A_e\xi\) be an orientation-preserving affine cell map.
Then \(V_e=8\det A_e\).
Push the reference directions forward as
\[
U_{e,i}(X_e(\xi))=A_eW_i(\xi),
\]
with zero extension outside the cell.
Determinant multiplicativity gives
\begin{equation}
q_{h,e}^Q\!\left(\sum_i a_iU_{e,i}\right)
=\gamma_ea_1a_2a_3,\qquad
q_{h,e}^I\!\left(\sum_i a_iU_{e,i}\right)=0,\qquad
\gamma_e=\frac{V_e}{8}\chi_n.
\label{eq:physical-bubble-coefficient}
\end{equation}
At any starting field, each direction supplies a rank-one parent-gradient
update. The cell-volume polynomial remains affine in each amplitude separately,
called \emph{tri-affine},
with the same \(\gamma_e\); its lower-order coefficients change with the
starting field.

\section{Constructive nonlinear feasibility}
\label{sec:tri-affine}

The tensor directions have a nonzero mixed cubic coefficient even when
the starting field is deformed. We now use that coefficient to construct
a zero of the cell equation, then combine independent cell corrections.
A nearby zero bounds distance to the feasible set. Since the directions
vanish on every face, the construction leaves exact cell volumes fixed.

A tri-affine polynomial is affine in each of its three amplitudes
separately. Its lower-order coefficients may depend on the starting
deformation. The following lemma controls the correction using only
the residual and the mixed coefficient.

Scalar trilinear polynomials are harmonic in their product coefficient
coordinates: their pure second derivatives vanish. Degree-based local
residual--distance estimates for harmonic polynomials are established
\citep[Theorem~3.1]{badger2017harmonic}. The following direct construction
supplies a coefficient-dependent bound and a correction using prescribed
directions. These features allow its application to individual cell equations.

\begin{lemma}[Four-corner correction]
\label{lem:tri-affine-four-corner}
Let \(a=(a_1,a_2,a_3)\in\mathbb R^3\), and let
\[
f(a)=r+\sum_{i=1}^3\ell_i a_i
       +\sum_{1\le i<j\le3}b_{ij}a_i a_j+\gamma a_1a_2a_3,
\qquad \gamma\ne0,
\]
where \(r,\ell_i,b_{ij},\gamma\) are real coefficients.
There is a zero \(a_*\) of \(f\) satisfying
\begin{equation}
\|a_*\|_\infty\le\left|\frac r\gamma\right|^{1/3},
\qquad
\|a\|_\infty=\max_i|a_i|.
\label{eq:four-corner-amplitude}
\end{equation}
For \(r\ne0\), the zero lies on a segment joining the origin to one of four
corners of the cube with this radius.
\end{lemma}

\begin{proof}
If \(r=0\), take \(a_*=0\).
Otherwise put \(\eta=|r/\gamma|^{1/3}\) and
\(\epsilon=-\operatorname{sign}(r\gamma)\). Consider the four sign vectors
\[
\mathcal S_\epsilon=
\{\sigma\in\{-1,1\}^3:\sigma_1\sigma_2\sigma_3=\epsilon\}.
\]
Their first and pairwise signed moments vanish. Their common triple product
is \(\epsilon\), so
\[
\frac14\sum_{\sigma\in\mathcal S_\epsilon}f(\eta\sigma)
=r+\epsilon\gamma\eta^3=0.
\]
At least one corner has the opposite weak sign to \(f(0)=r\).
Continuity on the segment from zero to that corner gives a root.
Every point of the segment satisfies \eqref{eq:four-corner-amplitude}.
\end{proof}

Only the residual \(r\) and the mixed coefficient \(\gamma\) enter the bound.
The linear and quadratic coefficients determine which corner supplies the
segment and where the root lies. Thus the construction remains available
at a deformed starting field, where those lower-order coefficients generally
are nonzero. In computation it requires four polynomial evaluations and a
scalar root on a known interval.

For completeness, the same argument gives a global bound for a scalar
trilinear form \(P:X_1\times X_2\times X_3\to\mathbb R\), where the
\(X_i\) are finite-dimensional real Hilbert spaces. Equip their product with
the sum-of-squares norm. Choose unit vectors \(e_i\in X_i\) with
\(\gamma=P(e_1,e_2,e_3)\ne0\).
Applying the lemma to \(P(x_1+a_1e_1,x_2+a_2e_2,x_3+a_3e_3)\) yields
\begin{equation}
\operatorname{dist}\bigl(x,P^{-1}(0)\bigr)
\le\sqrt3\,|\gamma|^{-1/3}|P(x)|^{1/3}.
\label{eq:scalar-trilinear-error-bound}
\end{equation}
The same estimate holds for a block-tri-affine scalar polynomial with this
top trilinear part. A known nonzero coefficient therefore suffices;
the construction requires no optimization of the tensor norm.

\subsection{Why a system needs additional structure}

Individual scalar bounds do not ensure that one correction satisfies several
coupled equations. An explicit example also shows why eliminating regular
equations can increase the relevant order. Split six real variables into
the blocks \((a,x),(b,y),(c,z)\), and define
\begin{equation}
\mathcal R(a,x,b,y,c,z)=
\begin{pmatrix}
xbc-ayz\\
ayc-abz\\
xyz
\end{pmatrix}.
\label{eq:coupled-trilinear-example}
\end{equation}
Each component is trilinear in the three blocks.
At the feasible point \(w_0=(1,0,1,0,1,0)\), use the path
\[
w(t)=(1,t^2,1,t,1,t).
\]
Its residual is \(\mathcal R(w(t))=(0,0,t^4)^T\).
Near \(w_0\), the coefficients \(a,b,c\) remain nonzero. The first two
equations give
\[
y=\frac{bz}{c},\qquad x=\frac{az^2}{c^2}.
\]
The remaining equation becomes \(abz^4/c^3=0\).
Hence the local zero set has \(x=y=z=0\), with \(a,b,c\) free.
The distance from \(w(t)\) to that set is
\(\sqrt{2t^2+t^4}\). For small \(t\), a closest point of the full zero set
also lies in this neighborhood, since \(w_0\) already gives an
order-\(|t|\) competitor. Thus the residual is of order \(t^4\) while
the distance is of order \(|t|\); a cube-root upper estimate fails.

The cellwise argument below supplies the additional structure missing from
this example. Each scalar volume equation receives directions supported
inside its own cell, and those directions leave all other equations unchanged.

\subsection{Cellwise correction in the full displacement space}
\label{sec:cellwise-feasibility}

Independent cell support lets the scalar corrections solve all volume
equations simultaneously. The starting field may move interior faces;
only the corrections have zero traces there.
We first state the general estimate. The tensor directions from
\cref{sec:tensor-admission} satisfy its quantitative hypotheses, as
verified in \cref{sec:tensor-scaling}.

Let \(\mathcal T_h=\{K_e\}_{e=1}^{m_h}\) partition a fixed bounded Lipschitz
domain \(\Omega\), let \(V_e=|K_e|>0\), and put
\(\mathsf V=\operatorname{diag}(V_1,\ldots,V_{m_h})\).
Use the physical norms and volume-weighted residual norm of
\cref{sec:discrete-volume-setting}, with
\(|v|_{H^1(K)}^2=\int_K\|\nabla_Xv\|_F^2\).
The full conforming space is \(\mathcal U_h\subset H^1_0(\Omega;\mathbb R^3)\),
with a fixed Poincar\'e constant \(C_P\).
Maximum-gradient bounds use the Frobenius norm.

\begin{theorem}[Quantitative cellwise feasibility]
\label{thm:cellwise-feasibility}
Suppose each cell has three directions \(U_{e,i}\in\mathcal U_h\),
\(i=1,2,3\), supported in \(\overline K_e\) and zero on \(\partial K_e\).
For every starting field \(u\in\mathcal U_h\), assume that
\[
a\longmapsto q_{h,e}^Q\!\left(u+\sum_{i=1}^3a_iU_{e,i}\right)
\]
is tri-affine with a coefficient \(\gamma_e\) of \(a_1a_2a_3\) independent
of \(u\). These changes leave all other residual rows unchanged.
Assume common positive constants \(\gamma_0,L_2,L_\infty\) satisfy
\begin{align}
|\gamma_e|&\ge\gamma_0V_e,\label{eq:cellwise-coefficient-admission}\\
\left|\sum_i a_iU_{e,i}\right|_{H^1(K_e)}
&\le L_2\sqrt{V_e}\,\|a\|_\infty,\qquad
\left\|\nabla_X\sum_i a_iU_{e,i}\right\|_{L^\infty(K_e)}
\le L_\infty\|a\|_\infty .
\label{eq:cellwise-direction-bounds}
\end{align}
Define the algebraic feasible set
\(\mathcal F_h^Q=\{v\in\mathcal U_h:q_h^Q(v)=0\}\).
Every \(u\in\mathcal U_h\) has a correction \(\delta u\in\mathcal U_h\), zero on every
cell face, with \(u+\delta u\in\mathcal F_h^Q\) and
\begin{equation}
\operatorname{dist}_{H^1}(u,\mathcal F_h^Q)
\le\|\delta u\|_{H^1}
\le \sqrt{1+C_P^2}\,L_2\gamma_0^{-1/3}
|\Omega|^{1/3}\|q_h^Q(u)\|_{\mathsf V^{-1}}^{1/3}.
\label{eq:full-space-holder-bound}
\end{equation}
If the hypotheses hold on a mesh family with the displayed constants
independent of \(h\), the estimate is uniform in \(h\).
\end{theorem}

\begin{proof}
For each cell put \(r_e=q_{h,e}^Q(u)\).
\Cref{lem:tri-affine-four-corner} gives amplitudes \(a_e\) for which the
cell residual vanishes and
\[
\|a_e\|_\infty
\le\left|\frac{r_e}{\gamma_e}\right|^{1/3}
\le\gamma_0^{-1/3}|r_e/V_e|^{1/3}.
\]
Use \(a_e=0\) when \(r_e=0\), and set
\(\delta u_e=\sum_i(a_e)_iU_{e,i}\).
Since each \(\delta u_e\) leaves all other rows unchanged,
\(\delta u=\sum_e\delta u_e\) makes every residual zero.
Its traces vanish on all cell faces.

The cell interiors are disjoint. Applying
\eqref{eq:cellwise-direction-bounds} and H\"older's inequality gives
\begin{align*}
|\delta u|_{H^1(\Omega)}^2
&=\sum_e|\delta u_e|_{H^1(K_e)}^2\\
&\le L_2^2\gamma_0^{-2/3}
       \sum_e V_e^{2/3}(r_e^2/V_e)^{1/3}\\
&\le L_2^2\gamma_0^{-2/3}
       \left(\sum_eV_e\right)^{2/3}
       \left(\sum_e r_e^2/V_e\right)^{1/3}.
\end{align*}
The Poincar\'e inequality supplies the \(L^2\) part of the full norm.
The constructed feasible point supplies the distance bound.
\end{proof}

The normalized coefficient bound \eqref{eq:cellwise-coefficient-admission}
is consequential. Nonzero coefficients on each individual mesh do not
alone control a refinement family. \Cref{sec:tensor-scaling} verifies the constants for the explicit
tensor directions.

\subsection{Admissible input and output neighborhoods}

The algebraic theorem has no small-displacement restriction. To obtain a
physically admissible corrected deformation, assume additionally that
\(\mathcal U_h\subset W^{1,\infty}(\Omega;\mathbb R^3)\) and that, for
\(0<\delta\le\delta_0\),
\begin{equation}
\|\nabla_Xu\|_{L^\infty}\le\delta
\quad\Longrightarrow\quad
\max_e\frac{|q_{h,e}^Q(u)|}{V_e}\le C_r\delta
\label{eq:cellwise-residual-growth}
\end{equation}
with constants \(C_r,\delta_0>0\) independent of \(h\).
Fix an output radius \(0<\rho_{\rm out}<1\), and choose a positive input
radius satisfying
\begin{equation}
\rho_{\rm in}\le
\min\left\{\delta_0,\frac{\rho_{\rm out}}2,\,
\frac{\gamma_0}{C_r}
\left(\frac{\rho_{\rm out}}{2L_\infty}\right)^3\right\}.
\label{eq:cellwise-admissibility-radii}
\end{equation}
For every input with maximum gradient at most \(\rho_{\rm in}\), the
correction in the theorem obeys
\[
\|\nabla_X\delta u\|_{L^\infty}
\le L_\infty(C_r\rho_{\rm in}/\gamma_0)^{1/3}
\le\rho_{\rm out}/2.
\]
Together with \(\rho_{\rm in}\le\rho_{\rm out}/2\), this places the corrected
field in the output neighborhood
\[
\mathcal F_h^Q(\rho_{\rm out})
=\{v\in\mathcal F_h^Q:
       \|\nabla_Xv\|_{L^\infty}\le\rho_{\rm out}\}.
\]
Consequently \eqref{eq:full-space-holder-bound} also holds with this
admissible set in the distance.

To justify injectivity, extend \(u+\delta u\) by zero outside \(\Omega\).
The complete outer clamp makes this extension a
\(W^{1,\infty}(\mathbb R^3;\mathbb R^3)\) field with Lipschitz constant at
most \(\rho_{\rm out}\). The map identity plus the extension has lower
Lipschitz constant \(1-\rho_{\rm out}>0\).
For each target point, the equation for its inverse is a contraction, so
the map is globally bijective. The matrices \(I+t\nabla_X(u+\delta u)\),
\(0\le t\le1\), are nonsingular; their determinants retain the positive
sign at \(t=0\). This proves orientation preservation as well.
The construction thus maps a smaller input neighborhood into a fixed
admissible output neighborhood. Smallness only in \(H^1\) would not supply
this maximum-gradient argument.

\subsection{Meaning of the corrected volumes}

Every cell-face trace is preserved by the correction. Exact determinant
volume depends only on that trace, so
\begin{equation}
q_h^I(u+\delta u)=q_h^I(u).
\label{eq:correction-preserves-exact-volume}
\end{equation}
The constructed field can therefore satisfy all raw volume equations while
retaining the starting field's nonzero exact cell-volume changes. This is a
feasibility property of the quadrature-defined equations. Its interpretation
as physical incompressibility would require a different constraint comparison.

The norm of this particular correction is an upper bound on feasible distance.
A closer raw-feasible field can change interior-face traces.
The fixed-trace construction and a closest-point projection consequently
answer different questions.

The theorem also applies when only some cells need correction, provided all
other rows are already zero and remain unchanged. An exactly integrated
cell has no boundary-zero direction that changes its volume. Arbitrary
residuals on a mixture of exact and underintegrated cells therefore require
an additional correction argument. No uniform bound for that extra step is
assumed here.

\subsection{Uniform physical bounds for the tensor construction}
\label{sec:tensor-scaling}

To verify the uniform norm hypotheses, set
\[
K_{\rm sh}=\sup_e\|A_e\|_2\|A_e^{-1}\|_2,\qquad
B_2^2=\sum_i\|D_\xi W_i\|_{L^2(\widehat K;F)}^2,\qquad
B_\infty^2=\sum_i\|D_\xi W_i\|_{L^\infty(\widehat K;F)}^2.
\]
The last two norms use the Frobenius matrix norm.
The reference gradients occupy disjoint component rows, so pointwise
\[
\left\|\sum_i a_iD_\xi W_i\right\|_F^2
=\sum_i a_i^2\|D_\xi W_i\|_F^2.
\]
The physical gradient is
\[
\nabla_X\sum_i a_iU_{e,i}
=A_e\left(\sum_i a_iD_\xi W_i\right)A_e^{-1}.
\]
Applying the matrix norm bound after the row identity and using
\(\det A_e=V_e/8\) gives the constants
\begin{equation}
\gamma_0=\frac{\chi_n}{8},\qquad
L_2=\frac{K_{\rm sh}B_2}{\sqrt8},\qquad
L_\infty=K_{\rm sh}B_\infty
\label{eq:tensor-cellwise-constants}
\end{equation}
in \Cref{thm:cellwise-feasibility}.
Thus, if every cell has the stated triple, any uniformly shape-regular
affine family of fixed admitted orders has the full-space bound
\eqref{eq:full-space-holder-bound} with an \(h\)-independent constant.
Its dependence on \(\chi_n\) and the bubble norms remains explicit;
no uniformity in polynomial order is asserted.

For the admissibility hypothesis, positive weights and
\(\|\nabla_Xu\|_\infty\le\delta<1\) give
\[
\frac{|q_{h,e}^Q(u)|}{V_e}
\le(1+\delta)^3-1\le7\delta.
\]
This verifies \eqref{eq:cellwise-residual-growth} with \(C_r=7\) on a
fixed small-strain interval, completing the physical application of the
general theorem.

The parent-volume normalization is important when comparing the Q3 formulas
below with unit-cube implementations.
Under \(\xi=2\eta-\mathbf1\), the unscaled functions
\(A(2\eta-1)=4\eta(1-\eta)\) and
\(B(2\eta-1)=4\eta(1-\eta)(2\eta-1)\) give the raw coefficient \(\chi_n\)
on \([0,1]^3\): three gradient factors of two cancel its volume factor.
The physical push-forward to that same cube has \(A_e=I_3/2\), halves
each displacement direction, and gives \(\gamma_e=\chi_n/8\).
Equation \eqref{eq:physical-bubble-coefficient} uses this latter convention.

\section{The assembled constraint geometry}
\label{sec:assembled-geometry}

The supported construction also identifies what changes in a mesh with free
interior faces. At an affine reference, the raw and exact maps agree through
their second derivatives. A remaining cubic term changes the number of
independent nonlinear constraints at nearby feasible states.
We prove that change, identify the cubic relation that survives elimination
of the independent linearized equations, and use it to establish sharpness
of the preceding distance estimate. The surviving relation is also the
input to the mechanical-response theorem in \cref{sec:singular-penalty}.

Consider a face-connected conforming mesh of \(m\) affine hexahedra.
Let \(\mathcal U_h\) be the full continuous \(\mathbb Q_p^3\) displacement
space, \(p\ge3\), with the complete outer clamp, and put
\(N=\dim\mathcal U_h\).
Cell rules have \(n_e\ge p+1\) Gauss points per coordinate.
Assume compatible face parameterizations, so the usual biquadratic
face-bubble traces and their two-cell extensions belong to \(\mathcal U_h\).
At least one cell \(e_*\) has \(2n_{e_*}<3p\).
Let \(U_1,U_2,U_3\) be its supported physical directions from
\eqref{eq:physical-bubble-coefficient}, and put
\[
\gamma_*=\frac{V_{e_*}}8\chi_{n_{e_*}},
\qquad
u_\tau=\tau(U_2+U_3).
\]
All \(m\) cell-volume equations are retained.

\begin{theorem}[Assembled rank change at feasible states]
\label{thm:assembled-volume-geometry}
At zero displacement, the raw and exact volume maps have equal values,
Jacobians and Hessians. Their common Jacobian
\[
B_h=Dq_h^Q(0)=Dq_h^I(0)\in\mathbb R^{m\times N}
\]
has rank \(m-1\).
The exact map has constant rank \(m-1\) in a neighborhood of zero.

For sufficiently small nonzero \(\tau\), the states \(u_\tau\) are admissible
and satisfy \(q_h^Q(u_\tau)=q_h^I(u_\tau)=0\), while
\[
\operatorname{rank}Dq_h^Q(u_\tau)=m.
\]
In fixed nodal and cell-coordinate Euclidean norms, the smallest of its
\(m\) singular values satisfies
\begin{equation}
\sigma_m\!\left(Dq_h^Q(u_\tau)\right)=\Theta(\tau^2).
\label{eq:feasible-quadratic-singular-value}
\end{equation}
The constants in this estimate concern the fixed mesh.
\end{theorem}

\begin{proof}
For an affine cell, all determinant terms with at most two displacement
gradients have coordinate degree at most \(2p\). The rule integrates them
exactly. Thus each raw-minus-exact cell map is homogeneous cubic in \(u\),
giving agreement of the first two derivatives at zero.
The exact total-volume identity gives
\(\mathbf1_m^TB_h=0\), so the rank is at most \(m-1\).

To prove the reverse inequality, take an interior face shared by two cells.
Its scalar biquadratic bubble is the product of the two tangential factors
\(1-\xi_j^2\). Multiply this trace by a constant physical face normal,
and extend it into both neighboring cells with a linear factor that vanishes
on the opposite face. The other traces vanish because of the tangential
bubble factors. These extensions are conforming members of
\(\mathcal U_h\).
The first variations of the two exact cell volumes are opposite nonzero
normal fluxes, and all other cell actions vanish.
The resulting columns span a scaled incidence matrix of the connected
cell-adjacency graph. Its rank is \(m-1\).
This proves the rank of \(B_h\). A nonzero \((m-1)\)-minor persists near
zero, while the exact total-volume identity preserves the upper bound.
Hence the exact rank stays \(m-1\).

The three-amplitude identity \eqref{eq:physical-bubble-coefficient} gives
the feasibility of \(u_\tau\) and the exact derivative
\begin{equation}
Dq_h^Q(u_\tau)[U_1]
=\gamma_*\tau^2\mathbf e_{e_*},
\label{eq:missing-component-volume-action}
\end{equation}
where \(\mathbf e_{e_*}\) is the coordinate vector of the selected cell.
Small amplitudes give a small physical displacement gradient, proving
admissibility.
Choose fixed directions \(v_1,\ldots,v_{m-1}\) whose columns under \(B_h\)
are independent, and set
\[
T=[v_1,\ldots,v_{m-1},U_1],\qquad
M_\tau=Dq_h^Q(u_\tau)T.
\]
The first \(m-1\) columns of \(M_\tau\) converge to a basis of
\(\mathbf1_m^\perp\). By \eqref{eq:missing-component-volume-action}, its
last column divided by \(\tau^2\) is
\(\gamma_*\mathbf e_{e_*}\), whose component sum is nonzero.
It follows that
\[
|\det M_\tau|\ge c|\tau|^2
\]
for small nonzero \(\tau\), with \(c>0\). Therefore the raw map has rank \(m\).

For the singular-value estimate, define the aggregate raw volume change
\begin{equation}
C_h(u)=\mathbf1_m^Tq_h^Q(u).
\label{eq:aggregate-cubic-constraint}
\end{equation}
It is homogeneous cubic: the exact aggregate is zero and each cubature
difference is cubic. With \(v_0=U_2+U_3\),
\[
\left\|Dq_h^Q(u_\tau)^T\frac{\mathbf1_m}{\sqrt m}\right\|_2
=\frac{|\tau|^2}{\sqrt m}\|DC_h(v_0)\|_2.
\]
The variational characterization of the smallest row singular value
therefore gives the required \(O(\tau^2)\) upper bound.

The matrices \(M_\tau\) have bounded norm near zero. Their determinant
bound and the product formula for singular values give
\[
\sigma_m(M_\tau)
\ge\frac{|\det M_\tau|}{\|M_\tau\|_2^{m-1}}
\ge c'|\tau|^2.
\]
For every unit row vector \(a\),
\(\|T^TDq_h^Q(u_\tau)^Ta\|_2
\le\|T\|_2\|Dq_h^Q(u_\tau)^Ta\|_2\).
Taking the minimum yields
\(\sigma_m(M_\tau)\le\|T\|_2\sigma_m(Dq_h^Q(u_\tau))\), proving the
matching lower bound. When \(m=1\), the flux block is empty and the
same proof uses the single nonzero column in
\eqref{eq:missing-component-volume-action}.
\end{proof}

Around the nonzero feasible states in this theorem, the raw map consequently
has one more independent equation than the exact map. At the reference
itself, their values, pressure forces and pressure tangents coincide.
These lower-order equalities therefore do not establish equality of the
nearby nonlinear constraint geometry.

\subsection{The first reduced compatibility term}

The rank change identifies the cubic relation that enters the small-load
penalty limit. To state this relation precisely, let
\(K_h=\ker B_h\), let \(P_{K_h}\) be its Euclidean orthogonal projector,
and let \(\Pi\) project onto
\(\operatorname{range}B_h=\mathbf1_m^\perp\).
The local coordinate map
\[
u\longmapsto\bigl(P_{K_h}u,\Pi q_h^Q(u)\bigr)
\]
has invertible derivative at zero:
\(B_h\) maps \(K_h^\perp\) isomorphically onto its range.
The inverse-function theorem therefore provides a smooth local inverse
\(\Phi(v,z)\), with \(v\in K_h\) and
\(z\in\operatorname{range}B_h\)
\citep{dontchev2014implicit}.
In particular,
\[
\Phi(v,0)=v+O(\|v\|_2^2),\qquad
\Pi q_h^Q(\Phi(v,0))=0.
\]
The remaining scalar equation is
\(\rho_h(v)=\mathbf1_m^Tq_h^Q(\Phi(v,0))=0\).
Since \(C_h\) is homogeneous cubic,
\begin{equation}
\rho_h(v)=P_{3,h}(v)+O(\|v\|_2^4),
\qquad
P_{3,h}=C_h|_{K_h}.
\label{eq:reduced-cubic-compatibility}
\end{equation}
This cubic is nonzero. Indeed,
\(v=U_1+U_2+U_3\) has \(B_hv=0\) by its zero face trace, and
\(C_h(v)=\gamma_*\ne0\).
Thus the first reduced compatibility order is exactly three.
For the exact-volume map, the remaining scalar relation is identically
zero after the same regular-coordinate construction.

The distinction from \eqref{eq:coupled-trilinear-example} is now explicit.
Here the total-volume identity and lower-order exactness produce a
homogeneous aggregate cubic that is already nonzero on the linearized kernel.
The regular-coordinate correction cannot remove that leading term.
The subsequent penalty theorem uses this identified reduced term.

\begin{corollary}[Sharp fixed-mesh cube-root exponent]
\label{cor:assembled-cubic-sharpness}
On each fixed mesh of \Cref{thm:assembled-volume-geometry}, let
\(v=U_1+U_2+U_3\). As \(t\downarrow0\),
\[
\|q_h^Q(tv)\|_{\mathsf V^{-1}}=\Theta(t^3),\qquad
\operatorname{dist}_{H^1}(tv,\mathcal F_h^Q)=\Theta(t),\qquad
q_h^I(tv)=0.
\]
Consequently no local residual--distance upper bound with exponent greater
than \(1/3\) holds at zero. When all cells satisfy the uniform admission
hypotheses of \Cref{thm:cellwise-feasibility}, its cube-root upper exponent
is therefore optimal.
\end{corollary}

\begin{proof}
The supported identity gives
\(q_h^Q(tv)=\gamma_*t^3\mathbf e_{e_*}\), and exact volumes are unchanged.
Every raw-feasible point belongs to the closed zero cone
\(\mathcal Z_h=\{u\in\mathcal U_h:C_h(u)=0\}\).
Since \(C_h(v)\ne0\), its distance from \(v\) in the fixed-mesh \(H^1\)
norm is positive. Homogeneity gives
\[
\operatorname{dist}_{H^1}(tv,\mathcal F_h^Q)
\ge\operatorname{dist}_{H^1}(tv,\mathcal Z_h)
=t\,\operatorname{dist}_{H^1}(v,\mathcal Z_h).
\]
Zero is feasible and supplies the matching upper bound \(t\|v\|_{H^1}\).
These bounds also hold when feasibility is restricted to any fixed
admissible neighborhood of zero. Comparing orders excludes every larger
upper-bound exponent.
\end{proof}

Only one admitted cell is required for the rank change, reduced cubic and
fixed-mesh sharpness. Correcting arbitrary residuals throughout the mesh
requires the cellwise hypotheses for every affected row.
The distance lower constant in the corollary may depend on the mesh; a
simultaneous mesh/load lower estimate is not asserted.
Projecting out a mean-pressure equation changes the constraint set and
requires its own comparison, as the Tet10 example will illustrate.

\section{Singular penalty limits and mechanical response}
\label{sec:singular-penalty}
\label{sec:finite-penalty-limit}

The additional volume relation also affects a finite-penalty problem when
the bulk coefficient increases as the load decreases. The relevant order
is the first nonzero term remaining after the regular constraint equations
have been solved. We first derive the resulting minimum-energy limit,
using regular coordinates and matching lower and upper estimates
\citep{bonnans1998perturbations,dontchev2014implicit}. The affine
finite-element application then supplies a positive mechanical Hessian and
a cubic reduced relation, giving a sextic limiting energy. The quadratic
relation of the Tet10 template gives the later quartic specialization
in \cref{sec:tet10-quartic-specialization}.

Mechanical-rigidity theory also relates nonlinear constraint compatibility
to quartic energy costs along motions with vanishing quadratic cost
\citep[Appendix~A]{damavandi2022energetic}. Here the joint load--bulk
scaling yields a weighted reduced energy and a limit theorem for
normalized minimum values and localized minimizers.

\subsection{Regular coordinates and reduced compatibility}

Let $N,m$ be positive integers and let $\mathcal U\subset\R^N$ be an open
neighborhood of zero. Let $q:\mathcal U\to\R^m$ be a $C^k$ constraint
map, where the integer $k\ge2$, and assume $q(0)=0$. Define
\[
 B_q=Dq(0),\qquad K=\ker B_q.
\]
Write $s=m-\rank B_q$ and suppose $s\ge1$. Let
$S\in\R^{m\times s}$ have full column rank and satisfy
$\operatorname{range}S=\ker B_q^T$. Its columns describe all resting
row relations; they need not be orthonormal. Let
$\mathsf V\in\R^{m\times m}$ be a fixed symmetric positive-definite
residual weight, and put
\[
 W_S=S^T\mathsf V S\succ0.
\]
For a positive-definite matrix $A$ and vectors of the appropriate
dimension, we use $\langle x,y\rangle_A=x^TAy$ and
$\|x\|_A=(x^TAx)^{1/2}$.

Let $P_K$ and $\Pi$ be the Euclidean orthogonal projectors onto $K$ and
$\operatorname{range}B_q$, respectively. The derivative at zero of
\[
 u\longmapsto\bigl(P_Ku,\Pi q(u)\bigr)
\]
is invertible as a map from $\R^N$ to
$K\times\operatorname{range}B_q$: the restriction of $B_q$ to
$K^\perp$ is an isomorphism onto its range. The inverse-function theorem
therefore gives a $C^k$ local inverse
\begin{equation}
 \Phi(v,z)=v+\psi(v,z),\qquad
 v\in K,\quad z\in\operatorname{range}B_q,\quad
 \psi(v,z)\in K^\perp.
 \label{eq:penalty-regular-coordinates}
\end{equation}
Here $P_K\Phi(v,z)=v$, $\Pi q(\Phi(v,z))=z$,
$\psi(0,0)=0$, and $D_v\psi(0,0)=0$. In a sufficiently small
neighborhood,
\begin{equation}
 \psi(v,z)=O(\|v\|_2^2+\|z\|_2).
 \label{eq:penalty-normal-coordinate-size}
\end{equation}
If $\rank B_q=0$, the regular coordinate $z$ is absent and
$\Phi(v,0)=v$.

The remaining relation is the map
\[
 \rho(v)=S^Tq(\Phi(v,0)),\qquad v\in K.
\]
Assume that its first nonzero homogeneous term has degree $k$:
\begin{equation}
 \rho(v)=P_k(v)+o(\|v\|_2^k),
 \qquad P_k:K\to\R^s,
 \label{eq:reduced-compatibility-polynomial}
\end{equation}
where $P_k$ is a nonzero homogeneous polynomial map. Thus $k$ is the
reduced compatibility order. Solving the regular equations can change
this order, as the coupled example in \cref{sec:tri-affine} illustrates;
the degree of the original constraint polynomial alone does not determine it.

Let the background energy $E:\mathcal U\to\R$ be $C^2$ and satisfy
\[
 E(0)=0,\qquad DE(0)=0,\qquad M=D^2E(0)\succ0.
\]
The matrix $M$ is the Hessian of the actual chosen energy. Fix a load
covector $\ell\in\R^N$ and a coefficient $\beta>0$. For load amplitude
$\varepsilon>0$, define
\begin{equation}
 J_{\varepsilon,\beta}(u)
 =E(u)-\varepsilon\ell^Tu
  +\frac{\beta}{2\varepsilon^{2k-2}}
       q(u)^T\mathsf V^{-1}q(u).
 \label{eq:singular-penalty-energy}
\end{equation}
Choose a sufficiently small fixed closed ball $\mathcal N$ centered at
zero, contained in the coordinate neighborhood, such that
$E(u)\ge c_0\|u\|_2^2$ on $\mathcal N$ for a constant $c_0>0$.
Positive definiteness of $M$ guarantees such a ball. In the finite-element
application it is also contained in the admissible deformation neighborhood.

\begin{theorem}[Penalty limit determined by reduced compatibility]
\label{thm:finite-penalty-limit}
For each $\varepsilon>0$, let $u_{\varepsilon,\beta}$ minimize
\eqref{eq:singular-penalty-energy} over $\mathcal N$. Define on $K$
\begin{equation}
 J_\beta^0(v)
 =\frac12v^TMv-\ell^Tv
  +\frac\beta2 P_k(v)^TW_S^{-1}P_k(v).
 \label{eq:general-penalty-limit}
\end{equation}
Then
\[
 u_{\varepsilon,\beta}=O(\varepsilon),\qquad
 q(u_{\varepsilon,\beta})=O(\varepsilon^k),
\]
and
\[
 \frac{\min_{u\in\mathcal N}J_{\varepsilon,\beta}(u)}
      {\varepsilon^2}
 \longrightarrow\min_{v\in K}J_\beta^0(v).
\]
Every cluster point of $u_{\varepsilon,\beta}/\varepsilon$ minimizes
$J_\beta^0$. If its minimizer $v_\beta$ is unique, then
$u_{\varepsilon,\beta}/\varepsilon\to v_\beta$. The minimizing points
are interior to $\mathcal N$ for all sufficiently small $\varepsilon$.
\end{theorem}

\begin{proof}
Continuity and compactness give a minimum. Comparison with zero yields
\[
 c_0\|u_{\varepsilon,\beta}\|_2^2
 \le\varepsilon\|\ell\|_2\|u_{\varepsilon,\beta}\|_2.
\]
Consequently $u_{\varepsilon,\beta}=O(\varepsilon)$. The same comparison
bounds the penalty energy by $O(\varepsilon^2)$. Since $\beta$ and
$\mathsf V$ are fixed and positive, this gives
$q(u_{\varepsilon,\beta})=O(\varepsilon^k)$. The displacement bound
also makes the minima interior for small $\varepsilon$.

Write
\[
 v_\varepsilon=P_Ku_{\varepsilon,\beta},\qquad
 z_\varepsilon=\Pi q(u_{\varepsilon,\beta}).
\]
Their orders are $O(\varepsilon)$ and $O(\varepsilon^k)$, respectively.
We first control the effect of the regular coordinate on the reduced
relation. Since $S^TDq(0)=0$ and $q$ is $C^2$, the derivative
$S^TDq(\Phi(v,z))$ has norm $O(\|v\|_2+\|z\|_2)$ near zero.
The derivative $D_z\Phi$ is bounded there. Integrating along the
segment from $0$ to $z$ gives
\begin{equation}
 S^Tq(\Phi(v,z))-\rho(v)
 =O\bigl((\|v\|_2+\|z\|_2)\|z\|_2\bigr).
 \label{eq:penalty-reduced-coordinate-error}
\end{equation}
Take a subsequence for which $v_\varepsilon/\varepsilon\to v\in K$.
Equation \eqref{eq:penalty-normal-coordinate-size} shows that
$u_{\varepsilon,\beta}/\varepsilon$ has the same limit. Homogeneity,
\eqref{eq:reduced-compatibility-polynomial}, and
\eqref{eq:penalty-reduced-coordinate-error} imply
\begin{equation}
 \frac{S^Tq(u_{\varepsilon,\beta})}{\varepsilon^k}
 \longrightarrow P_k(v).
 \label{eq:penalty-reduced-residual-limit}
\end{equation}

For every residual vector $r\in\R^m$,
\begin{equation}
 r^T\mathsf V^{-1}r
 \ge(S^Tr)^TW_S^{-1}(S^Tr).
 \label{eq:weighted-residual-lower-bound}
\end{equation}
Indeed, for a prescribed $a\in\R^s$ the vector
$r_* =\mathsf VSW_S^{-1}a$ satisfies $S^Tr_*=a$. If also
$S^Tr=a$, then $r-r_*$ is orthogonal to $r_*$ in the
$\mathsf V^{-1}$ inner product, because
\[
 (r-r_*)^T\mathsf V^{-1}r_*
 =(S^T(r-r_*))^TW_S^{-1}a=0.
\]
Thus $r_*$ is the least-norm residual with that projected value, and its
squared norm is $a^TW_S^{-1}a$. Applying
\eqref{eq:weighted-residual-lower-bound} with
\eqref{eq:penalty-reduced-residual-limit}, and expanding $E$ to second
order, proves
\[
 \liminf_{\varepsilon\to0}
 \frac{J_{\varepsilon,\beta}(u_{\varepsilon,\beta})}{\varepsilon^2}
 \ge J_\beta^0(v)
\]
along the chosen subsequence.

For the matching upper bound, fix any $v\in K$ and set
\[
 r_*(v)=\mathsf VSW_S^{-1}P_k(v),\qquad
 u_\varepsilon^{\rm rec}
 =\Phi\bigl(\varepsilon v,\varepsilon^k\Pi r_*(v)\bigr).
\]
This recovery sequence lies in $\mathcal N$ for small $\varepsilon$,
and $u_\varepsilon^{\rm rec}/\varepsilon\to v$.
Its residual satisfies
\[
 \frac{\Pi q(u_\varepsilon^{\rm rec})}{\varepsilon^k}
 =\Pi r_*(v),\qquad
 \frac{S^Tq(u_\varepsilon^{\rm rec})}{\varepsilon^k}
 \longrightarrow P_k(v)=S^Tr_*(v).
\]
These two projections determine the residual. More explicitly,
\[
 r=\Pi r+S(S^TS)^{-1}S^Tr\qquad(r\in\R^m),
\]
because $\operatorname{range}S$ is the Euclidean orthogonal complement
of $\operatorname{range}B_q$. Hence
$q(u_\varepsilon^{\rm rec})/\varepsilon^k\to r_*(v)$, and
\[
 \frac{J_{\varepsilon,\beta}(u_\varepsilon^{\rm rec})}{\varepsilon^2}
 \longrightarrow J_\beta^0(v).
\]

The positive quadratic part and nonnegative penalty make $J_\beta^0$
coercive on the finite-dimensional space $K$, so it has a minimum and a
compact minimizing set. Apply the upper bound at one limiting minimum,
and the lower bound to convergent subsequences of the bounded normalized
minimizers. This proves convergence of minimum values and the cluster-point
assertion. If the limiting minimum is unique, every subsequence has the
same limit, which proves convergence of the whole normalized family.
\end{proof}

The theorem concerns minima over the specified neighborhood. Interiority
makes them local minima of the unrestricted energy for small load.
Other stationary points, including metastable local minima selected by a
numerical method, require additional analysis. The proof also applies to an
expansion with $P_k\equiv0$; its limiting penalty then vanishes, and no
nonzero compatibility order is assigned at that degree.

The residual metric is independent of the basis chosen for the row
relations. Replacing $S$ by $SR$, with $R\in\R^{s\times s}$ invertible,
replaces $P_k$ by $R^TP_k$ and $W_S$ by $R^TW_SR$. Their quadratic
combination in \eqref{eq:general-penalty-limit} is unchanged.

\subsection{The mechanical Hessian and the assembled sextic limit}

The affine, fully clamped finite-element model supplies an actual
background energy satisfying the theorem. Let $\Omega\subset\R^3$ be
a bounded Lipschitz domain with a conforming affine tensor-element mesh
whose physical cells are $K_e$.
Let $u_h$ belong to its displacement space contained in
$H^1_0(\Omega;\R^3)$, the Sobolev space with zero external trace.
Its cell pullbacks have tensor degree $p_e$, and the material rule has at
least $p_e+1$ Gauss points in each coordinate. On cell $e$, take the
isochoric neo-Hookean density
\[
 W_{{\rm iso},e}(F)
 =\frac{\mu_e}{2}\bigl((\det F)^{-2/3}F{:}F-3\bigr),
 \qquad\det F>0,
\]
where the cellwise constant shear moduli satisfy
$\mu_e\ge\mu_0>0$. Denote by $E_h$ the sum of these quadrature
energies, with $F=I_3+\nabla_Xu_h$ and the reference-volume measure.

\begin{lemma}[Positive resting mechanical Hessian]
\label{lem:affine-mechanical-positivity}
The energy $E_h$ is stationary at zero. Its Hessian $M=D^2E_h(0)$ obeys
\begin{equation}
 M[u_h,u_h]
 =2\sum_e\mu_e\int_{K_e}
       |\dev\sym\nabla_Xu_h|_F^2
 \ge\mu_0\int_\Omega|\nabla_Xu_h|_F^2.
 \label{eq:affine-mechanical-hessian}
\end{equation}
It is positive definite on the clamped nodal space. If
$\|u_h\|_{L^2}\le C_P\|\nabla_Xu_h\|_{L^2}$, then
$M[u_h,u_h]\ge\mu_0\|u_h\|_{H^1}^2/(1+C_P^2)$.
\end{lemma}

\begin{proof}
For a matrix increment $H$, put
$\sym H=(H+H^T)/2$ and $\dev H=H-\tr(H)I_3/3$.
Expanding the modified first invariant at the identity gives
\[
 (\det(I_3+tH))^{-2/3}|I_3+tH|_F^2
 =3+2t^2|\dev\sym H|_F^2+O(t^3).
\]
The first derivative therefore vanishes and
$D^2W_{{\rm iso},e}(I_3)[H,H]=2\mu_e|\dev\sym H|_F^2$.
On an affine cell, this quadratic physical-gradient integrand has parent
coordinate degree at most $2p_e$ in every coordinate. The stated Gauss
rule integrates it exactly, proving the equality in
\eqref{eq:affine-mechanical-hessian}.

For $u\in H^1_0(\Omega;\R^3)$, integration by parts gives
\[
 \sum_{i,j=1}^3\int_\Omega
       \partial_j u_i\,\partial_i u_j
 =\int_\Omega(\operatorname{div}u)^2.
\]
It follows first for smooth compactly supported fields and then by
density in $H^1_0$ \citep{brezis2011functional}. Expanding the symmetric
and deviatoric parts yields the identity
\begin{equation}
 2\int_\Omega|\dev\sym\nabla_Xu|_F^2
 =\int_\Omega\left(|\nabla_Xu|_F^2
                    +\frac13|\operatorname{div}u|^2\right).
 \label{eq:clamped-deviatoric-gradient-identity}
\end{equation}
Comparison with $\mu_0$ proves the lower bound. A zero value then forces
$u$ to be constant, and its zero trace forces $u=0$. Poincar\'e's
inequality controls the $L^2$ term in the $H^1$ norm.
\end{proof}

This proof uses the mechanical energy directly. A fixed nonnegative
background mean-volume penalty adds
$\kappa_0B_q^T\mathsf V^{-1}B_q$ to the resting Hessian and preserves
positivity. Finite-strain prestress, contact, and follower loads require
their own hypotheses.

Now consider the affine assembled volume maps of
\cref{thm:assembled-volume-geometry}. Keep all $m$ cell-pressure
equations and the complete external displacement clamp. Write
$q_h^Q(u)$ and $q_h^I(u)$ for the raw and exact cell-volume changes of
\eqref{eq:global-volume-maps}. Their common resting Jacobian is $B_q$,
with rank $m-1$. Let $V_e>0$ be the reference cell volumes and set
\[
 \mathsf V=\operatorname{diag}(V_e),\qquad S=\mathbf1,
 \qquad W_S=\sum_eV_e=|\Omega|,
\]
where $\mathbf1\in\R^m$ is the vector of ones.
The aggregate raw constraint
\[
 C_h(u)=\mathbf1^Tq_h^Q(u)
\]
is the homogeneous cubic in \eqref{eq:aggregate-cubic-constraint}; the
supported witness makes it nonzero on $K=\ker B_q$. One cell with the
stated nonzero cubature defect suffices for this conclusion.

\begin{corollary}[Assembled sextic penalty limit]
\label{cor:assembled-sextic-limit}
Assume the affine assembled hypotheses just stated, and let the common
background energy satisfy the mechanical hypotheses of
\cref{thm:finite-penalty-limit}, for example as in
\cref{lem:affine-mechanical-positivity}. At bulk coefficient
$\kappa_\varepsilon=\beta/\varepsilon^4$, the raw localized minimum
problem has limiting energy
\begin{equation}
 J_\beta^{Q,0}(v)
 =\frac12v^TMv-\ell^Tv
  +\frac{\beta}{2|\Omega|}C_h(v)^2,
 \qquad v\in K.
 \label{eq:sextic-volume-limit}
\end{equation}
The exact-volume problem with the same background energy has the same
space $K$ and limiting energy
$J^{I,0}(v)=v^TMv/2-\ell^Tv$.
\end{corollary}

\begin{proof}
The normal correction satisfies $\psi(v,0)=O(\|v\|_2^2)$. Homogeneity
of the cubic gives
\[
 \mathbf1^Tq_h^Q(\Phi(v,0))
 =C_h(v+\psi(v,0))=C_h(v)+O(\|v\|_2^4).
\]
Thus the first reduced compatibility term is $P_3=C_h|_K$.
For the exact-volume map, the fixed external trace gives
$\mathbf1^Tq_h^I\equiv0$, so its reduced relation vanishes identically.
Apply \cref{thm:finite-penalty-limit} with $k=3$, including its
zero-polynomial case for the exact comparison.
\end{proof}

Choose any $v\in K$ with $C_h(v)\ne0$ and load $\ell=Mv$.
The exact limiting energy has unique minimum $v$. The derivative of the
raw limiting energy there, in the admissible direction $v$, is
\[
 DJ_\beta^{Q,0}(v)[v]
 =\frac{3\beta}{|\Omega|}C_h(v)^2>0.
\]
Consequently $v$ is separated from the compact raw minimizing set.
If $d_\beta>0$ is their distance in the $M$ norm, the localized minima
satisfy
\[
 \liminf_{\varepsilon\to0}
 \left\|\frac{u_{\varepsilon,\beta}^Q-
                  u_{\varepsilon,\beta}^I}{\varepsilon}\right\|_M
 \ge d_\beta.
\]
This is a fixed-mesh response consequence. The constants and regular
coordinate neighborhoods needed for a simultaneous mesh/load limit
would require further uniform estimates.

\subsection{Agreement at fixed finite bulk}

For clarity, let $q^Q$ and $q^I$ be any two smooth volume maps with
$q^Q(0)=q^I(0)=0$ and common derivative $B_q$. Hold the background
energy $E$, residual weight $\mathsf V$, and finite bulk coefficient
$\kappa\ge0$ fixed. The complete penalized energy has multiplier
$\pi=\kappa\mathsf V^{-1}q(u)$ and Hessian
\[
 D^2E(u)+\kappa Dq(u)^T\mathsf V^{-1}Dq(u)
       +\sum_{e=1}^m\pi_eD^2q_e(u).
\]
Both pressure contributions are retained. At zero the last term vanishes,
and the two resting Hessians coincide:
\[
 M_\kappa=M+\kappa B_q^T\mathsf V^{-1}B_q\succ0.
\]
The implicit-function theorem applied to the stationary equations gives
unique nearby branches, whose Hessians remain positive for small load,
with expansions
\begin{equation}
 u_\varepsilon^Q=\varepsilon M_\kappa^{-1}\ell+o(\varepsilon),
 \qquad
 u_\varepsilon^I=\varepsilon M_\kappa^{-1}\ell+o(\varepsilon).
 \label{eq:fixed-bulk-response-agreement}
\end{equation}
If $E$ and the maps are $C^3$, the remainders are $O(\varepsilon^2)$.
For $\ell=Mv$ with $v\in K$, one has $M_\kappa v=Mv=\ell$,
so their common leading response is $\varepsilon v$. The extra sextic
term in \eqref{eq:sextic-volume-limit} describes the nonuniform joint
small-load and large-bulk limit.

\subsection{A normalized extremal response}

The general limiting energy also gives an explicit normalized response
when the load is selected from an extremal compatibility mode.
The required norm identity for symmetric multilinear forms on real
Hilbert spaces is classical: their norm equals the supremum on repeated
unit arguments, usually called Banach's theorem. We use it through the
symmetric-tensor formulation discussed by
\citet[Theorem~9]{friedland2011rankone}.

\begin{corollary}[Normalized response for compatibility order $k$]
\label{cor:general-penalty-mode}
In \cref{thm:finite-penalty-limit}, choose $v_*\in K$ with
$\|v_*\|_M=1$ attaining
\[
 p_* =\max_{\substack{v\in K\\\|v\|_M=1}}
                      \|P_k(v)\|_{W_S^{-1}}>0,
 \qquad \ell=Mv_*.
\]
For a dimensionless parameter $0<c\le1/(k-1)$, set
$\beta=c/(kp_*^2)$. Then \eqref{eq:general-penalty-limit} has the
unique global minimum $t_cv_*$, where $t_c\in(0,1)$ is the unique root
\begin{equation}
 t_c+c\,t_c^{2k-1}=1.
 \label{eq:general-penalty-mode}
\end{equation}
For a comparison whose reduced relation vanishes identically and whose
$B_q$ and $M$ are unchanged, the limiting minimum is $v_*$. The
normalized $M$-distance between the limiting responses is $1-t_c$.
\end{corollary}

\begin{proof}
Existence of a maximizer follows from compactness of the $M$-unit sphere
in $K$. For each unit vector $a$ in the residual inner product
$\langle\cdot,\cdot\rangle_{W_S^{-1}}$, polarize the scalar polynomial
$\langle a,P_k(v)\rangle_{W_S^{-1}}$ to a symmetric $k$-linear form.
Its diagonal norm is at most $p_*$, so Banach's theorem bounds its
multilinear norm by $p_*$. Taking the supremum over such $a$ gives
\begin{align}
 \|P_k(v)\|_{W_S^{-1}}&\le p_*\|v\|_M^k,\nonumber\\
 \|D^2P_k(v)[w,w]\|_{W_S^{-1}}
 &\le k(k-1)p_*\|v\|_M^{k-2}\|w\|_M^2,
 \qquad v,w\in K.
 \label{eq:compatibility-polynomial-derivative-bound}
\end{align}
The second estimate follows by placing $k-2$ copies of $v$ and two
copies of $w$ in the polarized form; its value at zero is interpreted
by continuity.

Twice differentiating the limiting penalty yields
\[
 D^2J_\beta^0(v)[w,w]
 =\|w\|_M^2+\beta\|DP_k(v)[w]\|_{W_S^{-1}}^2
    +\beta\langle P_k(v),D^2P_k(v)[w,w]\rangle_{W_S^{-1}}.
\]
The derivative bounds imply
\[
 D^2J_\beta^0(v)[w,w]
 \ge\bigl[1-(k-1)c\|v\|_M^{2k-2}\bigr]\|w\|_M^2.
\]
Hence the energy is strictly convex in the open $M$-unit ball.
Every global minimum lies in that ball. Indeed, radial stationarity
and Euler's identity for $P_k$ give
\[
 \|v\|_M^2-\langle v_*,v\rangle_M
       +k\beta\|P_k(v)\|_{W_S^{-1}}^2=0.
\]
Cauchy--Schwarz excludes $\|v\|_M>1$. Equality at unit norm would
require $v=v_*$ and $P_k(v_*)=0$, contradicting $p_*>0$.

Stationarity of $\|P_k\|_{W_S^{-1}}^2$ on the unit sphere, together
with its radial derivative, gives for every $w\in K$
\[
 \langle DP_k(v_*)[w],P_k(v_*)\rangle_{W_S^{-1}}
 =kp_*^2\langle v_*,w\rangle_M.
\]
Homogeneity therefore makes $t v_*$ stationary exactly when
$t+ct^{2k-1}=1$. The left side is strictly increasing for $t\ge0$;
its values at zero and one bracket 1. Its unique root lies in $(0,1)$.
The energy is coercive, so a global minimum exists. All such minima
lie in the strictly convex ball and must coincide with this stationary
point. The comparison energy is quadratic on $K$ with minimum $v_*$,
and $\|v_*\|_M=1$ proves the distance formula.
\end{proof}

The maximization defining $v_*$ selects a mathematical load; finding that
global polynomial maximizer is not asserted to be inexpensive.
For $k=3$ and $c=1/2$, \eqref{eq:general-penalty-mode} gives
$t_c\simeq0.8174710190$ and a normalized separation of approximately
$18.2529\%$. This extremal normalization differs from the conservative
explicit-load two-cell calculation, whose normalized separation is about
$1.069\times10^{-4}$. Each calculation has its own load, strength, and
stationarity evidence.

An exact-volume comparison supplies the zero reduced relation directly
through its total-volume identity. More generally, a smooth comparison map
of constant rank $\rank B_q$ near zero has the same property in its own
regular coordinates: its local image is a graph over
$\operatorname{range}B_q$, and the graph's only nearby point with zero
range projection is zero. Thus solving its regular equations also solves
its complete constraint equations.

\section{A complete cubic-cell model}
\label{sec:sharp-models}

The preceding sections connect the supported cubic to assembled
geometry and mechanical response. On a \(\mathbb Q_3\) cell, the entire
interior volume equation can also be written explicitly.
Its normal form identifies inactive coordinates and resolves the
regular, quadratic and cubic local error regimes. Both results below
concern the full interior space of one cell; assembled sharpness was
proved in \cref{cor:assembled-cubic-sharpness}.

Fix one affine cell \(X_e(\xi)=b_e+A_e\xi\) and prescribe zero displacement
on all its faces. Define
\[
\phi_0(t)=1-t^2,\qquad
\phi_1(t)=t(1-t^2).
\]
Every scalar Q3 polynomial with zero face trace is the product
\(\prod_{\ell=1}^3(1-\xi_\ell^2)\) times a Q1 polynomial. Consequently
every displacement in this cell-interior
space has the unique representation
\begin{equation}
u\circ X_e
=A_e\sum_{i,j,k\in\{0,1\}}
d_{ijk}\phi_i(\xi_1)\phi_j(\xi_2)\phi_k(\xi_3),
\qquad d_{ijk}\in\mathbb R^3.
\label{eq:q3-bubble-coordinates}
\end{equation}
There are eight scalar basis functions and 24 displacement coordinates.
The factor \(A_e\) fixes the same dimensionless amplitude convention as
\eqref{eq:physical-bubble-coefficient}.

\begin{theorem}[Complete Q3 interior volume normal form]
\label{thm:q3-volume-normal-form}
Use four Gauss points per coordinate and define the coefficient matrices
by their columns:
\begin{equation}
\mathsf A=[d_{001},d_{010},d_{100}],\qquad
\mathsf B=[d_{110},d_{101},d_{011}].
\label{eq:q3-coefficient-matrices}
\end{equation}
Then the complete cell-volume changes on \eqref{eq:q3-bubble-coordinates} are
\begin{equation}
q_{h,e}^I(u)=0,\qquad
q_{h,e}^Q(u)=\frac{V_e}{8}\chi_4\,
                   \mathsf A:\operatorname{cof}\mathsf B,\qquad
\chi_4=\frac{131072}{4501875}.
\label{eq:q3-complete-volume-normal-form}
\end{equation}
The vectors \(d_{000}\) and \(d_{111}\) do not enter this scalar equation.
\end{theorem}

\begin{proof}
The exact cell volume is fixed by the zero boundary trace. By
\eqref{eq:bubble-pure-volume-cubic}, only the cubic displacement term
contributes to the Gauss4 volume change.
We compute that term for all eight scalar basis functions.

For a triple of one-dimensional factors
\((\phi_{\alpha_1},\phi_{\alpha_2},\phi_{\alpha_3})\), with
\(\alpha_r\in\{0,1\}\), use the one-dimensional four-point Gauss rule
\(\mathcal G_4^{(1)}\) to form the three-vector of derivative moments
\[
m_r=\mathcal G_4^{(1)}
\left[\phi_{\alpha_r}'\prod_{s\ne r}\phi_{\alpha_s}\right],
\qquad r=1,2,3.
\]
If the number of factors \(\phi_1\) is even, parity gives \(m=0\).
If there is exactly one such factor, in slot \(j\), the polynomials have
degree at most six and exact integration gives
\[
m=C(3e_j-\mathbf1_3),\qquad C=\frac{32}{105}.
\]
Here \(e_j\) is the \(j\)-th coordinate vector of \(\mathbb R^3\).
If all three factors are \(\phi_1\), the exact moment is zero and
\[
m=\lambda\mathbf1_3,\qquad \lambda=\frac{128}{3675}.
\]
Indeed, \(\phi_1'\phi_1^2\) has degree eight with leading coefficient
\(-3\). The first-missed-moment calculation in
\Cref{thm:tensor-volume-admission} gives
\(\mathcal G_4^{(1)}[\phi_1'\phi_1^2]
=3\int_{-1}^1\pi_4^2=128/3675\), where \(\pi_4\) is the monic
degree-four Legendre polynomial.

For three scalar tensor basis functions, the integrated determinant of
their three gradients is the determinant of the three moment columns,
one column for each coordinate.
With exactly one \(\phi_1\) in a coordinate, that column belongs to
\(\mathbf1_3^\perp\). Three such columns have determinant zero.
Two columns from the all-\(\phi_1\) case are parallel and also give zero.
Thus a nonzero determinant requires exactly one all-\(\phi_1\) coordinate
and two distinct single-\(\phi_1\) slots in the other coordinates.

Writing the scalar basis indices in lexicographic order, the only nonzero
triples and their coefficients are
\[
\begin{array}{c|r}
(001,011,101)&-9\lambda C^2\\
(010,011,110)& 9\lambda C^2\\
(100,101,110)&-9\lambda C^2 .
\end{array}
\]
Their common absolute coefficient is
\(9\lambda C^2=131072/4501875=\chi_4\).
This classifies all \(\binom83=56\) distinct basis triples.
Write \(\widehat v=A_e^{-1}(u\circ X_e)\) for the dimensionless parent
displacement. Expanding its gradient determinant over three-column
minors, the Cauchy--Binet formula gives
\begin{align*}
\mathcal G_4[\det D_\xi\widehat v]
={}&-\chi_4\det[d_{001},d_{011},d_{101}]\\
&+\chi_4\det[d_{010},d_{011},d_{110}]
-\chi_4\det[d_{100},d_{101},d_{110}].
\end{align*}
Expansion of the cofactor by columns identifies this expression as
\(\chi_4\,\mathsf A:\operatorname{cof}\mathsf B\).
None of the nonzero triples contains \(d_{000}\) or \(d_{111}\).
Finally the physical push-forward in \eqref{eq:q3-bubble-coordinates}
multiplies this determinant integral by
\(\det A_e=V_e/8\), proving \eqref{eq:q3-complete-volume-normal-form}.
\end{proof}

The inactive coefficients still change the displacement and material energy.
Their absence concerns only this scalar volume equation.
The 18 remaining coefficients have a useful explicit feasible-distance
description.

\subsection{A constructive bound and its local exponents}

Put
\[
\mathfrak p(\mathsf A,\mathsf B)=\mathsf A:\operatorname{cof}\mathsf B,
\qquad
\mathcal Z=\{(\mathsf A,\mathsf B):\mathfrak p(\mathsf A,\mathsf B)=0\},
\]
and use the product Frobenius norm
\(\|(\mathsf A,\mathsf B)\|_F^2
=\|\mathsf A\|_F^2+\|\mathsf B\|_F^2\).
Distance in this norm is denoted \(\operatorname{dist}_F\).
The six inactive coefficients can be held fixed in every correction.

\begin{proposition}[Sharp Q3 error-bound strata]
\label{prop:q3-error-strata}
For every pair of coefficient matrices,
\begin{equation}
\operatorname{dist}_F((\mathsf A,\mathsf B),\mathcal Z)
\le\sqrt2\,|\mathfrak p(\mathsf A,\mathsf B)|^{1/3}.
\label{eq:q3-constructive-error-bound}
\end{equation}
At a feasible point, let \(\theta\) be the largest exponent for which
there are a constant \(C_{\rm loc}>0\) and a neighborhood throughout
which
\[
\operatorname{dist}_F((\mathsf A,\mathsf B),\mathcal Z)
\le C_{\rm loc}|\mathfrak p(\mathsf A,\mathsf B)|^\theta
\]
holds. Then \(\theta=1\) when
\(D\mathfrak p\ne0\), \(1/2\) when \(D\mathfrak p=0\) and
\((\mathsf A,\mathsf B)\ne(0,0)\), and \(1/3\) when
\(\mathsf A=\mathsf B=0\).
These conclusions include arbitrary values of the inactive coefficients.
More generally, for a scalar block-tri-affine polynomial with nonzero
cubic part, the optimal local exponent at a feasible point is \(1\),
\(1/2\), or \(1/3\) according as its first nonzero derivative has order
one, two, or three.
\end{proposition}

\begin{proof}
Write \(c_B=\|\operatorname{cof}\mathsf B\|_F\).
If \(c_B=0\), the pair is already feasible.
Otherwise, orthogonally projecting \(\mathsf A\) onto
\((\operatorname{cof}\mathsf B)^\perp\) gives a feasible pair at distance
\(|\mathfrak p|/c_B\).
A second feasible pair is obtained by keeping only the largest singular
component of \(\mathsf B\).
If its singular values are \(\sigma_1\ge\sigma_2\ge\sigma_3\ge0\), this
correction costs \((\sigma_2^2+\sigma_3^2)^{1/2}\).
Since
\[
c_B^2=\sigma_1^2\sigma_2^2+\sigma_1^2\sigma_3^2
                         +\sigma_2^2\sigma_3^2\ge\sigma_2^4,
\]
that cost is at most \(\sqrt2\,c_B^{1/2}\).
Consequently
\[
\operatorname{dist}_F((\mathsf A,\mathsf B),\mathcal Z)
\le\min\{|\mathfrak p|/c_B,\sqrt2\,c_B^{1/2}\}
\le\sqrt2\,|\mathfrak p|^{1/3}.
\]
The last inequality follows by splitting at \(c_B=|\mathfrak p|^{2/3}\).

We next examine a feasible base point.
In the Frobenius coordinates,
\[
D\mathfrak p=(\operatorname{cof}\mathsf B,\,
                         \mathsf B\times\mathsf A),
\]
where \(\times\) is the cofactor polarization from
\Cref{sec:discrete-volume-setting}.
At a regular point, the implicit-function theorem supplies a local linear
error bound \citep{dontchev2014implicit}. A transverse ray has residual
and feasible distance both of first order, proving optimality of exponent one.

At a critical point, represent the Hessian on the two nine-entry matrix
blocks. It is
\begin{equation}
D^2\mathfrak p=
\begin{bmatrix}
0&D\operatorname{cof}(\mathsf B)\\
D\operatorname{cof}(\mathsf B)&D\operatorname{cof}(\mathsf A)
\end{bmatrix}.
\label{eq:q3-normal-form-hessian}
\end{equation}
The cofactor derivative is self-adjoint because it is the Hessian of the
determinant. The map \(R\mapsto D\operatorname{cof}(R)\) is injective:
for each entry of \(R\), evaluation of \(D^2\det(R)\) on two matrix units
in the complementary rows and columns recovers that entry up to sign.
Thus \eqref{eq:q3-normal-form-hessian} is nonzero whenever
\((\mathsf A,\mathsf B)\ne(0,0)\).
Every diagonal Hessian entry is zero, since \(\mathfrak p\) is affine in
each scalar entry. A nonzero symmetric trace-zero Hessian has both signs.

Choose unit coefficient directions \(e_+,e_-\) with strictly positive
and negative second variations. Continuity supplies a neighborhood and
\(\mu_*>0\) on which those variations are respectively at least \(\mu_*\)
and at most \(-\mu_*\).
At a nearby point with \(\mathfrak p>0\), choose the sign of \(e_-\) so
the first directional derivative is nonpositive. Along its segment,
Taylor's theorem gives
\[
\mathfrak p(z+te_-)\le\mathfrak p(z)-\mu_* t^2/2.
\]
Here \(z\) denotes the coefficient pair.
A zero occurs at distance at most \(\sqrt{2\mathfrak p(z)/\mu_*}\).
For negative residual use \(e_+\).
Shrinking the starting neighborhood keeps these segments inside the
curvature neighborhood, proving the local square-root upper bound.

For sharpness, choose a ray from the critical base with nonzero second
variation. Its residual is \(\Theta(t^2)\).
If feasible points lay at distance \(o(t)\), their quadratic Taylor
expansion at that base would force this second variation to vanish.
The feasible base itself gives the matching order-\(|t|\) upper distance.
This excludes every exponent above \(1/2\).

At \(\mathsf A=\mathsf B=0\), take
\(\mathsf A=tI_3,\mathsf B=tI_3\).
The residual is \(3t^3\), while homogeneity of the closed zero set gives
distance \(|t|\operatorname{dist}_F((I_3,I_3),\mathcal Z)>0\).
Together with \eqref{eq:q3-constructive-error-bound}, this proves the
optimal cube-root exponent. The inactive coordinates have no effect on
these distances or Taylor expansions.

For the stated generalization, block-tri-affinity gives zero diagonal
Hessian blocks. At a critical point with nonzero Hessian, the same
trace-zero and sign-direction argument proves the sharp square-root
bound. The regular case again follows from the implicit-function theorem.
If the first two derivatives vanish at a feasible base \(z_0\), the exact
Taylor expansion is \(f(z_0+\delta)=P_3(\delta)\), where \(P_3\) is its
nonzero homogeneous trilinear part.
Equation \eqref{eq:scalar-trilinear-error-bound} gives the cube-root
upper bound, and a nonzero-cubic ray in the translated zero cone proves
sharpness exactly as above.
\end{proof}

Fixed-dimensional norm equivalence transfers these local exponents to the
nodal or physical displacement norm on the fixed cell.
At an interior point of an admissible coefficient neighborhood, sufficiently
small corrections remain admissible. Uniformity across a mesh instead uses
the explicitly scaled construction of \Cref{thm:cellwise-feasibility}.
The scalar strata describe this complete local volume equation; they do not
classify every singular point of a coupled assembled constraint system.

\section{Tensor-element evidence}
\label{sec:evidence}

The numerical comparisons examine the added volume relation, the
constructive correction and the singular mechanical response.
The principal assembled model is the unit cube split into two affine
cells along the first coordinate. Continuous \(\mathbb Q_3\) displacements
with equispaced nodes and a complete outer clamp have \(60\) free
coordinates; both pressure equations are retained. Four Gauss points
per coordinate define raw volume and five integrate it exactly.
A one-cell \(\mathbb Q_4\) comparison retains its one pressure equation
and \(81\) interior displacement coordinates, with five-point raw and
six-point exact-volume rules.

Both models use the isochoric neo-Hookean energy of
\cref{sec:discrete-volume-setting}, with \(\mu=2\), under the same material
rule in each raw/exact comparison: Gauss four for \(\mathbb Q_3\),
Gauss five for \(\mathbb Q_4\).
Calculations use an independent polynomial finite-element
implementation in IEEE binary64. Rational integration and selected
\(80\)-digit checks are described in \cref{app:numerical-details}.

The two-cell observations in \cref{fig:tensor-rank-response} summarize
the link between rank and response. For the rank comparison,
\(u_\tau=\tau(U_2+U_3)\), using the supported directions in the first
cell. For the response comparison, let
\[
 M=D^2E_h(0),\qquad
 v_0=\frac{U_1+U_2+U_3}{\|U_1+U_2+U_3\|_M},\qquad
 \ell=Mv_0,\qquad \|v\|_M=(v^TMv)^{1/2}.
\]
The load is \(\varepsilon\ell\) and the bulk coefficient is
\(\kappa=\beta/\varepsilon^4\). The plotted response uses the
conservative strength \(\beta=3.4539763764\); its complete parameter
table is retained in \cref{app:tensor-response-protocol}.

\begin{figure}[H]
\centering
\includegraphics[width=\linewidth]{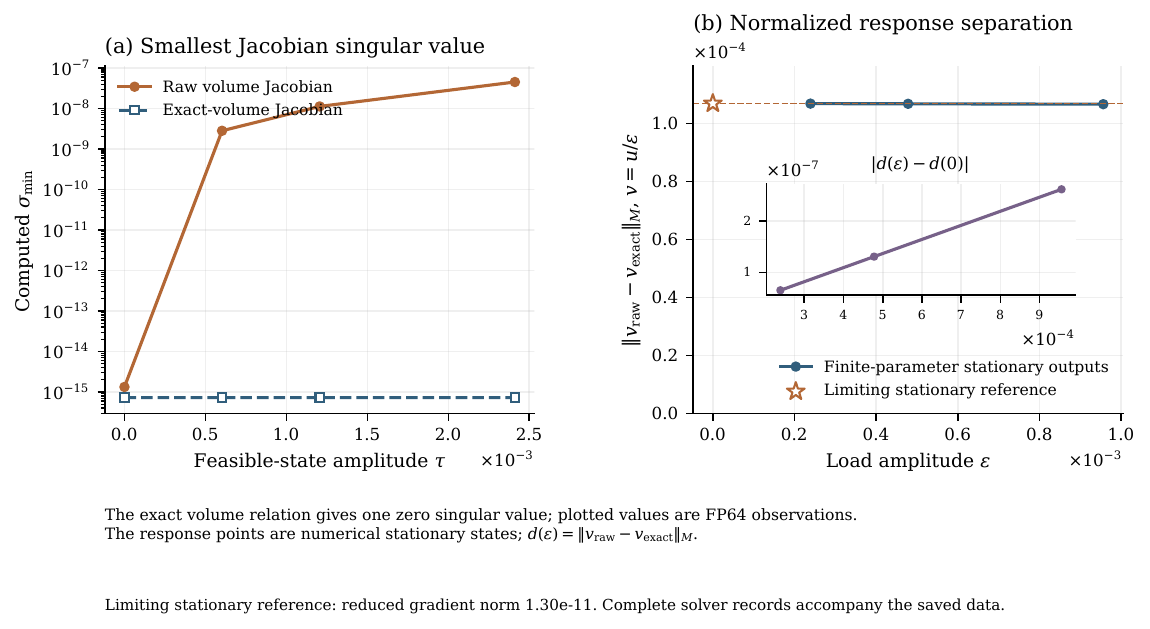}
\caption{Assembled tensor-element rank and response. The left panel shows
the smallest volume-Jacobian singular values at the feasible states
$u_\tau$. The right panel gives the $M$-norm separation of normalized raw
and exact-volume stationary responses, with the independently computed
limiting stationary reference. Its reduced gradient residual is reported.
The inset shows the absolute difference between the finite and limiting
separations. The exact-map null singular value is observed at rounding scale.}
\label{fig:tensor-rank-response}
\end{figure}

\subsection{The additional assembled constraint}

At three successively halved feasible-state amplitudes, the smallest
raw Jacobian singular value decreases from
\(4.518206883\times10^{-8}\) to \(2.823879839\times10^{-9}\);
the exact-map values remain near \(7.28\times10^{-16}\).
The raw value falls by approximately a factor of four per halving, as predicted by
\cref{thm:assembled-volume-geometry}. The \(\mathbb Q_4\) comparison
has the same quadratic scale. The observed rest-state offset is kept
in the detailed derivative comparison, rather than interpreted as
a change in the asymptotic order.

\Cref{app:tensor-rank-records} gives every amplitude, the missing-component
action, the two-cell flux minor, exact coefficient values and the
higher-order checks. Independent rational integration verifies the
specified fourth- and sixth-order witnesses; exact-order and
low-transverse-order controls have zero supported defect.

\subsection{Constructed feasibility and unchanged exact volumes}

Let \(X=(X_1,X_2,X_3)\), \(e_1=(1,0,0)^T\), and define the shared-face
direction
\[
 \widetilde w(X)=\min(2X_1,2-2X_1)
              4X_2(1-X_2)4X_3(1-X_3)e_1.
\]
Its outer trace vanishes. The sum of absolute parent-monomial
coefficients bounds each nonzero gradient entry by eight, so take
\(L_w=8\sqrt3\), \(w=\widetilde w/L_w\) and \(u_0=\delta w\).
The three starting scales are
\(\delta=1.379475369\times10^{-12}\) and its halves and quarters,
chosen from the maximum-gradient correction estimate.

At the largest scale, the four-corner correction changes the two raw
volumes from magnitudes \(4.42467\times10^{-14}\) to
\(-2.55\times10^{-21}\) and \(6.03\times10^{-21}\).
Their exact changes remain \(+4.42467\times10^{-14}\) and
\(-4.42467\times10^{-14}\), with all face traces fixed.
The correction \(\Delta u\) has \(H^1\) seminorm
\(2.23509\times10^{-4}\) and full norm \(2.24053\times10^{-4}\).
Its binary64 polynomial maximum-gradient bound is \(0.0158\),
without an outward-rounded interval certificate.
Zero is a closer feasible field if interior-face changes are allowed:
the norm measures this fixed-trace construction.

\begin{figure}[H]
\centering
\includegraphics[width=\linewidth]{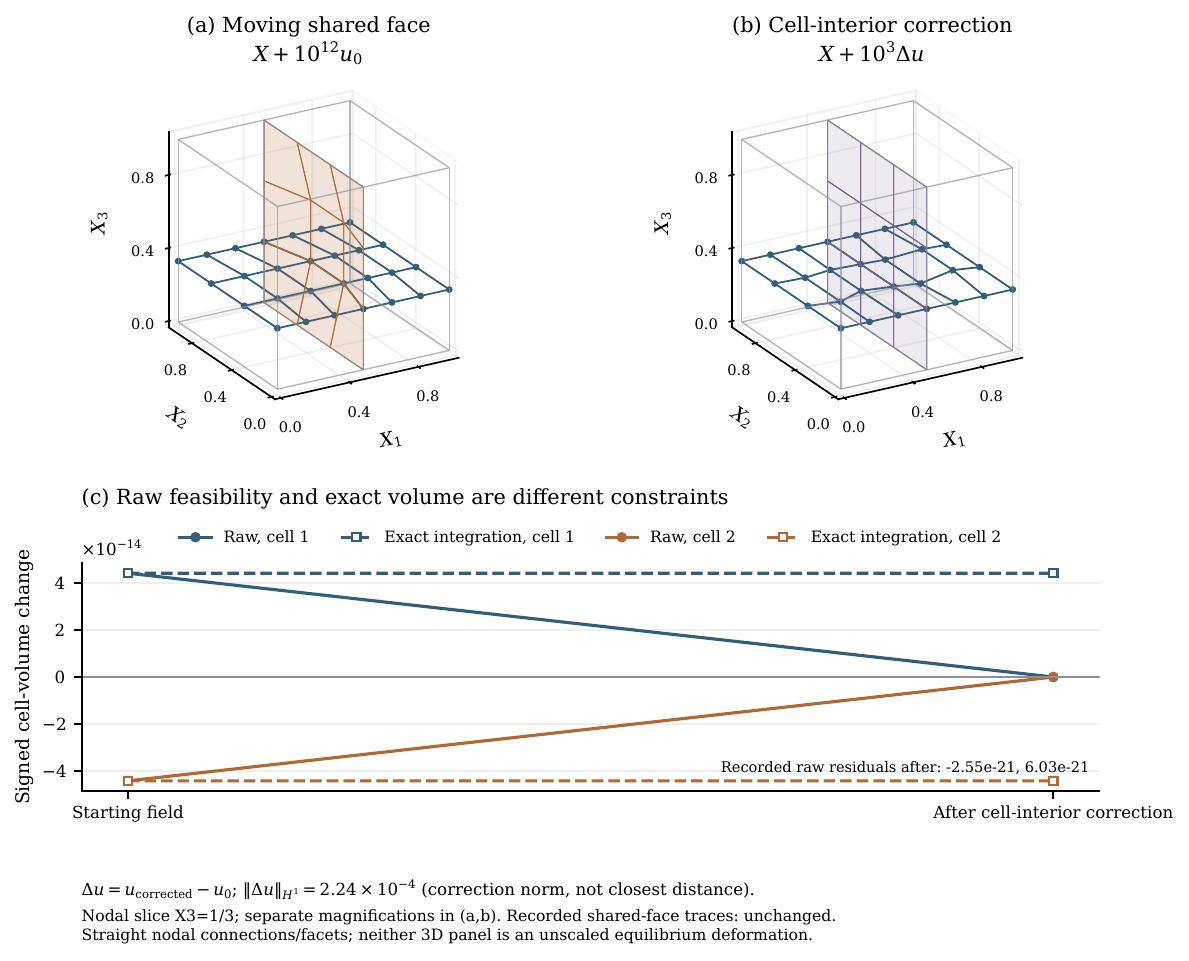}
\caption{Moving-face and cell-interior volume mechanisms at the largest
starting scale. Panel (a) displays $X+10^{12}u_0$; panel (b) displays
$X+10^3\Delta u$, where $\Delta u$ is the correction alone.
The nodal slice is $X_3=1/3$; segments and facets join retained nodes.
Panel (c) compares the signed volume changes before and after correction.
The displayed $H^1$ quantity is the full displacement-and-gradient
norm of the constructed correction.
Different geometric magnifications are stated in the panels.}
\label{fig:moving-face-correction}
\end{figure}

\subsubsection{Refinement}
\label{sec:refined-volume-observations}

Divide the unit cube into \(N^3\) affine cells, \(N\in\{1,2,4\}\),
with continuous \(\mathbb Q_3\) displacements, the complete outer clamp
and the same Gauss-four/five volume comparison.
One starting field is the supported ray
\(\tau(U_1+U_2+U_3)\) in the first cell. The other interpolates
\[
 \widetilde w(X)=16\prod_{i=1}^3X_i(1-X_i)e_1,\qquad
 w=\widetilde w/\max\{1,L_{\widetilde w}\},\qquad u_0=\delta w,
\]
where \(L_{\widetilde w}\) bounds the gradient by its polynomial
coefficients. This field moves interior faces for \(N=2,4\).
Four dyadic scales start at
\(\tau_0=9.246120462\times10^{-4}\) and
\(\delta_0=3.287737203\times10^{-12}\).
The complete amplitude and represented-input bounds appear in
\cref{app:refinement-protocol}. In all \(24\) pairs, the starting and
corrected deformations are injective and orientation preserving,
with displacement-gradient operator norm below \(1/4\).

\Cref{fig:journal-correction} shows cubic supported-ray residuals and
linear constructed correction norms. The moving-face fields become
raw feasible to the stated numerical scale while their exact
cell-volume changes remain fixed; independent \(80\)-digit evaluation
preserves them through all \(65\) reported significant digits.
The one-cell full-space field has exact zero volume defect.
Its residual near \(10^{-30}\) increases slightly after a rounding-driven
correction. That control, and the distinction between a constructed norm
and a closest distance, remain part of the comparison.

\begin{figure}[H]
\centering
\includegraphics[width=\linewidth]{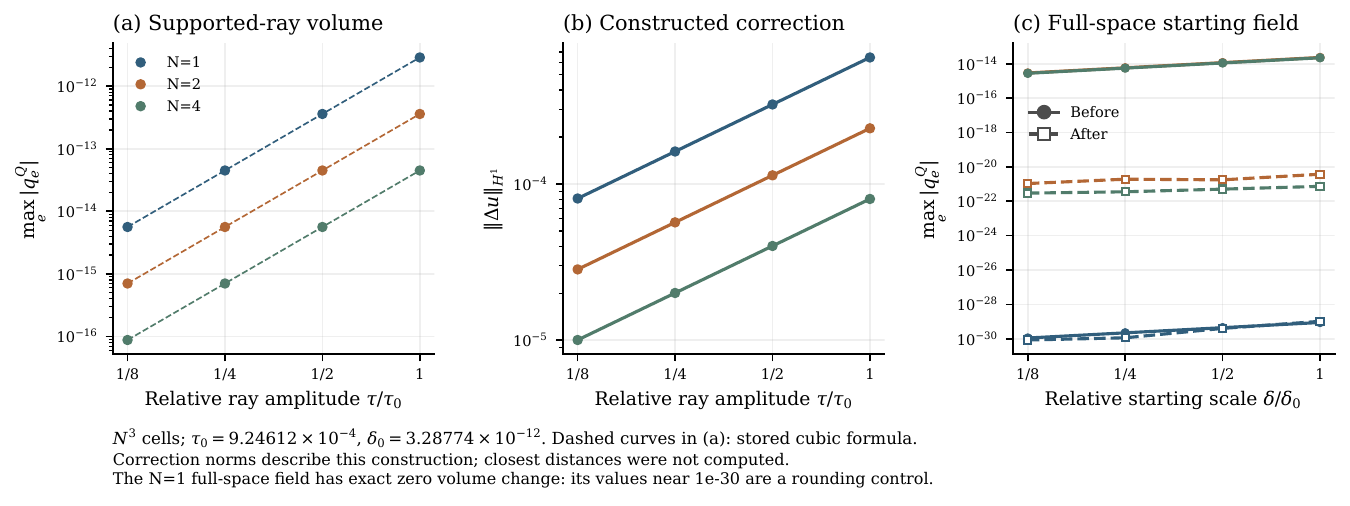}
\caption{Refined $\mathbb Q_3$ volume and correction observations.
Here $N$ is the number of cells per coordinate, and each horizontal
coordinate is divided by its recorded first amplitude.
The dashed curves in (a) are the cubic formula for the supported ray.
Panel (b) gives the $H^1$ norm of the constructed change $\Delta u$.
Panel (c) preserves the direct raw residuals before and after correction,
including the one-cell exact-zero control.
Closest-point distances were not computed.}
\label{fig:journal-correction}
\end{figure}

\subsection{Singular and fixed-bulk response}
\label{sec:tensor-response-evidence}

The response experiment uses the witness load \(\ell=Mv_0\) above,
with the corresponding \(M\)-normalized triple on each model.
The six conservative two-cell solves have scaled residual at most
\(1.61\times10^{-12}\). Their normalized separation approaches
\(1.069054666\times10^{-4}\), about \(0.0107\%\).
The limiting reference has a small nonzero gradient, whose effect
is quantified in \cref{sec:conservative-response-precision}.
These are stationary calculations; the general minimum theorem
does not certify finite-parameter minimality.

\subsubsection{Penalty strength and the fixed-bulk control}
\label{sec:extended-response-observations}

Let $P_3=C_h|_{\ker B_h}$ be the reduced cubic from
\cref{sec:singular-penalty}, and let $W_S=\sum_eV_e=1$.
Write $b$ for the Frobenius upper bound on
$|P_3(v)|/\sqrt{W_S}$ on the $M$-unit sphere, and define
\[
 s_0=\frac{|P_3(v_0)|}{\sqrt{W_S}},\qquad
 \beta_{\rm s}=\frac1{24b^2},\qquad
 \beta_{\rm l}=\frac1{6s_0^2}.
\]
For $\mathbb Q_3$, these strengths are $3.4539763764$ and
$16994.9915686$; for $\mathbb Q_4$, they are $0.6586461161$ and
$112440.650192$.
The conservative choice has the reduced unit-ball Hessian lower estimate
$3/4$. The load-normalized choice is a stronger stationary-branch
illustration without that convexity guarantee.
The loads are the explicit witnesses defined above.

Each model and strength uses four dyadic load amplitudes
\(\varepsilon_j=\varepsilon_0/2^j\), \(j=0,1,2,3\), with first
amplitudes and stationary initialization specified in
\cref{app:tensor-response-protocol}.
The fixed-bulk control uses \(\kappa/\mu=1000\).

\begin{table}[H]
\centering\small
\caption{Extended normalized stationary-response differences
$d_j=\|u_Q/\varepsilon_j-u_I/\varepsilon_j\|_M$.
The limiting column is a separate reduced stationary reference.
For fixed bulk, the leading raw/exact gap is zero and no measured
nonzero limiting gap is asserted.}
\label{tab:extended-response}
\begin{tabular}{@{}llrrrr@{}}
\toprule
Model & Regime & $d_0$ & $d_3$ & Limiting reference & Largest $\kappa/\mu$\\
\midrule
$\mathbb Q_3$ & $\beta_{\rm s}$ & $1.06643\,10^{-4}$ & $1.06873\,10^{-4}$ & $1.06905\,10^{-4}$ & $8.46743\,10^{15}$\\
$\mathbb Q_3$ & $\beta_{\rm l}$ & $0.182606$ & $0.182805$ & $0.182833$ & $4.16633\,10^{19}$\\
$\mathbb Q_3$ & Fixed bulk & $7.91385\,10^{-9}$ & $1.23779\,10^{-10}$ & $0$ & $1000$\\
$\mathbb Q_4$ & $\beta_{\rm s}$ & $6.48265\,10^{-6}$ & $6.50389\,10^{-6}$ & $6.50692\,10^{-6}$ & $4.66038\,10^{14}$\\
$\mathbb Q_4$ & $\beta_{\rm l}$ & $0.156291$ & $0.156405$ & $0.156421$ & $1.30586\,10^{21}$\\
$\mathbb Q_4$ & Fixed bulk & $8.24132\,10^{-14}$ & $3.36969\,10^{-13}$ & $0$ & $1000$\\
\bottomrule
\end{tabular}
\end{table}

The \(48\) stationary outputs include the six cases in
\cref{tab:tensor-response}. Six exact-volume solves report no progress
at scaled residuals \(4.92\times10^{-13}\)--\(1.11\times10^{-12}\);
these six outcomes remain marked as non-successes.
Every field passes the whole-cell admissibility check in
\cref{app:refinement-protocol}, with displacement-gradient operator
norm below \(1/4\).

The load-normalized comparisons give approximately $18.3\%$ and $15.6\%$
normalized separation. Their absolute displacement differences still shrink:
at the smallest strong loads the $M$-norm differences are
$2.18460\times10^{-5}$ and $1.26692\times10^{-5}$.
The large bulk/shear ratios in \cref{tab:extended-response} are part of
this conclusion. At fixed bulk, the $\mathbb Q_3$ gap decreases toward
zero, while the $\mathbb Q_4$ values fluctuate at the numerical floor.
No convergence rate is inferred from that latter sequence.

\begin{figure}[H]
\centering
\includegraphics[width=\linewidth]{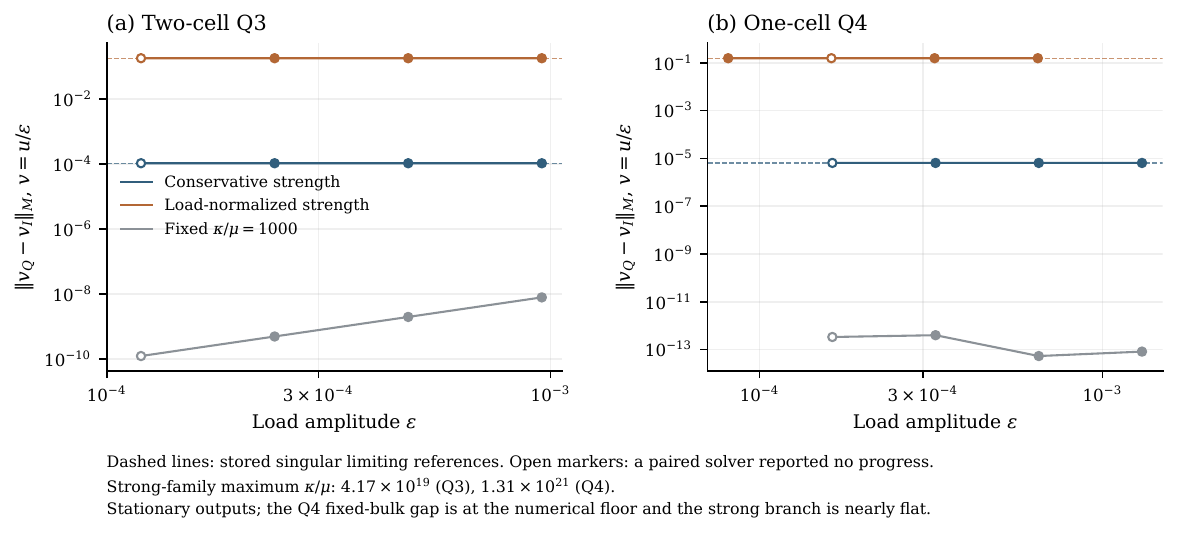}
\caption{Finite-strain $\mathbb Q_3$ and $\mathbb Q_4$ response comparisons.
The vertical quantity is the $M$-norm gap between normalized displacements
$v_Q=u_Q/\varepsilon$ and $v_I=u_I/\varepsilon$.
Dashed lines are stored singular limiting references.
Open markers preserve pairs with a solver non-success flag.
The load-normalized families use extreme bulk stiffness; the fixed-bulk
curves provide their separate control. These are stationary observations,
with no finite local or global minimum certification.}
\label{fig:journal-response}
\end{figure}

Pressure recovery and finite-to-limit sensitivity are assessed in
\cref{sec:strong-response-precision}. The stronger \(\mathbb Q_4\)
branch is nearly flat: its broad raw/exact gap is better resolved
than the fine correction to its limiting stationary reference.

\section{A second-order contrast on curved quadratic tetrahedra}
\label{sec:cubature}

Curved quadratic tetrahedra provide the second-order contrast to the
tensor construction. On the explicit template below, volume and its
first derivative are exact at rest, while a quadratic compatibility
defect changes feasible motion and produces a quartic penalty limit.
We develop the local defect, assemble the template and recover exact
volume from the same four Jacobian samples.

We first derive the local identity needed by that construction.
On the parent
\[
\That=\{\xi\in\R^3:\xi_i\ge0,\ \xi_1+\xi_2+\xi_3\le1\},\qquad |\That|=1/6,
\]
write \(\lambda_0=1-\xi_1-\xi_2-\xi_3\), \(\lambda_i=\xi_i\) for
\(i=1,2,3\), and let
\(\mathbb P_r(\That)\) denote total degree at most \(r\).
Tet10 uses \(N_i=\lambda_i(2\lambda_i-1)\) at the four vertices and
\(N_{ij}=4\lambda_i\lambda_j\) on its six edges. For nodal positions
\(X_A,y_A\in\R^3\), \(A=1,\ldots,10\), set
\[
X_h=\sum_A N_AX_A,\qquad y_h=\sum_A N_Ay_A,\qquad
J_X=\nabla_\xi X_h,\quad J_y=\nabla_\xi y_h.
\]
For \(\det J_X>0\), the physical deformation gradient and the gradients
of quadratic variations \(u_h,v_h\) are
\begin{equation}
F=J_yJ_X^{-1},\qquad H=J_uJ_X^{-1},\qquad L=J_vJ_X^{-1}.
\label{eq:isoparametric-F}
\end{equation}
Here \(J_u,J_v\) are the parent gradients.
The tetrahedral rule \(\Qfour\) has weight \(1/24\) at each permutation of
\begin{equation}
(\alpha_4,\beta_4,\beta_4,\beta_4),\qquad
\alpha_4=\frac{5+3\sqrt5}{20},\qquad
\beta_4=\frac{5-\sqrt5}{20}.
\label{eq:q4-points}
\end{equation}
It is exact through total degree two \citep{shunn2012symmetric}.

\begin{proposition}[Determinant pullback]
\label{prop:det-pullback}
For fixed \(J_X\in\GLp\) and \(J_y=FJ_X,\ J_u=HJ_X,\ J_v=LJ_X\),
\begin{equation}
\det J_X\,D^2\det(F)[H,L]=D^2\det(J_y)[J_u,J_v].
\label{eq:det-pullback}
\end{equation}
\end{proposition}
\begin{proof}
Differentiate \(\det(J_y+sJ_u+tJ_v)=\det(F+sH+tL)\det J_X\)
once in each of \(s,t\) at zero.
\end{proof}

An affine matrix field is determined by its vertex values:
\(A(\lambda)=\sum_{a=0}^3\lambda_aA_a\).
Put \(\bar A=\frac14\sum_aA_a\), write
\(\mathcal I_{\That}[f]=\int_{\That}f\), and define
\begin{equation}
d_4(A)=\Qfour[\det A]-\mathcal I_{\That}[\det A].
\label{eq:defect-definition}
\end{equation}

\begin{theorem}[Exact symmetric four-point determinant defect]
\label{thm:cubic-defect}
For every affine matrix field \(A\) on \(\That\),
\begin{equation}
d_4(A)=c_4\sum_{a=0}^3\det(A_a-\bar A),\qquad
c_4=\frac{3\sqrt5-5}{1800}.
\label{eq:cubic-defect}
\end{equation}
\end{theorem}
\begin{proof}
Let \(B_a=A_a-\bar A\), so \(\sum_aB_a=0\), and
\(B(\lambda)=\sum_a\lambda_aB_a\).
All terms of \(\det(\bar A+B)\) containing \(\bar A\) have degree
at most two in \(\lambda\). Hence
\begin{equation}
d_4(A)=(\Qfour-\mathcal I_{\That})[\det B].
\label{eq:centered-error}
\end{equation}
Let \(e_3,e_{21},e_{111}\) be the respective errors for
\(\lambda_i^3,\lambda_i^2\lambda_j,\lambda_i\lambda_j\lambda_k\),
with distinct indices where indicated. Symmetry and the moment formula
\[
\int_{\That}\prod_a\lambda_a^{m_a}
=\frac{\prod_am_a!}{(3+\sum_am_a)!}
\]
give
\[
\begin{array}{c|c|c}
f&\Qfour[f]&\mathcal I_{\That}[f]\\ \hline
\lambda_i^3&(\alpha_4^3+3\beta_4^3)/24&1/120\\
\lambda_i^2\lambda_j&
(\alpha_4^2\beta_4+\alpha_4\beta_4^2+2\beta_4^3)/24&1/360\\
\lambda_i\lambda_j\lambda_k&
(3\alpha_4\beta_4^2+\beta_4^3)/24&1/720 .
\end{array}
\]
Denote normalized determinant polarization by
\(\mathcal P_{\det}\), with
\(\mathcal P_{\det}(M,M,M)=\det M\), and set
\[
S_3=\sum_i\det B_i,\quad
S_{21}=\sum_{i\ne j}\mathcal P_{\det}(B_i,B_i,B_j),\quad
S_{111}=\sum_{i<j<k}\mathcal P_{\det}(B_i,B_j,B_k).
\]
The zero sum of the \(B_i\) gives \(S_{21}=-S_3\) and
\(S_{111}=S_3/3\). Expanding \(\det B\) yields
\(d_4(A)=(e_3-3e_{21}+2e_{111})S_3\).
Substitution of the displayed moments gives the coefficient \(c_4\).
\end{proof}

For \(A=J_y\), write its vertex values and mean as \(J_a,\bar J\).
For variations \(v_h,w_h\), use \(V_a,\bar V\) and \(W_a,\bar W\)
for their parent-gradient vertex values and means.
Derivatives of \(d_4(J_y)\) below act on the composed nodal map
\(y_h\mapsto d_4(\nabla_\xi y_h)\).

\begin{corollary}[Polarized cubic tangent]
\label{cor:polarized-defect}
For quadratic variations,
\begin{equation}
D^2d_4(J_y)[v_h,w_h]
=c_4\sum_{a=0}^3D^2\det(J_a-\bar J)
[V_a-\bar V,W_a-\bar W].
\label{eq:polarized-defect}
\end{equation}
The formula also holds at singular centered matrices.
\end{corollary}
\begin{proof}
The vertex gradients and their means depend linearly on nodal values.
Differentiate \eqref{eq:cubic-defect}; the determinant is polynomial
on all matrices.
\end{proof}

\subsection{Exact value and force with a nonzero tangent defect}

\begin{proposition}[A second-variation defect invisible to first-order checks]
\label{prop:invisible-first-order-defect}
If all centered vertex Jacobians of a quadratic map have rank at most
one, then its determinant cubature defect and its first variation vanish.
Its second variation can be nonzero.
\end{proposition}

\begin{proof}
Rank at most one gives zero determinant and zero cofactor for every
centered vertex Jacobian. The defect formula and its first derivative
therefore vanish.
For an explicit example, take parent coordinates
$\xi=(\xi_1,\xi_2,\xi_3)$, fix a scalar $\alpha>0$, and define
\[
y(\xi)=(\xi_1+\alpha\xi_1^2,\xi_2,\xi_3),\quad
v(\xi)=(0,\xi_1\xi_2,0),\quad
w(\xi)=(0,0,\xi_1\xi_3).
\]
For $\alpha>0$ the Jacobian determinant is $1+2\alpha\xi_1>0$.
The centered Jacobians have only one nonzero row. Direct substitution
in \eqref{eq:polarized-defect} gives
\[
D^2d_4(J_y)[v,w]
=2\alpha c_4\left[(3/4)^3+3(-1/4)^3\right]
=3\alpha c_4/4,
\]
which is nonzero.
\end{proof}

A one-component quadratic perturbation of an affine map has this
rank-one centered structure. The example explains why exact volume
and force observations do not certify the tangent quadrature.
For a mesh, let $e$ index its cells and write $J_{y,e}$ for the current
parent-coordinate Jacobian on cell $e$.
Under a complete clamp, a uniform shift of the prescribed linear
pressure bias contributes only a null-Lagrangian constant to the
exactly integrated energy. In four-point integration its tangent
contribution is that shift times $\sum_eD^2d_4(J_{y,e})$.
An exactly compensated reference form can retain that term in the declared
discrete equation \citep{liu2026twominor}. Replacing the scalar pressure
volume by exact integration defines a different discretization.
\Cref{sec:pressure-tangent} gives exact-volume restoration;
\cref{tab:pressure-routes} compares it with compensated reference forms.

The affine-annihilation result and its complete second-jet
factorization are proved in \cref{sec:tet10-quotient}, alongside the
supporting pressure-tangent analysis.

\subsection{The assembled volume maps}
\label{sec:pressure-volume}

Use a conforming Tet10 mesh with positive reference Jacobians and a
complete outer clamp. Let \(X,y\) denote the assembled reference and
current nodal configurations, and \(X_e,y_e\) their quadratic cell maps
from \(\That\), with parent gradients \(J_{X,e},J_{y,e}\).
For $r\in\{Q,I\}$, define
\[
V_e^Q(y)=\Qfour[\det J_{y,e}],\qquad
V_e^I(y)=\int_{\That}\det J_{y,e}\,d\xi,
\qquad V_{e0}^r=V_e^r(X).
\]
The vector of cell-volume changes is denoted here by
\begin{equation}
g_e^r(y)=V_e^r(y)-V_{e0}^r,\qquad
g^r(y)=q_h^r(y-X).
\label{eq:tet10-volume-notation}
\end{equation}
Derivatives act on free nodal coordinates. The boundary-volume identity
from \cref{sec:discrete-volume-setting} gives
\begin{equation}
\sum_e g_e^I(y)=0.
\label{eq:exact-volume-identity}
\end{equation}
This holds throughout the admissible connected neighborhood of $X$.
By \Cref{thm:cubic-defect}, the raw aggregate instead satisfies
\begin{equation}
\sum_e g_e^Q(y)
=\sum_e\bigl[d_4(J_{y,e})-d_4(J_{X,e})\bigr].
\label{eq:aggregate-volume-defect}
\end{equation}

For the mechanical example, use unit shear modulus and the same
four-point isochoric energy for both volume maps. With
$\bar I_1(F)=(\det F)^{-2/3}F{:}F$, set
\begin{equation}
E_{\rm iso}^Q(y)=
\frac12\sum_e\Qfour[
(\bar I_1(J_{y,e}J_{X,e}^{-1})-3)\det J_{X,e}].
\label{eq:template-isochoric-energy}
\end{equation}
These comparisons retain every cell-average equation and require no
prescribed pressure bias.

\subsection{A necessary second-order compatibility condition}

The feasible tangent cone consists of limits of scaled feasible
displacements. Taylor expansion compares it with the Jacobian kernel.

\begin{lemma}[Compatibility of a redundant constraint]
\label{lem:constraint-compatibility}
Let $N,m$ be positive integers and let $g:\R^N\to\R^m$ be twice
continuously differentiable near
$x_0$, with $g(x_0)=0$. Put $B=Dg(x_0)$ and let
$a\in\R^m$ satisfy $a^TB=0$. If a differentiable feasible curve
$x(t)$ has $x(0)=x_0$, $g(x(t))=0$, and derivative $v=x'(0)$,
then
\[
Bv=0,\qquad v^TD^2(a^Tg)(x_0)v=0.
\]
The same conditions hold for any limit
$v=\lim_j(x_j-x_0)/t_j$ with $g(x_j)=0$ and $t_j\downarrow0$.
\end{lemma}

\begin{proof}
Write $x(t)-x_0=tv+o(t)$. First-order Taylor expansion gives $Bv=0$.
The scalar function $a^Tg$ has zero value and derivative at $x_0$.
Its second-order expansion is therefore
\[
0=a^Tg(x(t))
=\tfrac12t^2v^TD^2(a^Tg)(x_0)v+o(t^2).
\]
Division by $t^2$ proves the assertion. The same expansion with
$x_j-x_0=t_jv+o(t_j)$ proves the sequential statement.
\end{proof}

A linearized feasible direction can therefore fail nonlinear compatibility.

\subsection{An explicit curved template}

Divide the unit cube into $2^3$ subcubes and each subcube into the six
Freudenthal tetrahedra. Let $e_i$, $i=1,2,3$, denote the Cartesian
unit vectors. More explicitly, for each permutation
$(i,j,k)$ of $(1,2,3)$ and subcube lower corner $b$, take the vertices
$b$, $b+e_i/2$, $b+(e_i+e_j)/2$, and
$b+(e_1+e_2+e_3)/2$, orienting each cell positively.
Add all edge midpoints to obtain the quadratic nodes.
At each uncurved nodal coordinate $s=(s_1,s_2,s_3)$ prescribe
\begin{equation}
X(s)=s+e_3\prod_{i=1}^3s_i(1-s_i),
\label{eq:dyadic-template}
\end{equation}
and interpolate quadratically within every tetrahedron.
Only the third coordinate is curved. The outer boundary remains fixed.
There are $48$ cells, $125$ nodes, and $81$ free displacement coordinates.

For this fixed template, write $B=B_h=Dg^Q(X)$ and
$C_0=D^2(\sum_e g_e^Q)(X)$. The unit-shear pointwise reference form at
$F=I_3$ is
\[
Q_{I_3}(H,L)=H{:}L+\frac13(\tr H)(\tr L).
\]
For nodal variations $v,w$, put $H_v=J_vJ_{X,e}^{-1}$ and
$H_w=J_wJ_{X,e}^{-1}$ in cell $e$. Define
\begin{equation}
G_{0,e}(v,w)=\Qfour\!\left[
 Q_{I_3}(H_v,H_w)\det J_{X,e}\right],
\qquad G_0=\sum_eG_{0,e}
\label{eq:template-gauge}
\end{equation}
after injection of the free nodal coordinates. For a bulk modulus
$\kappa\ge0$, also set
\[
G_\kappa=G_0+
\kappa B^T\operatorname{diag}\bigl((V_{e0}^Q)^{-1}\bigr)B.
\]
Define \(z\) by its three nonzero nodal values:
\begin{equation}
\begin{array}{c|c}
\text{uncurved node }s & z(s)\\ \hline
(1/2,1/4,1/4)&(0,-2,-527/256)\\
(3/4,1/4,1/4)&(-1,-1,-1)\\
(3/4,1/2,1/4)&(2/3,0,2533/3840).
\end{array}
\label{eq:sparse-volume-witness}
\end{equation}
All remaining nodal values are zero.

\begin{lemma}[Exact template properties]
\label{lem:pressure-template}
The template \eqref{eq:dyadic-template} is one-to-one, with
$\det J_{X,e}\ge119/1024$ on every parent cell.
Its cell volumes and their first derivatives agree for the two rules:
$V_e^Q(X)=V_e^I(X)$ and $Dg_e^Q(X)=Dg_e^I(X)$.
Let $\mathbf1\in\R^{48}$ have every entry equal to one. Moreover,
\begin{equation}
\rank B=47,\quad \mathbf1^TB=0,\quad
Bz=0,\quad z^TC_0z=-c_4/4,\quad
-\tfrac13G_0\preceq C_0\preceq\tfrac13G_0 .
\label{eq:template-facts}
\end{equation}
The gauge energy of $z$ is below $229$.
\end{lemma}

\begin{proof}
The first two rows of each Jacobian are constant, so its determinant
and cofactor are affine. Vertex values therefore bound the determinant
throughout the cell. The derivative of the vertical coordinate with
respect to $s_3$ is positive; the first two coordinates and the outer
trace are unchanged. This proves global one-to-one correspondence.
Affine determinant and cofactor fields make the volume and its first
variation polynomials of degree at most one and two, respectively.
Both are integrated exactly by $Q_4$. Equation
\eqref{eq:exact-volume-identity} gives $\mathbf1^TB=0$.

The remaining finite rational calculations, including the stated
Jacobian minimum, are specified in \Cref{app:pressure-certificate}.
Exact row elimination gives $47$ pivots. Substitution of
\eqref{eq:sparse-volume-witness} gives all $48$ volume actions zero
and $z^TC_0z/c_4=-1/4$. Rational element bounds give
$z^TG_0z<229$ and a cellwise bound
$|D^2d_4(v,v)|\le r_*G_{0,e}(v,v)$ with
$r_*=99672223/301989888<1/3$.
Summation proves the last inequality.

The form $G_0$ is positive definite on the clamped space. Indeed,
$G_0(v,v)=0$ and $Q_{I_3}(H,H)\ge\|H\|_F^2$ force the physical gradient
of $v$ to vanish at all four quadrature points. Its parent gradient is
affine, and those points are affinely independent, so the gradient
vanishes throughout each cell. Continuity and the outer clamp then give
$v=0$. Direct differentiation of the unit-shear isochoric density
gives $Q_{I_3}-D^2\det(I_3)$; equivalently this is
\eqref{eq:nh-completion} at $\mu=1$.
The determinant pullback and the exact boundary-volume identity therefore give
\begin{equation}
D^2E_{\rm iso}^Q(X)=G_0-C_0
 \succeq\frac23G_0\succ0.
\label{eq:template-isochoric-hessian}
\end{equation}
\end{proof}

\begin{theorem}[Loss and restoration of volume-constraint regularity]
\label{thm:volume-regularity}
For the template, $z$ belongs to the linearized four-point volume
kernel and lies outside the tangent cone of the feasible set defined
by all $48$ nonlinear four-point volume constraints.
Along $y(t)=X+tz$, the Jacobian $Dg^Q(y(t))$ has rank $48$ for all
sufficiently small nonzero $t$, and its smallest singular value is
$\Theta(|t|)$.
The exact-volume map $g^I$ has constant rank $47$ in a neighborhood
of $X$; its local feasible set is a smooth manifold of dimension $34$.
Every vector in $\ker B$, including $z$, is tangent to a smooth curve
in that exact-volume feasible set.
\end{theorem}

\begin{proof}
The failure of feasibility follows from
\Cref{lem:constraint-compatibility} with $a=\mathbf1$ and
\eqref{eq:template-facts}. Positivity and injectivity persist for
sufficiently small $t$.
For the rank assertion, apply constant orthogonal changes of row and
column coordinates to $B$. Choose the last row proportional to
$\mathbf1^T$ and the last $34$ columns as a basis of $\ker B$.
The resulting matrix at $t=0$ is
$\left[\begin{smallmatrix}A&0\\0&0\end{smallmatrix}\right]$
with $A$ an invertible $47$-by-$47$ matrix.
At small $t$, bounded row elimination removes the last row's first
$47$ entries. Its remaining row is $td+O(t^2)$ for a fixed
row vector $d\in\R^{34}$.
This row is nonzero to first order because its action on the
coordinates of $z$ is proportional to $z^TC_0z\ne0$.
Elimination of the upper-right block by a bounded column operation
gives an invertible $47$-by-$47$ block and the row $td+O(t^2)$.
The smallest singular value is therefore bounded above and below
by positive constants times $|t|$.

For $g^I$, \eqref{eq:exact-volume-identity} bounds the rank by $47$
throughout the neighborhood. Its rank at $X$ equals $47$.
Continuity preserves a nonzero $47$-by-$47$ minor, so its rank
is constantly $47$ locally. The constant-rank theorem yields the
$81-47=34$ dimensional manifold and the tangent-curve assertion.
\end{proof}

\begin{corollary}[Effect of removing the mean-pressure equation]
\label{cor:mean-pressure-projection}
Fix a weight vector $a\in\R^{48}$ with $a^T\mathbf1\ne0$.
Let $P=I_{48}-aa^T/(a^Ta)$ be the orthogonal projection onto
the pressure space $a^\perp=\{p\in\R^{48}:a^Tp=0\}$, where $I_{48}$ is
the $48$-by-$48$ identity matrix.
Equal weights give zero coefficient mean; choosing $a_e=V_{e0}^Q$
gives zero reference-volume-weighted pressure mean.
The projected map $Pg^Q$, viewed as a map into that $47$-dimensional
space, has full rank near $X$. Its local feasible manifold admits
a smooth curve with derivative $z$. Along every such curve,
\[
\sum_e g_e^Q(y(t))=-\frac{c_4}{8}t^2+o(t^2).
\]
\end{corollary}

\begin{proof}
The range of $B$ is $\mathbf1^\perp$, while $\ker P=\spann\{a\}$.
These spaces intersect only at zero since $a^T\mathbf1\ne0$.
Hence $PB$ has rank $47$ and the same kernel as $B$.
The implicit-function theorem supplies the manifold and the
tangent curve. The aggregate volume has zero first derivative
and second variation $z^TC_0z=-c_4/4$, so its Taylor expansion
gives the displayed formula.
\end{proof}

Mean-pressure projection changes the raw constraint set. The displayed
aggregate is a quadrature volume change; exact total volume stays fixed.

\begin{figure}[tbp]
\centering
\includegraphics[width=\linewidth]{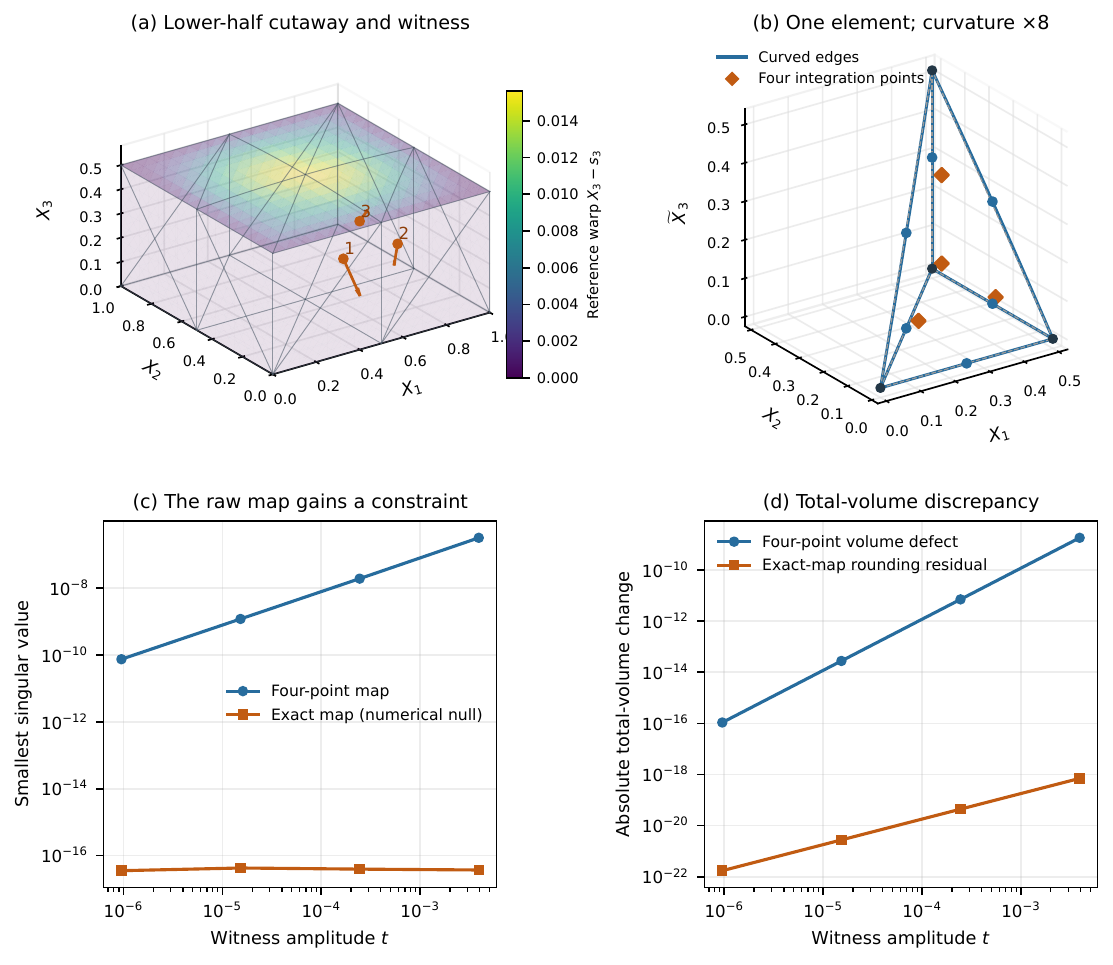}
\caption{Geometry and volume-map response of the template.
(a) A lower-half cutaway at actual geometric scale; orange nodes support
the witness $z$, with arrows scaled by $0.055$.
(b) One element and its four integration points, displayed as
$\widetilde X=X_{\rm aff}+8(X-X_{\rm aff})$, where $X_{\rm aff}$
is the map determined by the corner nodes.
(c,d) Smallest volume-Jacobian singular values and absolute total-volume
changes along the prescribed displacement $u=tz$.
The exact map's numerical null singular values and nonzero total-volume
changes are rounding-scale observations. All quantities are dimensionless;
the path is prescribed and carries no equilibrium assertion.}
\label{fig:pressure-volume-geometry}
\end{figure}
\FloatBarrier

\subsection{Residual-to-displacement error and small-load response}
\label{sec:volume-error-bound}

Full-rank implicit-function arguments give local linear error bounds
\citep{dontchev2014implicit}. The raw volume map violates such a bound.

Let $u\in\R^{81}$ denote the free nodal displacement from $X$.
Choose a sufficiently small closed ball $\mathcal N$ about zero
whose represented deformations are admissible, and define
\[
\mathcal F^r=\{u\in\mathcal N:g^r(X+u)=0\},
\qquad r\in\{Q,I\}.
\]
For a fixed positive definite matrix $M$, set
$\|u\|_M=(u^TMu)^{1/2}$ and define
$\operatorname{dist}_M(u,\mathcal F)=
\inf_{w\in\mathcal F}\|u-w\|_M$.

\begin{theorem}[Failure of a linear volume-residual error bound]
\label{thm:volume-error-bound}
Let $v\ne0$ satisfy $Bv=0$ and $v^TC_0v\ne0$.
For the template, as $t\downarrow0$,
\begin{equation}
\begin{aligned}
\|g^Q(X+tv)\|_2&=\Theta(t^2),\\
\operatorname{dist}_M(tv,\mathcal F^Q)&=\Theta(t),\\
\operatorname{dist}_M(tv,\mathcal F^I)&=O(t^2).
\end{aligned}
\label{eq:projection-orders}
\end{equation}
Consequently no constants $L>0$ and $\theta>1/2$ give the local
bound
$\operatorname{dist}_M(u,\mathcal F^Q)
\le L\|g^Q(X+u)\|_2^\theta$ for all sufficiently small $u$.
The exact-volume map does admit a local linear bound:
\[
\operatorname{dist}_M(u,\mathcal F^I)
\le L_I\|g^I(X+u)\|_2
\]
for some $L_I>0$ and all sufficiently small $u$.
Both the printed direction $z$ and the supplementary dense rational
witness $w$, described with its certificate in
\Cref{app:pressure-certificate}, satisfy the hypotheses.
\end{theorem}

\begin{proof}
The volume maps are cubic polynomials in the nodal coordinates.
Since $Bv=0$, Taylor expansion gives $g^Q(X+tv)=O(t^2)$.
Its aggregate has leading term
$t^2v^TC_0v/2\ne0$, giving the matching lower bound.

Suppose the raw distance has no positive lower bound proportional
to $t$. Then there are $t_j\downarrow0$ and $u_j\in\mathcal F^Q$
with $u_j/t_j\to v$. The sequential conclusion of
\Cref{lem:constraint-compatibility} would imply $v^TC_0v=0$,
a contradiction. The feasible point zero supplies the upper bound
$\operatorname{dist}_M(tv,\mathcal F^Q)\le t\|v\|_M$.
For the exact map, the smooth feasible curve from
\Cref{thm:volume-regularity} has displacement $tv+O(t^2)$,
which proves the third assertion. The claimed raw error bound
would require a quantity of order $t$ to be $O(t^{2\theta})$,
which is impossible for $\theta>1/2$.

For the exact linear bound, any $47$ rows of $B$ are independent:
the sole row relation is their sum.
Complete the first $47$ exact volume coordinates by $34$ linear
coordinates on $\ker B$ to obtain an invertible derivative.
The inverse-function theorem gives a locally Lipschitz inverse
coordinate map. Holding the last $34$ coordinates fixed and
setting the first $47$ to zero changes $u$ by at most a constant
times their residual norm. The final volume constraint also
vanishes by \eqref{eq:exact-volume-identity}.
Equivalence of finite-dimensional norms gives the stated bound.
\end{proof}

The constants concern this fixed template; a full-space square-root
upper bound is not asserted.

\begin{corollary}[A first-order small-load discrepancy]
\label{cor:small-load-discrepancy}
Let $E$ be a $C^3$ displacement energy near zero, with
$E(0)=0$, $DE(0)=0$, and $M=D^2E(0)\succ0$.
Take the dead load $tMv$, where $v$ satisfies the hypotheses of
\Cref{thm:volume-error-bound}.
For sufficiently small positive $t$, the exact-volume constrained
energy $E(u)-t(Mv)^Tu$ has a local minimizer
$u_I(t)=tv+O(t^2)$.
The raw constrained energy also has a minimizer in a sufficiently
small closed neighborhood, with displacement $O(t)$.
Every raw-volume feasible displacement, including any nearby
minimizer of the same loaded energy, differs from $u_I(t)$ in $M$-norm by
at least $ct$ for some $c>0$ independent of $t$.
No prescribed uniform pressure bias is required.
\end{corollary}

\begin{proof}
Use $47$ independent exact constraints, and write $B_{47}$ for
the first $47$ rows of $B$. At zero load, the
linearization of their stationarity system is
\[
\begin{bmatrix}M&B_{47}^T\\B_{47}&0\end{bmatrix}.
\]
For a displacement unknown $a$ and a pressure multiplier, this
matrix's homogeneous system gives $B_{47}a=0$ and then
$a^TMa=0$, so $a=0$; full row rank gives
the zero multiplier. Thus the block matrix is invertible.
The implicit-function theorem yields a
smooth stationary branch. Its derivative at zero solves the
same system with right-hand side $(Mv,0)$ and is $(v,0)$.
Positive definiteness on the constraint tangent persists,
giving a strict local minimizer.
After shrinking $\mathcal N$, Taylor expansion gives
$E(u)\ge c_E\|u\|_2^2$ for some $c_E>0$.
For small $t$, the loaded energy is positive on its boundary,
while zero is raw feasible with energy zero.
Compactness supplies an interior raw minimizer $u_Q$.
Its energy is at most zero, so
$c_E\|u_Q\|_2^2\le t\|Mv\|_2\|u_Q\|_2$ and $u_Q=O(t)$.
Finally, distance to a closed set is Lipschitz in the chosen
norm, so \eqref{eq:projection-orders} gives
$\operatorname{dist}_M(u_I(t),\mathcal F^Q)
\ge ct-O(t^2)$.
Shrinking the load interval completes the assertion.
\end{proof}

The four-point neo-Hookean energy satisfies the corollary by
\eqref{eq:template-isochoric-hessian}. Its constraints preserve cell
averages rather than imposing \(J=1\) pointwise. The finite-penalty
response follows from \cref{thm:finite-penalty-limit}, with quadratic
reduced compatibility and a quartic limiting energy.

\subsection{The quartic penalty consequence}
\label{sec:tet10-quartic-specialization}

Use the hypotheses and notation of \cref{thm:finite-penalty-limit}:
\(B_q=Dq(0)\), \(K=\ker B_q\), \(M=D^2E(0)\succ0\), fixed positive
residual weight \(\mathsf V\), and regular-coordinate inverse
\(\Phi(v,0)=v+\psi(v,0)\).
For the one-relation, second-order case, suppose \(q,E\) are
\(C^3\), \(\rank B_q=m-1\), and choose
$a\in\ker B_q^T\setminus\{0\}$. Define
\[
 d_a=a^T\mathsf V a>0,\qquad C_q=D^2(a^Tq)(0).
\]
Since $a^TB_q=0$ and $\psi(v,0)=O(\|v\|_2^2)$, Taylor expansion gives
\[
 a^Tq(\Phi(v,0))=\tfrac12v^TC_qv+O(\|v\|_2^3).
\]
With $S=a$, the reduced polynomial is $P_2(v)=v^TC_qv/2$ and
$W_S=d_a$. Thus the energy
\begin{equation}
 J_{\varepsilon,\beta}(u)
 =E(u)-\varepsilon\ell^Tu
       +\frac{\beta}{2\varepsilon^2}q(u)^T\mathsf V^{-1}q(u)
 \label{eq:finite-penalty-energy}
\end{equation}
has the localized limit
\begin{equation}
 J_\beta^0(v)
 =\frac12v^TMv-\ell^Tv
       +\frac{\beta}{8d_a}(v^TC_qv)^2,
 \qquad v\in K.
 \label{eq:quartic-volume-limit}
\end{equation}
The displacement and residual orders are $O(\varepsilon)$ and
$O(\varepsilon^2)$, with the minimum-value and cluster-point conclusions
of \cref{thm:finite-penalty-limit}. These statements also hold when
$C_q$ vanishes on $K$, in which case the displayed quartic term is zero.

For the Tet10 template of \cref{lem:pressure-template}, take
$q=q_h^Q$, $a=\mathbf1$, and
$\mathsf V=\operatorname{diag}(V_{e0})$, where $V_{e0}$ are its common
raw and exact reference volumes. Then $C_q=C_0$, the template's aggregate
second-variation matrix. The exact-volume map has the same resting
Jacobian and identically zero aggregate volume change. Any
$v\in K$ with $v^TC_0v\ne0$, including the certified sparse witness,
therefore gives a first-order normalized separation under the joint
scaling when $\ell=Mv$. The radial derivative of the raw quartic energy
at $v$ is $\beta(v^TC_0v)^2/(2d_a)>0$.

\begin{corollary}[Normalized Tet10 finite-bulk response]
\label{cor:finite-penalty-mode}
On $K$, let $v_*$ have unit $M$-norm and be a generalized eigenvector
of $(C_q/\sqrt{d_a},M)$ for a nonzero eigenvalue $\lambda$ of maximum
absolute value. The restricted eigenvector equation means
$w^TC_qv_*=\lambda\sqrt{d_a}\,w^TMv_*$ for every $w\in K$.
Set $\ell=Mv_*$ and $\beta=2c/\lambda^2$, where
$0<c\le1$. The quartic energy \eqref{eq:quartic-volume-limit} has
the unique global minimum $t_cv_*$, where
\[
 t_c+ct_c^3=1,\qquad 0<t_c<1.
\]
For an exact-volume comparison with the same $B_q$ and $M$ and an
identically zero reduced relation, the limiting minimum is $v_*$ and
the limiting relative $M$-distance is $1-t_c$.
\end{corollary}

\begin{proof}
The maximum quadratic Rayleigh magnitude gives $p_*=|\lambda|/2$,
so the normalization agrees with \cref{cor:general-penalty-mode} at
$k=2$. The quartic case also has a direct global certificate.
Put $v_0=t_cv_*$ and
$s_0=\beta(v_0^TC_qv_0)/(2d_a)$. The restricted eigenvector equation
and the scalar equation for $t_c$ give stationarity at $v_0$ on $K$.
For every $v\in K$, direct expansion yields
\begin{align*}
 J_\beta^0(v)-J_\beta^0(v_0)
 &=\tfrac12(v-v_0)^T(M+s_0C_q)(v-v_0)\\
 &\quad+\frac{\beta}{8d_a}
              (v^TC_qv-v_0^TC_qv_0)^2.
\end{align*}
Extremality of $|\lambda|$ gives
$M+s_0C_q\succeq(1-ct_c^2)M$ on $K$. Since $c\le1$ and $t_c<1$,
this bound is strictly positive. The expansion proves the unique global
minimum. The exact-volume limit is quadratic with minimum $v_*$;
unit $M$-norm gives the stated distance.
\end{proof}

At each fixed finite $\kappa$, \eqref{eq:fixed-bulk-response-agreement}
instead gives $\varepsilon v_*+O(\varepsilon^2)$ for both volume maps
under the present $C^3$ hypotheses. The observations in
\cref{sec:finite-penalty-evidence} test both scalings with a positive
quadratic driving energy; the tensor examples use the nonlinear
isochoric energy.

\subsection{Numerical illustrations of the second-order consequences}
\label{sec:tetrahedral-observations}

The \(81\)-coordinate template supplies separate quadratic-energy
illustrations of the second-order consequences.
Using the pressure-exact resting Hessian at unit shear and bulk \(10\)
as a fixed metric \(M\), the hard-volume projection experiment minimizes
\(\|u-tw\|_M^2/2\) for the dense rational witness \(w\).
Its raw normalized distance approaches the cone reference \(0.6121334\);
the exact and mean-projected distances decrease with \(t\).
These are nonlinear constraints with a quadratic driving energy.

The corresponding finite-penalty experiment uses an extreme
\(M\)-normalized mode of \cref{cor:finite-penalty-mode}.
The two prescribed singular scalings approach normalized offsets
\(0.15229240\) and \(0.31767220\); a fixed-bulk control approaches the
common leading response. These observations, their complete parameter
tables and the independent precision checks appear in
\cref{sec:projection-evidence,sec:finite-penalty-evidence}.

The supplementary controls in \cref{app:pressure-controls,app:numerical-details}
also cover geometric variation, the
pressure-bias negative modes, variance stabilization and loaded
physical operators. They retain the complete pressure terms and
distinguish exact sign certificates from numerical eigenvalue counts.

\subsection{Exact-volume restoration}
\label{sec:pressure-tangent}

Exact integration restores the boundary-volume identity and the stated
local constraint regularity. For Tet10 it requires only the four
complete parent Jacobians already evaluated at the material points.

\begin{proposition}[Four-Jacobian volume identity]
\label{prop:four-jacobian-volume}
Let $J_{y,q}$ be the four parent-coordinate Jacobian matrices of a quadratic cell at the
points \eqref{eq:q4-points}, and let
$\bar J_y=\frac14\sum_{q=1}^4J_{y,q}$. Then
\begin{equation}
\int_{\That}\det J_y\,d\xi
=\frac1{24}\sum_{q=1}^4\det J_{y,q}
-\frac{3-\sqrt5}{72}\sum_{q=1}^4\det(J_{y,q}-\bar J_y).
\label{eq:four-jacobian-volume}
\end{equation}
Differentiation of this scalar identity gives the exact volume
gradient and Hessian with respect to every cell nodal coordinate.
\end{proposition}

\begin{proof}
The Jacobian is affine. Its Gauss-point deviations from the mean
are the corresponding vertex deviations divided by $\sqrt5$.
Their determinants differ by the factor $5\sqrt5$.
Substitute this relation into \Cref{thm:cubic-defect} and use
$5\sqrt5\,c_4=(3-\sqrt5)/72$.
Both sides are polynomial functions of the nodal coordinates,
so their derivatives agree identically.
\end{proof}

This identity supplies exact integration for this determinant
functional; general exact tetrahedral cubature is established.
Consistent pressure condensation uses the corrected scalar volume
and its derivatives together. The resulting pressure-exact equation
differs from the raw four-point equation.
An exactly compensated reference construction instead preserves the
raw equation and retains its complete tangent \citep{liu2026twominor}.

The supporting pressure analysis in \cref{app:pressure-controls}
derives the complete pressure Hessian and examines bias and stabilization.
\Cref{tab:pressure-routes} records the distinction between changing
the volume equations and retaining them in a compensated reference form.

\section{Discussion and conclusions}
\label{sec:discussion}
\label{sec:conclusions}

The contribution is the explicit connection between an integration
choice, nonlinear constraint geometry and mechanical response.
The tensor construction proves a sharp supported-defect criterion;
the full-space correction and assembled rank argument quantify its
geometric consequences; the reduced compatibility term determines
the additional leading penalty energy. The complete \(\mathbb Q_3\)
normal form resolves the local exponents, and Tet10 supplies the
second-order counterpart with exact four-Jacobian restoration.

The feasible-distance estimate concerns the prescribed quadrature
equations. Its cell-interior correction preserves exact cell volumes,
including existing volume errors, and its norm is an upper bound
rather than a closest distance. Coefficient and physical-norm control
give mesh-uniform upper constants at fixed order. Sharpness is proved
on each fixed tensor mesh, while admissibility uses separate
maximum-strain neighborhoods. For the Tet10 template, the specified
directions exclude exponents above \(1/2\); a full-space square-root
upper bound remains unproved.

The mechanical theorem concerns localized minima in a joint
small-load, large-bulk limit. Fixed finite bulk gives the common
leading response. The numerical examples use explicit witness loads
and stationary references, distinct from the extremal load in
\cref{cor:general-penalty-mode}. Their larger normalized separations
occur at extreme bulk/shear ratios while absolute displacements
shrink. Pressure recovery amplifies volume error; the mixed residual
and its conditioned displacement effect delimit the interpretation.
In particular, the nearly flat \(\mathbb Q_4\) branch resolves the broad
raw/exact gap better than its fine finite-to-limit correction.
Finite-parameter minimality is not certified by these observations.

Exact-volume restoration, mean-pressure projection, variance
stabilization and compensated reference forms change different parts
of the problem. The first restores the volume identity and the stated
local constant rank; the second removes a raw equation; the third
changes the energy with the raw constraint map retained.
MORTIS's compensated reference forms preserve the original equation
\citep{liu2026twominor}; \cref{tab:pressure-routes} makes this distinction
explicit. The comparison with \citet{foulk2021extending} concerns
their scalar ratio/variance mechanism on the stated Tet10 kinematics,
leaving their composite fields and operators outside the comparison.

The complete pressure tangent retains both its positive gradient
outer term and signed geometric term. The supporting large-bias,
cell-scale-curved construction has negative directions that no
increase of the cell-average bulk penalty removes. Loaded-state
controls have different operators and boundary conditions and
preclude an unconditional positive-tangent claim for exact volume.
These restrictions concern tangent-based remedies, not a change
to the volume constraints used in the main analysis.

The results assume fixed spaces and cubature with the stated support,
coefficient and geometry properties. Arbitrary simplex orders,
rational geometry, mixed-element correction and changing quadrature
remain separate questions. Uniform mixed inf-sup stability and a
simultaneous mesh/load lower limit are also open: the tetrahedral
catalogue has rank \(44\) in its affine control and \(47\) in the
curved cases. Within the proved scope, the constructions distinguish
quadrature feasibility, exact cell-volume preservation and singular
mechanical response without relying on reference-state derivative
agreement alone.

\appendix
\section{Complete pressure tangents and supporting stability analyses}
\label{app:pressure-controls}

This appendix gives the supporting local quotient construction,
complete pressure Hessian and Tet10 stability comparisons.
The pressure-bias and volume-variance results concern the stated
discretizations; exact-volume restoration and compensated reference
forms retain their distinct roles.

\subsection{Affine directions and second-jet coordinates}
\label{sec:tet10-quotient}

The local defect of \cref{sec:cubature} acts only on the quadratic
part of a variation. Here \(J_y\) is the parent gradient of the current
quadratic cell map. For a variation \(v_h\), write
\(V_a=D_\xi v_h(a)\) for its four vertex gradients and
\(\bar V=\frac14\sum_aV_a\); use the same convention for \(w_h\).
The polarized formula \eqref{eq:polarized-defect} gives the following
annihilation and quotient construction.

\begin{corollary}[Affine annihilation]
\label{cor:p1-annihilation}
For \(v_h\in[\mathbb P_1(\That)]^3\) and \(w_h\in[\mathbb P_2(\That)]^3\),
\begin{equation}
D^2d_4(J_y)[v_h,w_h]=0.
\label{eq:p1-annihilation}
\end{equation}
The arguments may be exchanged.
\end{corollary}
\begin{proof}
An affine variation has \(V_a=\bar V\). Every term in
\eqref{eq:polarized-defect} vanishes, and the form is symmetric.
\end{proof}

For vertex and edge-midpoint values \(v_i,v_{ij}\), define
\begin{equation}
\delta_{ij}(v_h)=v_{ij}-\tfrac12(v_i+v_j),\qquad
\delta(v_h)\in\R^{18}.
\label{eq:midside-deviation}
\end{equation}
These six vector deviations determine the quadratic part modulo affine
fields, equivalently its constant vector Hessian. They are the
\emph{second-jet coordinates} used below.

\begin{theorem}[Tet10 second-jet factorization]
\label{thm:quotient-factor}
The deviations induce the isomorphism
\begin{equation}
[\Ptwo(\That)]^3/[\Pone(\That)]^3\cong\R^{18}.
\label{eq:quotient-isomorphism}
\end{equation}
On cell \(e\), a unique symmetric form \(R_e(J_y)\) satisfies
\begin{equation}
D^2d_4(J_y)[v_h,w_h]
=\delta(v_h)^TR_e(J_y)\delta(w_h).
\label{eq:quotient-factorization}
\end{equation}
In vertex/midside-deviation coordinates its matrix is
\begin{equation}
\begin{bmatrix}0&0\\0&R_e\end{bmatrix}.
\label{eq:edge-only-block}
\end{equation}
\end{theorem}
\begin{proof}
A zero deviation means that each midpoint equals its affine endpoint
average, so \(\ker\delta=[\Pone]^3\).
The domain, kernel and target dimensions are \(30,12,18\).
Affine annihilation therefore lets the symmetric defect form descend
uniquely through this quotient, giving the displayed block matrix.
\end{proof}

The classical \(7+5+3+3\) orthogonal decomposition of these
second-derivative coordinates \citep{lazar2016irreducible}
includes a three-dimensional sector with scalar symmetric gradient
\citep{dain2006korn}. The factorization itself uses only affine
annihilation. It localizes the constant-coefficient determinant defect;
variable constitutive and volumetric weights retain their own
quadrature dependence.

\subsection{The complete pressure Hessian}
\label{sec:p0-scope}

For a cellwise-constant pressure (\(P_0\)) discretization, consider
cell \(e\). On that cell, let \(w_q>0\) be the physical reference quadrature
weights, \(V_e^Q=\sum_qw_q\), and \(J_q(u)\) the physical deformation
determinant. Define
\begin{equation}
c_e(u)=\sum_qw_q(J_q(u)-1),\qquad
\pi_e(u)=\frac{\kappa_e}{V_e^Q}c_e(u),\qquad \kappa_e\ge0.
\label{eq:p0-pressure}
\end{equation}
For Tet10, \(w_q=\det J_{X,e}(\xi_q)/24\) and \(c_e(u)=g_e^Q(X+u)\).
The cell-average pressure energy is
\begin{equation}
\mathcal E_{P0,e}(u)=\frac{\kappa_e}{2V_e^Q}c_e(u)^2.
\label{eq:p0-energy}
\end{equation}
For nodal directions \(v,w\), differentiation gives
\begin{align}
Q_{P0,e}[v,w]&=\frac{\kappa_e}{V_e^Q}Dc_e(u)[v]Dc_e(u)[w],
\label{eq:qp0}\\
P_{P0,e}[v,w]&=\pi_e(u)D^2c_e(u)[v,w].
\label{eq:pp0}
\end{align}
Here \(Q_{P0,e}\) is positive semidefinite
of rank at most one, and \(P_{P0,e}\) is signed. For
\(\mathcal E_{P0}=\sum_e\mathcal E_{P0,e}\),
\[
D^2\mathcal E_{P0}
=B_h(u)^T\operatorname{diag}(\kappa_e/V_e^Q)B_h(u)
 +\sum_e\pi_e(u)D^2c_e(u),
\]
where the rows of \(B_h(u)\) are \(Dc_e(u)\).

A quadrature-local penalty \(\sum_qw_q\kappa_q(J_q-1)^2/2\), with
\(\kappa_q\ge0\), instead gives
\begin{align}
Q_{{\rm plain},e}[v,w]
 &=\sum_qw_q\kappa_qDJ_q(u)[v]DJ_q(u)[w],
\label{eq:qplain}\\
P_{{\rm plain},e}[v,w]
 &=\sum_qw_q\kappa_q(J_q(u)-1)D^2J_q(u)[v,w].
\label{eq:pplain}
\end{align}
Each cell contributes its declared penalty; the plain terms vanish if
all cells use \(P_0\). Let \(K_{\rm iso}\) be the complete isochoric
Hessian, and let the four pressure symbols without \(e\) denote their
assemblies. Define
\begin{align}
K_{\rm partial}&=K_{\rm iso}+Q_{\rm plain}+P_{\rm plain},
\label{eq:kfactor}\\
K_{\rm energy}&=K_{\rm partial}+Q_{P0}+P_{P0}.
\label{eq:kenergy}
\end{align}
The first matrix omits the two condensed cell-pressure terms; it is used
in the operator comparison of \cref{sec:retained-operator-scope}.
The chain rule above proves the complete-energy identity
\begin{equation}
K_{\rm energy}-K_{\rm partial}=Q_{P0}+P_{P0}.
\label{eq:scope-identity}
\end{equation}
Its signed term motivates the eigenvalue and stability comparisons below.

\subsection{Eigenvalue counts and reference decompositions}

For a symmetric matrix \(A\), write
\(\operatorname{In}(A)=(n_+(A),n_-(A),n_0(A))\) for its positive,
negative and zero eigenvalue counts. Its negative index is \(n_-(A)\);
its minimum positive-repair rank is
\[
\min\{\operatorname{rank}E:E=E^T,\ A+E\succ0\}.
\]
If \(K_{\rm iso}=G-C_{\rm iso}\) is a constitutive gauge decomposition,
\(G\) is the chosen reference and \(C_{\rm iso}\) its compensating
curvature. Equivalently,
\begin{equation}
C_{\rm partial}=C_{\rm iso}-P_{\rm plain},\qquad
C_{\rm energy}=C_{\rm partial}-P_{P0}.
\label{eq:cenergy}
\end{equation}
Then \(K_{\rm partial}=G+Q_{\rm plain}-C_{\rm partial}\) and
\(K_{\rm energy}=G+Q_{P0}+Q_{\rm plain}-C_{\rm energy}\).
The first-invariant split is derived in
\cref{app:constitutive-identities}.

\subsection{A boundary-null pressure bias in the two volume discretizations}

For the Tet10 maps of \Cref{sec:pressure-volume}, retain the identical
unit-shear isochoric energy \eqref{eq:template-isochoric-energy}.
Let $\kappa\ge0$ be a bulk modulus and let $\eta\in\R$ be a spatially
uniform linear pressure bias. The complete reduced energies are
\begin{equation}
\Pi_{\kappa,\eta}^r(y)=E_{\rm iso}^Q(y)
+\frac{\kappa}{2}\sum_e\frac{(g_e^r(y))^2}{V_{e0}^r}
+\eta\sum_e g_e^r(y),\qquad r\in\{Q,I\}.
\label{eq:pressure-volume-energies}
\end{equation}
For $\kappa>0$, the penalty is stationary elimination of one pressure
from $\pi_eg_e^r-V_{e0}^r\pi_e^2/(2\kappa)$ in each cell.
At $\kappa=0$, the penalty is omitted. Exact pressure integration changes
only $g^r$ and its consistent derivatives; the isochoric quadrature
remains four-point.

Equation \eqref{eq:exact-volume-identity} makes the clamped displacement
energy $\Pi_{\kappa,\eta}^I$ independent of $\eta$, although boundary
reactions may depend on the bias. Equation
\eqref{eq:aggregate-volume-defect} shows how the raw rule acquires an
interior contribution from this boundary-null term. The following
negative-mode construction uses a large fixed pressure-to-shear ratio.
The constraint-regularity and hard-load results in
\cref{thm:volume-regularity,cor:small-load-discrepancy} require no such bias; the joint finite-bulk
limit is separately analyzed in \Cref{thm:finite-penalty-limit}.

\subsection{A volume-order family of quadrature-induced negative modes}

For either volume rule $r\in\{Q,I\}$, denote the complete identity-state
energy Hessian by $K^r=D^2\Pi_{\kappa,\eta}^r(X)$.
On a replicated mesh, $G_\kappa$ denotes the corresponding assembly
of the same cell forms.

\begin{theorem}[Bulk modes unaffected by the bulk penalty]
\label{thm:bulk-quadrature-modes}
For every positive integer $k$, place $k^3$ translated copies of
\eqref{eq:dyadic-template}, scaled by $1/k$, in the unit cube.
Write $h=1/(2k)$ and keep only the complete outer clamp.
At the identity state, for every $\kappa\ge0$ and the fixed bias
$\eta=2^{20}$, the complete four-point Hessian has at least $k^3$
negative directions in the kernel of all linearized cell-volume
constraints. Its negative index and its minimum positive-repair rank
are $\Theta(h^{-3})$.
The pressure-exact Hessian satisfies
$K^I\succeq(2/3)G_\kappa\succ0$ on the same meshes, for every $\eta$.
The identity is a global minimizer of the pressure-exact discrete
energy over its admissible clamped deformations.
\end{theorem}

\begin{proof}
At the template identity, both rules have the same cell-volume
gradients, and the penalty volumes vanish. The complete Hessians are
therefore
\begin{equation}
K^Q=G_\kappa+(\eta-1)C_0,\qquad
K^I=G_\kappa-C_0.
\label{eq:bias-hessians}
\end{equation}
The coefficient $-1$ is the induced isochoric pressure at unit shear.
Since $Bz=0$, the bulk penalty contributes zero in the direction $z$.
Using $c_4>1/1125$ and the rational energy bound in
\cref{app:pressure-certificate} gives
\[
z^TK^Qz
\le \frac{106051042979793387703}{464834915205120000}
-\frac{2^{20}-1}{4500}<-4.
\]
The same inequalities hold for larger positive bias.
Copy $z$ into each scaled template, with zero extension.
Every copied quadratic energy is $1/k$ times its original value.
Their supports have disjoint interiors and they vanish on template
boundaries. Hence cross terms are zero, and they span a negative
space of dimension $k^3$ in the global linearized volume kernel.
The total free dimension is $3(4k-1)^3=O(k^3)$.
The negative index has the asserted order.
Any correction of rank less than $k^3$ has a nonzero kernel
vector in the constructed negative space, so it cannot make
the tangent positive definite.
A sufficiently large positive multiple of the identity gives
a correction of rank at most the full dimension.
These bounds establish the stated repair order.

The scaled templates retain the same bound
$-G_0/3\preceq C_0\preceq G_0/3$, so
$K^I\succeq G_\kappa-G_0/3\succeq(2/3)G_\kappa$.
The positivity of $G_0$ was proved with \Cref{lem:pressure-template};
adding the nonnegative pressure outer term gives $G_\kappa\succ0$.
For the global energy assertion, the arithmetic--geometric mean inequality
for the singular values gives $\bar I_1(F)\ge3$ when $\det F>0$,
the quadrature weights are positive, and the penalty terms are
nonnegative. The uniform exact-volume term vanishes by
\eqref{eq:exact-volume-identity}. The energy is zero at the identity.
\end{proof}

The family is uniformly shape-regular with cell-scale curvature.
For a fixed \(C^2\) mapped geometry, centered parent Jacobians instead
scale as \(O(h^2)\); smooth-mapped interface and reference-form results
therefore have different geometric hypotheses \citep{liu2026twominor}.
A second rational displacement, denoted by $w$, has its full nodal
vector and certificates retained in the reproducibility data and
described in \Cref{app:pressure-certificate}. It strengthens the
negative-energy conclusion to $\eta=2^{17}$;
the numerical observations below use that smaller bias.

The negative directions belong to the finite-penalty energy.
Pointwise penalties can act on them differently. Under exact enforcement
of the raw constraints, those directions can be nonintegrable
linearized motions, so their Lagrangian curvature need not describe
an instability along feasible motion.

\subsection{Comparison with a volume-variance stabilization}
\label{sec:variance-control}

\citet[Section~3.7, Eqs.~(66)--(70)]{foulk2021extending} add a
logarithmic-squared local-to-average Jacobian-ratio penalty to their
five-field composite formulation. Its identity-state expansion penalizes
local volume increments about their mean. We transplant only this scalar mechanism to the present Tet10
kinematics; their projected gradient, auxiliary fields and composite
interpolation are outside the comparison.

For a scalar cell field $q$, define its reference-volume-weighted
quadrature average by
\[
\langle q\rangle_e=
\frac{\Qfour[q\det J_{X,e}]}{V_{e0}^Q}.
\]
Let $J_e(\xi,y)=\det(J_{y,e}J_{X,e}^{-1})$ be the local volume ratio.
For $\alpha\ge0$, define the additional energy
\begin{equation}
\mathcal V_\alpha(y)=\frac{\alpha}{2}
\sum_e V_{e0}^Q
\left\langle\bigl(J_e-\langle J_e\rangle_e\bigr)^2\right\rangle_e .
\label{eq:variance-energy}
\end{equation}
Let $j_e[v]=DJ_e(X)[v]$ denote its first volume-ratio variation
at the identity. Write $S_{\rm var}$ for the Hessian
$D^2\mathcal V_1(X)$ on the free displacement coordinates.

\begin{proposition}[Variance penalty and the mean-volume kernel]
\label{prop:variance-penalty}
The matrix $S_{\rm var}$ is positive semidefinite and
for any free nodal variations $v,w$ satisfies
\begin{align}
v^TS_{\rm var}w
=\sum_eV_{e0}^Q\bigl[
\langle j_e[v]j_e[w]\rangle_e
-\langle j_e[v]\rangle_e\langle j_e[w]\rangle_e
\bigr].
\label{eq:variance-hessian}
\end{align}
The energy $\mathcal V_\kappa$ equals the difference between the
pointwise quadratic volume penalty and the condensed-$P_0$
cell-volume penalty, with the same four-point rule.
The condition $Bv=0$ removes the mean-volume penalty.
The variance contribution $v^TS_{\rm var}v$ can remain positive.
For any fixed finite $\alpha$, the controlled Hessian
$K^Q+\alpha S_{\rm var}$ remains pressure-bias dependent on the
template and becomes indefinite for sufficiently large positive
$\eta$.
\end{proposition}

\begin{proof}
At $X$, every $J_e$ and its mean equal one.
The variance and its first derivative are zero.
Differentiating twice gives \eqref{eq:variance-hessian}.
Positive quadrature weights express its quadratic value as a sum
of squared deviations, proving positive semidefiniteness.
The scalar identity
\[
\langle(J_e-1)^2\rangle_e
=\left\langle(J_e-\langle J_e\rangle_e)^2\right\rangle_e
+(\langle J_e\rangle_e-1)^2
\]
and $g_e^Q=V_{e0}^Q(\langle J_e\rangle_e-1)$ prove the energy
difference statement.
Moreover, $Bv=0$ gives only $\langle j_e[v]\rangle_e=0$;
the remaining second moment can be positive.

Equation \eqref{eq:bias-hessians} gives
$\partial_\eta(K^Q+\alpha S_{\rm var})=C_0\ne0$.
On the printed witness $z$, the quadratic value is
$z^T(G_\kappa+\alpha S_{\rm var})z-(\eta-1)c_4/4$.
Its first term is finite and independent of $\eta$, so the value
tends to minus infinity as $\eta$ increases.
\end{proof}

The tangent-level connection to the published ratio penalty can be made
exact on the present kinematics. Absorb the shear-modulus prefactor into
$\alpha$ and define, for positive quadrature volume ratios,
\[
\mathcal L_\alpha(y)=\frac\alpha2\sum_e V_{e0}^Q
\left\langle\left[\log\left(\frac{\langle J_e\rangle_e}{J_e}\right)\right]^2\right\rangle_e.
\]

\begin{corollary}[Identity-state ratio and variance tangents]
\label{cor:ratio-variance-tangent}
At $y=X$, both $\mathcal L_\alpha$ and $\mathcal V_\alpha$ have
zero value and first variation, and
\[
D^2\mathcal L_\alpha(X)=D^2\mathcal V_\alpha(X)=\alpha S_{\rm var}.
\]
\end{corollary}
\begin{proof}
Every $J_e$ and its weighted mean equal one at $X$. For a variation $v$,
the derivative of the logarithm is
$\langle j_e[v]\rangle_e-j_e[v]$. The logarithm itself is zero, so the
second derivative of its half-square is the product of these first
derivatives. Averaging gives \eqref{eq:variance-hessian}.
\end{proof}

The nonlinear ratio and variance energies generally differ; the local
identity needs no global convexity of the logarithmic penalty.
Variance stabilization changes the energy with \(g^Q\) retained.
Its tangent sign and the raw constraint-rank defect are separate
properties. The comparison leaves the composite formulation and
pressure-dependent stabilization choices open.

\subsection{Volume restoration and compensated reference forms}

Exact-volume restoration changes the scalar constraint and its
consistent derivatives. A compensated reference construction retains
the original equation and reorganizes its complete tangent.
\Cref{tab:pressure-routes} records these distinct uses of the same
cubature defect.

\begin{table}[H]
\centering
\caption{Two uses of the cubature defect. The pressure-volume map and the
complete displacement equation remain explicit in each route.}
\label{tab:pressure-routes}
\begin{tabularx}{\linewidth}{@{}>{\raggedright\arraybackslash}p{.24\linewidth}
>{\raggedright\arraybackslash}p{.29\linewidth}
>{\raggedright\arraybackslash}X@{}}
\toprule
Route & Declared equation & Role in the argument \\
\midrule
Pressure-exact discretization & Uses $g^I$ and its consistent derivatives;
the isochoric rule stays four-point & Restores the volume identity and,
on the constructed family, local constraint regularity. \\
Compensated reference form & Retains $g^Q$ and the complete original
four-point tangent & Uses a positive reference under its stated hypotheses,
with explicit compensation, preserving the original equation
\citep{liu2026twominor}. \\
\bottomrule
\end{tabularx}
\end{table}

The variance control in \cref{sec:variance-control} compares a pointwise
volume penalty with the cell-average penalty while retaining the raw
volume map. Its reference-state tangent agrees with the transplanted
Jacobian-ratio penalty from \citet{foulk2021extending} on the specified
kinematics. The composite method's additional fields and operators are
outside that comparison.

\section{Rational verification of the pressure-volume template}
\label{app:pressure-certificate}

This appendix specifies the finite calculations used in
\Cref{lem:pressure-template,thm:bulk-quadrature-modes}. The geometry, shape functions, moment
formulas, witness, and bounds determine the certificate independently
of an eigensolver. The ancillary files \texttt{pressure-dense.json} and
\texttt{pressure-sparse.json}, retained by the authors, record the exact
witnesses and their rational bounds. The verification script
\texttt{pressure\_certificate.py} reconstructs the printed sparse witness,
volume matrix, pivot columns and geometric bounds. These materials are
available upon reasonable request and are planned for public release in a
subsequent version.

Let $a=0,1,2,3$ index the parent vertices. For one cell, write
$J_a=J_X(a)$ and $\bar J=\frac14\sum_aJ_a$.
For a nodal variation $v$, write $W_a=J_v(a)$ and
$\bar W=\frac14\sum_aW_a$.
An overbar on any other parent gradient denotes the same mean of
its four vertex values.
The quadratic Lagrange basis consists of
$N_i=\lambda_i(2\lambda_i-1)$ on vertices and
$N_{ij}=4\lambda_i\lambda_j$ on edges.
Together with the explicit cells and nodal map
\eqref{eq:dyadic-template}, these formulas determine every matrix entry.
All reference nodal coordinates have denominator dividing $4096$.

\subsection{Volume, gradient, and curvature}

Since $\det J_X$ and $\cof J_X$ are affine on this template,
the two rules agree on the reference volume and its derivative.
Suppress their superscripts here by writing
$V_{e0}=V_{e0}^Q=V_{e0}^I$ and $Dg_e(X)=Dg_e^Q(X)=Dg_e^I(X)$.
They satisfy
\begin{align}
V_{e0}
&=\frac1{24}\sum_a\det J_a,\\
Dg_e(X)[v]
&=\frac1{120}\left[
\sum_a\cof J_a:W_a+
\left(\sum_a\cof J_a\right):\left(\sum_aW_a\right)
\right].
\label{eq:certificate-volume-gradient}
\end{align}
The second formula follows from
$\int_{\That}\lambda_a\lambda_b=(1+\delta_{ab})/120$,
where $\delta_{ab}$ is the Kronecker symbol.
Substitution gives total volume one, minimum parent Jacobian determinant
$119/1024$, and zero aggregate derivative on every free coordinate.
Construct the $48$-by-$81$ derivative matrix from
\eqref{eq:certificate-volume-gradient}.
Exact rational row reduction has $47$ pivots; its final row is zero.
The retained rank certificate records the pivot columns in the
nodal ordering, with each node's three components consecutive.

For two variations $v,w$, the aggregate four-point curvature is
\[
C_0(v,w)
=c_4\sum_{e,a}
D^2\det(J_a-\bar J)[J_v(a)-\overline{J_v},
                              J_w(a)-\overline{J_w}].
\]
For a single variation this becomes
\begin{equation}
\frac{C_0(v,v)}{c_4}
=2\sum_{e,a}(J_a-\bar J):\cof(W_a-\bar W).
\label{eq:certificate-curvature}
\end{equation}
The exactly integrated aggregate curvature has canceled under the
complete clamp.
Inserting the three nodal values in
\eqref{eq:sparse-volume-witness} into
\eqref{eq:certificate-volume-gradient} gives zero in all $48$ cells.
Inserting them into \eqref{eq:certificate-curvature} gives $-1/4$.

\subsection{Rational reference-form and negative-energy bounds}

Define the following cell constants from the four vertex Jacobians:
\[
d_e=\min_a\det J_a,\quad
Q_e=\max_a\|J_a\|_F^2,\quad
M_e=\max_a\|J_a-\bar J\|_1,\quad
Z_e=\max_a\|\cof J_a\|_F^2,
\]
where $\|A\|_1=\sum_{i,j}|A_{ij}|$ in this appendix.
The determinant is affine and the squared norms are convex on
affine fields, so these constants bound the entire parent cell.
The form $G_{0,e}$ is the quadrature-defined unit-shear reference form
in \eqref{eq:template-gauge}. At the identity its pointwise density is
$Q_{I_3}(H,H)=\|H\|_F^2+(\tr H)^2/3$.
It satisfies
\[
\|H\|_F^2\le Q_{I_3}(H,H)\le2\|H\|_F^2.
\]
Consequently, with
$I_e(v)=\int_{\That}\|J_v\|_F^2$,
\begin{equation}
\frac{d_e}{Q_e}I_e(v)
\le G_{0,e}(v,v)
\le\frac{2Z_e}{d_e}I_e(v),\qquad
I_e(v)=\frac{\sum_a\|W_a\|_F^2+\|\sum_aW_a\|_F^2}{120}.
\label{eq:certificate-gauge-bounds}
\end{equation}
Positive quadrature weights justify these inequalities for the
quadrature-defined gauge.

The determinant derivative estimate
$|D^2\det(A)[B,C]|\le2\|A\|_F\|B\|_F\|C\|_F$,
obtained from the column-cross-product formula, yields
$|D^2d_4(v,v)|\le240c_4M_e I_e(v)$.
The elementary rational inequalities
$11/5<\sqrt5<9/4$ give
$1/1125<c_4<7/7200$ and $240c_4<7/30$.
Define the coefficient-one local defect norm by
$\delta_e=\sup_{G_{0,e}(v,v)>0}|D^2d_4(v,v)|/G_{0,e}(v,v)$.
It is bounded by
\[
\delta_e\le\frac7{30}\frac{M_eQ_e}{d_e},\qquad
\max_e\frac7{30}\frac{M_eQ_e}{d_e}
=\frac{99672223}{301989888}<\frac13 .
\]
For the printed sparse witness, summing the upper bounds
\eqref{eq:certificate-gauge-bounds} gives
\[
U_z=\frac{106051042979793387703}{464834915205120000}<229.
\]
Hence
$z^TK^Qz\le U_z-(2^{20}-1)/4500<-4$ for every $\kappa\ge0$.

The supplementary dense rational witness $w$ uses the same template
and exactly satisfies $Bw=0$.
Its certificate verifies
\[
-0.563<C_0(w,w)/c_4<-0.562,\qquad
w^TG_0w<18,\qquad
w^TK^Qw<-47\quad\text{at }\eta=2^{17}.
\]
A floating-point constrained eigendirection proposed $w$.
Its free coordinates were rationalized and its remaining coordinates
were recovered by exact row reduction. The final verification
reconstructs the rational geometry, constraint actions, and scalar
inequalities from the stored vector without repeating that proposal.
The short three-node witness is independently sufficient for the
stated volume-order theorem.

\section{A reference for the limiting displacement projection}
\label{app:projection-reference}

Use the fixed 48-cell template of \Cref{lem:pressure-template}, with its
Jacobian $B\in\R^{48\times81}$ and aggregate curvature $C_0$.
Let $M=M^T\succ0$ be a fixed $81$-by-$81$ metric, with
$\|u\|_M=(u^TMu)^{1/2}$, and retain the raw feasible set
$\mathcal F^Q$ from \Cref{sec:volume-error-bound}.
The quadratic compatibility condition defines the closed cone
\[
\mathcal K=\{u\in\ker B:u^TC_0u=0\}.
\]
Fix a direction $v\in\ker B$ with $v^TC_0v<0$.
Let $T\in\R^{81\times34}$ have range $\ker B$ and
satisfy $T^TMT=I_{34}$, where $I_{34}$ is the identity of order $34$. Define
$D=T^TC_0T/c_4$ and $a=T^TMv$.
The factor $c_4>0$ only scales the cone equation.

\begin{proposition}[Regular cone-projection reference]
\label{prop:cone-projection-reference}
Suppose there is a scalar $\lambda$ for which
$I_{34}+\lambda D\succ0$ and
\[
s_\lambda=(I_{34}+\lambda D)^{-1}a,\qquad
s_\lambda^TDs_\lambda=0.
\]
Then $Ts_\lambda$ is the unique closest point to $v$ in
$\mathcal K$ under the $M$ norm. Moreover,
\[
\lim_{t\downarrow0}
\frac{\operatorname{dist}_M(tv,\mathcal F^Q)}{t}
=\|Ts_\lambda-v\|_M .
\]
\end{proposition}

\begin{proof}
Writing a cone point as $Ts$ with $s\in\R^{34}$, its squared
distance from $v$ is $\|s-a\|_2^2$ because $v$ lies in the range
of $T$. On the cone, the objective $\|s-a\|_2^2/2$ equals the Lagrangian
$\|s-a\|_2^2/2+\lambda s^TDs/2$.
Its positive definite Hessian gives the unique global minimizer
$s_\lambda$, proving the first assertion.
Also $Ds_\lambda\ne0$, because otherwise $a=s_\lambda$ would
contradict $v^TC_0v<0$.

For nonzero $t$, write the raw constraints at displacement $tu$
as their first $47$ equations divided by $t$, together with
their aggregate divided by $c_4t^2$.
The cubic volume formulas extend this system smoothly to $t=0$.
Its limiting equations are
$B_{47}u=0$ and $u^TC_0u/(2c_4)=0$.
Their derivative at $Ts_\lambda$ has full row rank because
$Ds_\lambda\ne0$.
The implicit-function theorem therefore produces a raw feasible
branch $t(Ts_\lambda+O(t))$. This gives the desired upper limit.

Any sequence of closest raw feasible points has displacements
$O(t)$, since zero is feasible.
After division by $t$, bounded subsequences have limits in
$\mathcal K$ by \Cref{lem:constraint-compatibility}.
Their limiting distance from $v$ is at least its distance to
$\mathcal K$, proving the lower limit.
\end{proof}

If $d_i$ are the eigenvalues of $D$ and $a_i$ are the coordinates
of $a$ in the same orthonormal eigenbasis, the scalar equation is
\[
\sum_i\frac{d_i a_i^2}{(1+\lambda d_i)^2}=0.
\]
The numerical reference uses its interval on which
$I_{34}+\lambda D$ is positive definite.
The observed minimum eigenvalue of that matrix is $0.20005$.
The reference predicts a relative distance $0.6121334$ for the
retained dense direction. These numerical values have no
interval certification. The exact order obstruction in
\Cref{thm:volume-error-bound} is independent of their computation.

\section{Numerical protocols and supplementary controls}
\label{app:numerical-details}

The following details support the observations in
\cref{sec:evidence,sec:tetrahedral-observations}.
The tensor comparisons use the nonlinear isochoric energy and the
raw/exact volume rules stated there. The Tet10 projection and penalty
comparisons use their explicitly defined quadratic driving metric.
All numbers retain their original input and arithmetic definitions.

\subsection{Tensor coefficient and rank observations}
\label{app:tensor-rank-records}

Use the two-cell \(\mathbb Q_3\) and one-cell \(\mathbb Q_4\) models
of \cref{sec:evidence}, with the directions from
\eqref{eq:physical-bubble-coefficient}.
The coefficient \(\chi_n\) in \eqref{eq:gauss-witness-coefficient} has the
following exact values:
\[
\begin{array}{c|c|c}
p&n&\chi_n\\ \hline
3&4&131072/4501875\\
4&5&131072/14586075\\
5&6&524288/196101675\\
5&7&524288/676350675\\
6&8&33554432/152178901875
\end{array}
\]
Independent rational integration agrees with the coefficients for
\(p=4,n=5\) and \(p=6,n=8\); Gauss nine integrates the latter witness exactly.
The anisotropic case with degrees $(3,2,2)$ and
Gauss counts $(4,3,3)$ has no supported cubic obstruction, as predicted by
\cref{thm:tensor-volume-admission}.

A direct $\mathbb Q_5$ calculation on one physical unit cube uses the
boundary-zero sum of the three witness fields at amplitude
$0.002255274489$. Its raw Gauss-six volume change is
$3.833512557\times10^{-12}$; the cubic formula gives
$3.833512197\times10^{-12}$, a difference of $3.60\times10^{-19}$.
Gauss-eight evaluation gives $-1.97\times10^{-18}$ in place of the ideal
zero. This comparison evaluates volume at the prescribed field.

On the two-cell $\mathbb Q_3$ mesh, place the witness directions
$U_1,U_2,U_3$ in the first cell and set
$u_\tau=\tau(U_2+U_3)$. These are feasible states by
\cref{thm:assembled-volume-geometry}. At
\[
\tau=2.410989843\times10^{-3},\quad
1.205494922\times10^{-3},\quad
6.027474608\times10^{-4},
\]
the smallest raw Jacobian singular values are
$4.518206883\times10^{-8}$, $1.129551748\times10^{-8}$, and
$2.823879839\times10^{-9}$, respectively.
The exact-volume values remain approximately $7.28\times10^{-16}$.
The missing-component derivative
$Dq_h^Q(u_\tau)U_1$ differs from
$\gamma_1\tau^2(1,0)^T$ by approximately $1.54\times10^{-16}$ in
Euclidean norm, where $\gamma_1=V_1\chi_4/8$.
The volume residuals of these prescribed feasible states remain at
approximately $10^{-20}$. The quadratic singular-value change is shown in
\cref{fig:tensor-rank-response}.

For physical coordinates $X=(X_1,X_2,X_3)$ and $e_1=(1,0,0)^T$, define
the unnormalized face-flux direction
\[
\widetilde w(X)=\min(2X_1,2-2X_1)\,
       4X_2(1-X_2)\,4X_3(1-X_3)e_1.
\]
An additional dyadic state at $\tau=3.013737304\times10^{-4}$ gives
raw singular value $7.059704454\times10^{-10}$ and exact-rule observation
$7.276023568\times10^{-16}$. The two-direction minor and its
prediction are
\[
\begin{aligned}
\det[Dq_h^Q(u_\tau)\widetilde w,\ Dq_h^Q(u_\tau)U_1]
 &=7.345564635\times10^{-11},\\
\tfrac49\gamma_1\tau^2&=7.345557811\times10^{-11}.
\end{aligned}
\]
The one-cell $\mathbb Q_4$ comparison retains its single pressure equation
and all $81$ interior displacement coordinates. Use the corresponding
$p=4,n=5$ witness directions in $u_\tau=\tau(U_2+U_3)$. At
$\tau=2.952847445\times10^{-3}$ and three successive halves, its raw
singular values are $5.219414872\times10^{-8}$,
$1.304853718\times10^{-8}$, $3.262134292\times10^{-9}$ and
$8.155335708\times10^{-10}$. The exact-rule values remain near
$1.3395\times10^{-15}$. The missing-component action has the persistent
offset $-1.611\times10^{-16}$ from $(\chi_5/8)\tau^2$, including the same
rest-state evaluation term. This fixed offset explains the relative growth at smaller \(\tau\).

\subsection{Refinement construction and represented-input bounds}
\label{app:refinement-protocol}

To examine the cellwise construction under refinement, divide the unit
cube into $N^3$ affine cells, with $N\in\{1,2,4\}$, and use continuous
$\mathbb Q_3$ displacements with the complete outer boundary fixed.
The raw and exact-volume rules remain tensor Gauss four and five.
Two starting fields have different purposes. For each $N$, let $U_1,U_2,U_3$
be the physical push-forwards of the supported witness in the first cell.
The first field is their three-component ray. The second is obtained by
sampling and interpolating the normalized polynomial
\[
 \widetilde w(X)=16\prod_{i=1}^3X_i(1-X_i)e_1,\qquad
 w=\widetilde w/\max\{1,L_{\widetilde w}\},
\]
where $L_{\widetilde w}$ is the parent-monomial upper estimate for its
gradient in the ideal affine frame. This field moves interior faces for $N=2,4$;
the one-cell case is a boundary-zero control.

Let $L_j$ bound the maximum gradient of the supported direction $U_j$,
and put $L=L_1+L_2+L_3$. On the ray $u_\tau=\tau(U_1+U_2+U_3)$,
the exact cubic residual is $\gamma_1\tau^3$, with
$\gamma_1=\chi_4/(8N^3)$ for its first cell.
The four-corner correction therefore has each amplitude at most $\tau$,
so the initial and corrected gradient bounds are at most $2L\tau$.
We choose $\tau_0=1/(16\max\{1,L\})$ before the calculation.
For $u_0=\delta w$, put $C_\infty=L(56/\chi_4)^{1/3}$, using
$\gamma_0=\chi_4/8$ and the residual-growth estimate from
\cref{thm:cellwise-feasibility}; use
$\delta_0=\min\{1/8,(1/(8C_\infty))^3\}$.
Both families use four dyadic scales, indexed by $j=0,1,2,3$.
The evaluated first amplitudes are
$\tau_0=9.246120462\times10^{-4}$ and
$\delta_0=3.287737203\times10^{-12}$.

Whole-cell admissibility is checked on the stored inputs.
Let $t\in[0,1]^3$ be the uniform parameter coordinate, let $X_h(t)$
be the reference interpolant of the stored nodal positions, and write
$w_h(t)=u_h(X_h(t))$.
Exact promotion of the nodal values and interpolation abscissae to
rational numbers gives monomial-coefficient bounds
\[
 \|D_tX_h-I_3\|_2\le\rho_X,\qquad
 \|D_tw_h\|_2\le\rho_u.
\]
The bounds include the small difference between the represented reference
interpolant and its ideal affine map.
If $\rho_X+\rho_u<1$, the reference and current parameter maps are
Lipschitz perturbations of the identity with constants below one.
They are therefore globally injective and orientation preserving.
The chain rule and the smallest-singular-value bound give
\begin{equation}
 \|\nabla_Xu_h\|_2\le\frac{\rho_u}{1-\rho_X},
 \qquad
 \det(I_3+\nabla_Xu_h)
 \ge\left(\frac{1-\rho_X-\rho_u}{1+\rho_X}\right)^3>0.
 \label{eq:represented-input-admissibility}
\end{equation}
For the determinant estimate, the current parameter Jacobian has all
singular values at least $1-\rho_X-\rho_u$, while the reference determinant
is at most $(1+\rho_X)^3$. Its sign stays positive along the straight
homotopy from the identity.
All $24$ starting/correction pairs satisfy this check, with the physical
displacement-gradient bound below $1/4$.

On the supported ray, the $H^1$ correction norm decreases from
$6.461663\times10^{-4}$ to $8.077079\times10^{-5}$ for $N=1$,
from $2.272291\times10^{-4}$ to $2.840363\times10^{-5}$ for $N=2$,
and from $8.022895\times10^{-5}$ to $1.002862\times10^{-5}$ for $N=4$.
The cell residual displays its cubic scale, while the constructed norm
is linear in the ray amplitude.
\Cref{fig:journal-correction} compares the three refinements.

For the full-space field at its first scale, the largest raw cell residuals
are $2.343427\times10^{-14}$ and $2.241365\times10^{-14}$ for $N=2,4$.
Their direct corrected maxima are $3.75087\times10^{-21}$ and
$7.37792\times10^{-22}$.
The corresponding $H^1$ correction norms are
$2.586352\times10^{-4}$ and $3.484160\times10^{-4}$.
All face traces remain unchanged, and an isolated first-cell update
leaves every other cell's raw volume unchanged.

The retained exact cell-volume changes are resolved independently.
At the smallest scale, ideal-rule $80$-digit evaluation gives first-cell
exact changes $2.9292838872488705\times10^{-15}$ and
$5.788648109103542\times10^{-16}$ for $N=2,4$.
They are unchanged through all $65$ reported significant digits after
correction. The corresponding ideal-rule raw residuals are
$5.25015\times10^{-21}$ and $7.64509\times10^{-22}$.
Represented-rule evaluation gives different small residuals; this separates
quadrature representation from the exact-volume invariance.

For $N=1$, the global one-component starting field has exact zero volume
defect. Its direct residuals near $10^{-30}$ are a rounding control.
The first absolute residual increases from $8.82538\times10^{-30}$ to
$1.03045\times10^{-29}$ after a small rounding-driven correction;
its signed value changes from positive to negative.
Its norm is a constructed change rather than a closest distance.

\subsection{Conservative response and load amplitudes}
\label{app:tensor-response-protocol}

On the two-cell \(\mathbb Q_3\) model, use \(M\), \(v_0\) and
\(\ell=Mv_0\) from \cref{sec:evidence}.
The penalty coefficient is
$\kappa=\beta/\varepsilon^4$, with $\beta=3.4539763764$.
For its selection, an $M$-orthonormal basis of the resting kernel $K=\ker B_h$ expresses the
normalized aggregate cubic as a symmetric coefficient tensor.
Its Frobenius norm gives the bound $b=0.1098334683$, and we use
$\beta=1/(24b^2)$. The strength is conservative and \(v_0\) is an explicit witness load.

The stationary equations are solved in scaled displacement and pressure
coordinates, with an analytic Jacobian. The aggregate cubic is evaluated
directly, while the remaining pressure equation is scaled separately.
This invertible row transformation retains both pressure equations at
every positive $\varepsilon$. The raw and exact-volume stationary
references of the limiting energy are computed independently.

\begin{table}[H]
\centering
\caption{Finite-strain response on the two-cell tensor mesh.
The last column is $\|u_Q/\varepsilon-u_I/\varepsilon\|_M$, where
$u_Q,u_I$ are the computed stationary displacements. The limiting
stationary reference gives $1.069054666\times10^{-4}$.}
\label{tab:tensor-response}
\begin{tabular}{@{}rrr@{}}
\toprule
Load amplitude $\varepsilon$ & Bulk/shear $\kappa/\mu$
& Normalized separation\\
\midrule
$9.560363938\times10^{-4}$ & $2.06724\times10^{12}$ & $1.066431241\times10^{-4}$\\
$4.780181969\times10^{-4}$ & $3.30759\times10^{13}$ & $1.067743061\times10^{-4}$\\
$2.390090985\times10^{-4}$ & $5.29214\times10^{14}$ & $1.068398945\times10^{-4}$\\
\bottomrule
\end{tabular}
\end{table}

The six stationary solves converge with scaled residual at most
\(1.61\times10^{-12}\). Their normalized separation approaches
\(1.069054666\times10^{-4}\), about \(0.0107\%\), at the bulk/shear ratios
in \cref{tab:tensor-response}.

\paragraph{The extended load families.}
For the strengths \(\beta_{\rm s},\beta_{\rm l}\) defined in
\cref{sec:extended-response-observations}, the amplitude construction is
as follows.
For each strength, compute separate raw and exact stationary references
of the reduced sextic energy. Let $L_0$ bound the maximum gradient of
$v_0$, and let $L_Q,L_I$ be the corresponding polynomial upper estimates
for the raw and exact reduced references.
The prescribed first load amplitude is
\[
 \varepsilon_0=\min\left\{\frac18,\,
 \frac1{16\max\{1,L_0,L_Q,L_I\}}\right\},\qquad
 \varepsilon_j=\varepsilon_0/2^j,\quad j=0,1,2,3.
\]
For $\mathbb Q_3$, both strengths use
$\varepsilon_0=9.560363938\times10^{-4}$.
The $\mathbb Q_4$ conservative and load-normalized values are
$1.304338071\times10^{-3}$ and $6.480205150\times10^{-4}$.
A third family fixes $\kappa/\mu=1000$, hence $\kappa=2000$, on each
model's conservative amplitude grid.
Its initial state is the finite-bulk linear response
\[
 v_{\kappa}=(M+\kappa B_h^T\mathsf V^{-1}B_h)^{-1}\ell,
\]
with the corresponding linear pressure. The same scaled mixed equations
are used with $\beta=\kappa\varepsilon^4$.
Each solve starts from its prescribed reference.

\subsection{Force residuals and response conditioning}
\label{sec:tensor-force-norms}

For a nodal force residual \(f\), let \(K=\ker B_h\) and define the
restricted dual norm by
\begin{equation}
\|f|_K\|_{M,*}
=\sup_{\substack{v\in K\\\|v\|_M=1}}|f^Tv|
=\|T^Tf\|_2,
\label{eq:kernel-force-norm}
\end{equation}
where \(T\) is an \(M\)-orthonormal basis of \(K\).
The comparisons below use this norm alongside the full condensed-force
residual and local displacement estimates.

\subsection{Conditioning of the conservative tensor response}
\label{sec:conservative-response-precision}

For the two-cell \(\mathbb Q_3\) model of \cref{sec:tensor-response-evidence},
the reduced raw reference stopped with an objective-based trust-region
warning and gradient norm $1.30354\times10^{-11}$. Its smallest Hessian
eigenvalue was $0.99970085$. A direct Newton diagnostic changes the
normalized displacement by $1.30281\times10^{-11}$ and leaves gradient
norm $5.08\times10^{-17}$. The predicted objective decrease,
$8.49\times10^{-23}$, lies below the approximately $1.11\times10^{-16}$
rounding scale of the objective value. The refined separation is
$1.069054796\times10^{-4}$. The figure retains the original approximate
reference and reports its residual; the diagnostic quantifies its effect.

Pressure recovery has a different sensitivity. The formula
$p=\kappa\mathsf V^{-1}q_h(u)$ amplifies tiny volume errors at the bulk
ratios in \cref{tab:tensor-response}. At the smallest load amplitude,
direct binary64 condensed-force residuals are $9.51\times10^{-7}$ for
raw volume and $4.48\times10^{-6}$ for exact volume, despite the much
smaller scaled mixed residuals. Independent $80$-digit reevaluation of the
stored states uses both the represented binary64 Gauss rule and the
ideal Gauss rule. The stored displacements and loads are unchanged.
The kernel-force quantity uses \eqref{eq:kernel-force-norm}.

\begin{table}[H]
\centering\small
\caption{High-precision reevaluation at the smallest tensor-response
amplitude. For the condensed force residual $f$, the force column is its
Euclidean norm and the kernel-force column is its restricted dual
$M$-norm on $K=\ker B_h$. The last column is the $M$-norm
of a normalized displacement Newton estimate from the stable mixed
linearization. These are local sensitivity estimates, without a
rigorous error enclosure or a finite-parameter global-minimum certificate.}
\label{tab:tensor-response-precision}
\begin{tabular}{@{}llrrr@{}}
\toprule
Volume & Gauss rule & Force norm & Kernel force & Displacement estimate\\
\midrule
Raw & Represented & $5.51\times10^{-7}$ & $3.29\times10^{-14}$ & $1.39\times10^{-10}$\\
Raw & Ideal & $2.23\times10^{-5}$ & $2.24\times10^{-14}$ & $4.34\times10^{-13}$\\
Exact & Represented & $3.54\times10^{-7}$ & $1.35\times10^{-16}$ & $5.65\times10^{-13}$\\
Exact & Ideal & $1.47\times10^{-5}$ & $1.35\times10^{-16}$ & $5.65\times10^{-13}$\\
\bottomrule
\end{tabular}
\end{table}

The condensed residual is concentrated in a stiff regular-constraint
direction. Its displacement estimate is far below the
\(1.07\times10^{-4}\) separation. Higher arithmetic precision retains
the representation error of the stored displacement and load.

\subsection{Conditioning of the stronger stationary branches}
\label{sec:strong-response-precision}

The raw load-normalized limiting optimizers stop when an objective
improvement cannot be predicted accurately.
For $\mathbb Q_3$, the reduced gradient norm is $3.26398\times10^{-9}$
and the numerical minimum Hessian eigenvalue is $0.38447$.
One gradient-based Newton diagnostic has $M$-norm $8.23945\times10^{-9}$
and leaves gradient norm $1.06602\times10^{-16}$.
Its predicted objective decrease $1.32675\times10^{-17}$ lies below
the $1.02249\times10^{-16}$ rounding scale.
The corresponding $\mathbb Q_4$ gradient is $6.82713\times10^{-10}$,
with a nearly flat minimum Hessian eigenvalue $1.81514\times10^{-6}$.
Its diagnostic step has norm $1.02894\times10^{-6}$ and leaves gradient
norm $1.04957\times10^{-11}$.
The predicted decrease $1.03464\times10^{-18}$ is below the objective
rounding scale $1.07414\times10^{-16}$.
The response plots use the original references.

Each of the four smallest-load raw and exact states on the two models
is independently reevaluated at $80$ digits under both quadrature definitions.
With represented and ideal quadrature,
the raw $\mathbb Q_3$ normalized displacement Newton estimates are
$1.12375\times10^{-6}$ and $1.13175\times10^{-12}$.
Its total condensed-force norms are $0.01765$ and $0.78415$, whereas their
restricted dual-$M$ norms on the resting kernel are
$1.83932\times10^{-10}$ and $1.30750\times10^{-10}$.
Pressure-range amplification thus remains distinct from displacement
sensitivity. The exact-state displacement estimates are about
$6.47\times10^{-13}$.

For the raw $\mathbb Q_4$ state, the represented and ideal displacement
estimates are $1.22059\times10^{-4}$ and $4.94460\times10^{-7}$.
The represented estimate is far below the raw/exact gap $0.156405$,
but exceeds the stored raw finite-to-limit distance $6.67572\times10^{-5}$.
The broad response separation is therefore better resolved than this
fine branch correction.
Its restricted dual-$M$ force norms are $1.80447\times10^{-9}$ and
$1.30262\times10^{-15}$; the exact-state displacement estimates are
about $1.34616\times10^{-12}$.
These are local linearized sensitivities. The nearly flat
\(\mathbb Q_4\) branch limits the finer convergence comparison.

\subsection{Independent implementation checks}

The complete pressure Hessian contains both the volume-gradient Gram term
and the pressure-weighted volume Hessian. Independently computed mixed
second-variation actions on the polynomial examples differ by at most
$1.92\times10^{-16}$ in absolute value. Deliberately omitting the
pressure-geometric term produces a relative error of approximately $33.8\%$
in the selected finite-state control. The percentage measures the effect of the omitted Hessian term
in this control.

The separate pointwise material-Hessian checks have normalized
discrepancies no larger than $3.05\times10^{-16}$. Direct mixed equations
and their algebraically scaled equivalents agree to relative error at
most $7.12\times10^{-14}$.
Scaled finite-difference Jacobian comparisons show the expected
truncation-to-rounding transition. Exact rational coefficient checks
supply an independent arithmetic route for the local polynomial identities.

\subsection{Pressure-volume rank loss and restoration}
\label{sec:pressure-volume-evidence}

Use the dense rational witness \(w\) from
\cref{app:pressure-certificate} in \(y(t)=X+tw\).
Its binary64 path has an exact rational relative-gradient bound
below \(0.318\), giving injectivity and orientation on the convex cube.

\begin{table}[H]
\centering
\caption{Volume-map observations along the admissible path.
The volume error is evaluated by the centered cubic formula,
avoiding subtraction of two nearly equal total volumes.
The two ranks use the same binary64 numerical rank convention;
the rank behavior is also supported by
\Cref{thm:volume-regularity} and the nonzero curvature of $w$.}
\label{tab:volume-path}
\begin{tabular}{@{}rrrr@{}}
\toprule
$t$ & Total-volume defect & Smallest raw singular value
& Raw/exact rank\\
\midrule
0 & 0 & $3.78\times10^{-17}$ & 47/47\\
$2^{-20}$ & $-2.43\times10^{-16}$ & $7.35\times10^{-11}$ & 48/47\\
$2^{-16}$ & $-6.22\times10^{-14}$ & $1.18\times10^{-9}$ & 48/47\\
$2^{-12}$ & $-1.59\times10^{-11}$ & $1.88\times10^{-8}$ & 48/47\\
$2^{-8}$ & $-4.07\times10^{-9}$ & $3.02\times10^{-7}$ & 48/47\\
\bottomrule
\end{tabular}
\end{table}

The aggregate raw directional second variation remains approximately
$-5.33877\times10^{-4}$ along the path.
The exact and corrected aggregate values remain at rounding scale,
and the smallest non-gauge exact singular value remains above
$2.40\times10^{-4}$.
The four-Jacobian volume formula differs from an independent
degree-three rule by at most $4.511\times10^{-17}$ absolutely;
their gradients differ by at most $9.066\times10^{-16}$ relatively.

\subsection{Nonlinear feasible projection with no pressure bias}
\label{sec:projection-evidence}

The residual-distance consequence was tested through a nonlinear
displacement projection on the same $81$-coordinate template.
The fixed positive metric $M$ was its pressure-exact identity Hessian
with unit shear and bulk parameter $10$.
Bulk \(10\) defines the metric; the volume constraints are hard cell averages.
For target displacement $tw$, we minimized
$\|u-tw\|_M^2/2$ under three sets of constraints: all exact volumes,
all raw four-point volumes, and the reference-volume-weighted
mean-pressure projection of the raw volume equations.
The driving energy is quadratic, the constraints are nonlinear,
and the pressure bias is zero.

We used the scaled displacement $u=t\widehat u$.
The first $47$ raw volume equations were divided by $t$ and their
aggregate by $c_4t^2$; the latter was evaluated directly through
the centered cubic coefficients.
This row transformation preserves all $48$ raw equations for
each nonzero $t$.
The exact-volume route required $47$ independent rows.
The mean-projected route used the equivalent equations
$g_i^Q-(V_{i0}^Q/V_{48,0}^Q)g_{48}^Q=0$, $i=1,\ldots,47$.
Sequential quadratic programming used analytic objective and
constraint gradients, with tolerance $10^{-12}$ and at most
$300$ iterations. The limiting-cone reference in
\Cref{app:projection-reference} supplied the initial raw direction.
All twelve selected solves completed in $5$--$21$ iterations.

\begin{table}[H]
\centering
\caption{Relative distance from the target $tw$ in the fixed $M$ metric.
Raw distances approach the computed cone reference $0.6121334$;
exact and mean-projected distances decrease in proportion to $t$,
corresponding to second-order absolute displacement differences.
The final column is the maximum raw cell-volume residual of the
unprojected target $tw$, independently evaluated at its rounded nodes.}
\label{tab:nonlinear-projection}
\begin{tabular}{@{}rrrrr@{}}
\toprule
$t$ & Exact & Raw & Mean-projected raw & Target raw residual\\
\midrule
$2^{-8}$ & $3.2551\times10^{-3}$ & 0.61155
& $3.3225\times10^{-3}$ & $1.187\times10^{-7}$\\
$2^{-10}$ & $8.1453\times10^{-4}$ & 0.61210
& $8.3141\times10^{-4}$ & $7.417\times10^{-9}$\\
$2^{-12}$ & $2.0364\times10^{-4}$ & 0.61213
& $2.0787\times10^{-4}$ & $4.635\times10^{-10}$\\
$2^{-14}$ & $5.0913\times10^{-5}$ & 0.61213
& $5.1968\times10^{-5}$ & $2.897\times10^{-11}$\\
\bottomrule
\end{tabular}
\end{table}

Every rounded final map had a rational relative-gradient bound
below $0.323$, proving admissibility.
Independent $80$-digit integration of the ideal rules at those
rounded coordinates gave selected-constraint residuals no larger
than $1.080\times10^{-17}$.
Residuals evaluated in the scaled polynomial coordinates reach
$10^{-23}$. The rounded nodal geometry
has the independently evaluated residuals reported above.
The mean-projected solutions satisfy their selected equations
to the same numerical scale while retaining nonzero omitted raw
volume residuals between $8.746\times10^{-11}$ and
$2.136\times10^{-14}$.

The raw projection stays separated despite its target's vanishing
residual; mean-pressure projection recovers the first-order response
with a changed constraint set. The finite-\(t\) optimizer outputs
have no global-optimality certificate.

\subsection{Finite bulk stiffness and the displacement limit}
\label{sec:finite-penalty-evidence}

The finite-penalty comparison uses the same $81$ free coordinates,
reference geometry and positive quadratic FEM metric as
\cref{sec:projection-evidence}. The metric has unit shear and reference
bulk $10$. The background energy is the fixed quadratic
$\mathcal E(u)=u^TMu/2$; the added volume penalty uses the full nonlinear
maps $g^Q$ and $g^I$.

The target $v_*$ is the numerical extreme generalized eigenmode in
\cref{cor:finite-penalty-mode}, normalized to unit $M$-norm.
It differs from the earlier rational sparse witness. Its eigenvalue is
$\lambda=2.77485093\times10^{-4}$ and the normalization gives
$2/\lambda^2=2.59747085\times10^7$.
The dimensionless strength $c$ sets the bulk coefficient through
$\kappa=2c/(\lambda^2\varepsilon^2)$, where $\varepsilon$ is the load
amplitude in \cref{cor:finite-penalty-mode}.
Stationary equations are solved in rescaled displacement and pressure
coordinates, retaining all $48$ pressure equations. The mean-volume
equation uses the centered determinant identity. An analytic Jacobian
avoids finite-difference perturbations of small pressure components.

\begin{table}[H]
\centering\small
\caption{Finite-penalty displacement response. Distances are
$\|u/\varepsilon-v_*\|_M$. The predicted raw limits are $0.15229240$
for $c=1/4$ and $0.31767220$ for $c=1$; the exact-volume limit is zero.
The final row fixes the bulk coefficient and lets the normalized strength vary.}
\label{tab:finite-penalty}
\begin{tabular}{@{}rrrrr@{}}
\toprule
$c$ & $\varepsilon$ & $\kappa$ & Raw distance & Exact distance\\
\midrule
$1/4$ & $2^{-8}$ & $4.256\times10^{11}$ & $0.15232562$ & $0.00381842$\\
$1/4$ & $2^{-10}$ & $6.809\times10^{12}$ & $0.15229446$ & $0.00095577$\\
$1/4$ & $2^{-12}$ & $1.089\times10^{14}$ & $0.15229253$ & $0.00023896$\\
$1$ & $2^{-8}$ & $1.702\times10^{12}$ & $0.31757561$ & $0.00381842$\\
$1$ & $2^{-10}$ & $2.724\times10^{13}$ & $0.31766607$ & $0.00095577$\\
$1$ & $2^{-12}$ & $4.358\times10^{14}$ & $0.31767181$ & $0.00023896$\\
$5.874\times10^{-7}$ & $2^{-8}$ & $10^6$ & $0.00343422$ & $0.00346725$\\
\bottomrule
\end{tabular}
\end{table}

All fourteen stationary solves converge, with scaled mixed residual norms
at most $6.69\times10^{-14}$. The represented-coordinate relative-gradient
bound is below $0.25$ in every case, certifying admissibility.
At fixed $\kappa=10^6$, the raw and exact target-mode amplitudes are
$0.9999734754$ and $0.9999735584$.

The limiting minimum theorem concerns \eqref{eq:quartic-volume-limit};
the finite-parameter calculations establish stationarity, without
certifying minimality.

At the largest coefficient, pressure recovery is sensitive to coordinate
representation. For the exact-volume case, recovering cell pressures from
rounded binary64 absolute positions differs from the mixed pressure by
$0.3471$ in vector norm, against a mixed-pressure norm of
$1.2923\times10^{-4}$. Independent $80$-digit direct quadrature and
stationary refinement confirm the displacement effect: the changes in
$u/\varepsilon$ have $M$-norms $2.01\times10^{-15}$ for raw volumes and
$4.90\times10^{-15}$ for exact volumes. Their refined distances are
$0.3176718129608771$ and $0.0002389604427563$, with mixed residuals below
$5.3\times10^{-60}$. The displacement is resolved despite the pressure-recovery sensitivity.

\subsection{Geometric scope of the pressure-volume defect}

A ten-geometry comparison uses the same \(48\)-cell connectivity and
complete outer clamp.
Six geometries prescribe
$X(s)=s+\alpha e_3\prod_i s_i(1-s_i)$ at the quadratic nodes, with
$\alpha\in\{0,1/4,1/2,1,2,4\}$.
The other four use two fixed interior three-component integer nodal fields,
each scaled by $1/1024$ and $2/1024$; their entries lie between $-8$ and $8$
and their boundary values vanish. The actual integer fields are retained
by the authors with the verification reference. A reference-gradient perturbation bound below one
certifies injectivity and orientation for every case; its maximum is $0.45107$.

Exact rational integration gives total reference volume one and zero
constant-pressure force for the exact-volume map in all ten geometries.
The one-component maps also have exactly zero raw scalar and first-variation
defects. The generic maps already have a nonzero raw constant-pressure
force: the Euclidean norm of $D(\sum_e g_e^Q)(X)$ ranges from
$2.5811\times10^{-5}$ to $1.1391\times10^{-4}$.
The largest relative difference between raw and exact-volume linear
responses is $5.0088\times10^{-4}$, using one fixed nodal load and bulk
values $0,10,1000$. These are differences between the two prescribed
discretizations, not absolute engineering-accuracy measurements.

The affine control has exact volume-Jacobian rank $44$; every curved case
has rank $47$. The affine pressure kernel contains the constant mode and
three additional independent zero-mean modes, all retained in the exact
certificate. For an affine tetrahedron, the vertex basis
$N_i=\lambda_i(2\lambda_i-1)$ satisfies
$\int_T\nabla N_i=0$, since the mean of the barycentric coordinate
$\lambda_i$ is $1/4$. Curving activates these vertex columns, although
the midside-only block itself also changes from rank $44$ to $47$.
The rank recovery therefore cannot be assigned to vertex activation alone.

\subsection{Pressure-bias negative-mode control}

At $\eta=2^{17}$ and $\kappa=10$, the $48$-cell template has
$18$ negative four-point eigenvalues among $81$ free coordinates.
Eight scaled templates have $236$ negatives among $1029$ coordinates.
Both pressure-exact counterparts have zero observed negative
eigenvalues, with minimum generalized eigenvalues relative to
$G_\kappa$ of $0.9996119$ and $0.9993353$.
The certified lower bounds are one and eight; the larger numerical
counts have no exact-arithmetic certification.

\subsection{A prior-art-informed volume-variance control}
\label{sec:variance-evidence}

We evaluated \eqref{eq:variance-energy} on the same $81$-coordinate
template at $\kappa=10$ and $\eta=2^{17}$.
The variance Hessian was assembled as a positive weighted Gram matrix.
Its independent comparison with the pointwise-minus-cellwise volume
penalty agreed to $4.520\times10^{-16}$ relatively.
For the retained dense rational witness $w$, the raw tangent value
was $-67.9304$, while $w^TS_{\rm var}w=0.746264$.
The mean-volume action is exactly zero by its earlier rational
certificate.

For variance coefficient \(\alpha\), \cref{tab:variance-control}
gives numerical spectra relative to the same template's pressure-exact
Hessian \(K^I\).

\begin{table}[H]
\centering
\caption{Defined volume-variance control on the pressure-bias template.
Each row uses $K^Q+\alpha S_{\rm var}$. The final column is the
Frobenius norm of the added term divided by $\|K^I\|_F$.
The table neither implements the composite five-field method nor
estimates its engineering accuracy.}
\label{tab:variance-control}
\begin{tabular}{@{}rrrrr@{}}
\toprule
$\alpha$ & Negative count & Minimum eigenvalue & Maximum eigenvalue
& Added/reference norm\\
\midrule
0 & 18 & $-49.598$ & $51.889$ & 0\\
1 & 18 & $-49.426$ & $52.064$ & 0.0703\\
10 & 18 & $-47.996$ & $53.770$ & 0.7026\\
100 & 16 & $-40.741$ & $86.460$ & 7.026\\
1000 & 7 & $-12.942$ & $640.151$ & 70.26\\
$10^4$ & 0 & 0.07325 & 6342.7 & 702.6\\
$10^5$ & 0 & 0.94032 & 63386.4 & 7025.7\\
$10^6$ & 0 & 1.21821 & 633825.5 & 70256.8\\
\bottomrule
\end{tabular}
\end{table}

The first positive row in this grid is $\alpha=10^4$; the grid
does not determine a minimum stabilizing coefficient.
A seeded random right-hand-side solve has relative residual
$6.895\times10^{-14}$ for that controlled equation, while its
solution differs from the pressure-exact solution by $0.9694$
relatively. This is a linear operator-response comparison.
The variance term removes the observed negative directions with a
large stiffness change. Exact volume restores the null pressure
contribution without that parameter.

\subsection{Complete-operator scope in a loaded physical state}
\label{sec:retained-operator-scope}

The complete-pressure identity \eqref{eq:scope-identity} matters beyond
the small constraint examples. A pneumatic-finger/contact
configuration has $176{,}982$ free displacement coordinates and
$34{,}295$ quadratic tetrahedra. Reconstructing the complete energy Hessian
and the partial Hessian from the same state closes that identity to
$1.193\times10^{-16}$ in relative Frobenius norm.
The partial Hessian omits the two condensed-pressure contributions specified
in \eqref{eq:scope-identity}. A cuDSS sparse factorization reports
$22$ negative eigenvalues for this partial Hessian and zero for the complete
energy Hessian, for both neo-Hookean and finite-extensibility Gent
reconstructions \citep{gent1996constitutive}.
These inertia counts are numerical estimates from the sparse factorization.

A separate loaded-state volume correction has an unfavorable sign.
For its partial Hessian, exact-minus-four-point correction changes
the numerical negative count from $21$ to $36$ and the fixed-vertex midside
block count from zero to seven. Its assembled correction norm is
$4.67165\times10^{-5}$ relative to the partial-Hessian norm.
The cell-Hessian identity agrees with independent integration to
absolute error at most $4.60059\times10^{-14}$.
This loaded, mixed-boundary partial Hessian differs from the
complete fully clamped operator in the pressure-bias theorem.
Its sign change precludes an unconditional positive-tangent
interpretation of exact-volume correction.

\section{First-invariant constitutive identities}
\label{app:constitutive-identities}

This appendix supplies the constitutive identities used by the
Tet10 mechanical examples. The isochoric--volumetric separation follows the
classical Flory setting \citep{flory1961thermodynamic}; generalized
first-invariant laws are considered on their admitted constitutive domains
\citep{hartmann2003polyconvexity}.

Let $F\in\GLp$ be a deformation gradient and $H,L\in\R^{3\times3}$
matrix increments. Put
\[
J=\det F,\qquad I_1=F{:}F,\qquad
\bar I_1=J^{-2/3}I_1,\qquad
\theta_H=F^{-T}{:}H=\tr(F^{-1}H),
\]
with the analogous definition for $\theta_L$.
Determinant differentiation gives
\begin{align}
DJ(F)[H]&=J\theta_H,\label{eq:det-first}\\
D^2J(F)[H,L]
&=J\bigl(\theta_H\theta_L-\tr(F^{-1}HF^{-1}L)\bigr).
\label{eq:det-second}
\end{align}
These identities follow by differentiating determinant multiplicativity
and $D(F^{-1})[L]=-F^{-1}LF^{-1}$.

Let $\mathcal S$ be an admitted interval of modified-invariant values,
$f\in C^2(\mathcal S)$ and $g\in C^2((0,\infty))$. The density is
\begin{equation}
W(F)=f(\bar I_1)+g(J).
\label{eq:material-class}
\end{equation}
Here $g$ is a scalar volumetric potential; the superscripted maps $g^r(y)$
in \Cref{sec:pressure-volume} instead collect discrete cell-volume changes.
No existence, global calibration or polyconvexity assertion is required
for the pointwise identities below.

Define the scalar linear functional
\begin{equation}
\ell_F(H)=F{:}H-\frac{I_1}{3}\theta_H
\label{eq:ell-definition}
\end{equation}
and the symmetric bilinear form
\begin{align}
Q_F(H,L)
={}&\left(H-\frac23\theta_HF\right){:}
       \left(L-\frac23\theta_LF\right)
       +\frac{I_1}{9}\theta_H\theta_L.
\label{eq:q-definition}
\end{align}
Its quadratic value is a sum of squares:
\begin{equation}
Q_F(H,H)=
\left\lVert H-\frac23\theta_HF\right\rVert_F^2
+\frac{I_1}{9}\theta_H^2.
\label{eq:q-square}
\end{equation}
Equivalently, the linear factor map
\begin{equation}
\mathcal B_F(H)=
\left(H-\frac23\theta_HF,\frac{\sqrt{I_1}}{3}\theta_H\right)
\in\R^{3\times3}\times\R
\label{eq:flory-factor-map}
\end{equation}
gives $Q_F(H,L)=\langle\mathcal B_F(H),\mathcal B_F(L)\rangle$
in the standard product inner product.

\begin{lemma}[Positive Flory base]
\label{lem:q-positive}
For every $F\in\GLp$, $Q_F$ is an inner product on
$\R^{3\times3}$.
\end{lemma}
\begin{proof}
Symmetry and bilinearity follow from \eqref{eq:q-definition}. Suppose
$Q_F(H,H)=0$. Since $I_1>0$, both nonnegative terms in
\eqref{eq:q-square} vanish, giving $\theta_H=0$ and
$H-(2/3)\theta_HF=0$. Hence $H=0$, which proves positive definiteness.
\end{proof}

\begin{theorem}[Generalized first-invariant factorization]
\label{thm:general-factor}
Let $F\in\GLp$, let $f\in C^2(\mathcal S)$ at
$\bar I_1(F)$, and let $H,L\in\R^{3\times3}$. Then
\begin{align}
D\bar I_1(F)[H]
={}&2J^{-2/3}\ell_F(H),\label{eq:first-modified}\\
D^2\bar I_1(F)[H,L]
={}&2J^{-2/3}Q_F(H,L)
-\frac{2I_1}{3J^{5/3}}D^2J(F)[H,L].
\label{eq:second-modified}
\end{align}
Define the generalized gauge and its induced determinant coefficient by
\begin{align}
\widetilde A_f(F)[H,L]
={}&2f'(\bar I_1)J^{-2/3}Q_F(H,L)
+4f''(\bar I_1)J^{-4/3}\ell_F(H)\ell_F(L),
\label{eq:general-gauge}\\
\pi_f(F)&=\frac{2f'(\bar I_1)I_1}{3J^{5/3}}.
\label{eq:induced-pressure}
\end{align}
Then
\begin{equation}
D^2(f\circ\bar I_1)(F)[H,L]
=\widetilde A_f(F)[H,L]-\pi_f(F)D^2J(F)[H,L].
\label{eq:f-factorization}
\end{equation}
For the full density \eqref{eq:material-class},
\begin{equation}
D^2W(F)=\widetilde A_f+g''(J)DJ\otimes DJ
+\bigl[g'(J)-\pi_f(F)\bigr]D^2J,
\label{eq:full-material-factor}
\end{equation}
where $(DJ\otimes DJ)[H,L]=DJ[H]DJ[L]$.
\end{theorem}
\begin{proof}
Differentiate $I_1=F{:}F$ and $J^{-2/3}$. Using
$D(J^{-2/3})[H]=-(2/3)J^{-2/3}\theta_H$ gives
\[
D\bar I_1[H]
=2J^{-2/3}\left(F{:}H-\frac{I_1}{3}\theta_H\right),
\]
which is \eqref{eq:first-modified}. A second differentiation gives
\begin{align*}
D^2\bar I_1[H,L]
={}&2J^{-2/3}H{:}L
-\frac43J^{-2/3}\bigl(\theta_LF{:}H+\theta_HF{:}L\bigr)\\
&+\frac{4I_1}{9}J^{-2/3}\theta_H\theta_L
+\frac{2I_1}{3}J^{-2/3}\tr(F^{-1}HF^{-1}L).
\end{align*}
Expanding \eqref{eq:q-definition} and substituting
\eqref{eq:det-second} gives \eqref{eq:second-modified}.
The scalar chain rule
\[
D^2(f\circ\bar I_1)
=f''D\bar I_1\otimes D\bar I_1+f'D^2\bar I_1
\]
proves \eqref{eq:general-gauge}--\eqref{eq:f-factorization}.
Finally, $D^2(g\circ J)=g''DJ\otimes DJ+g'D^2J$ proves
\eqref{eq:full-material-factor}.
\end{proof}

For the neo-Hookean law $f_{\rm NH}(s)=\mu(s-3)/2$, with shear modulus
$\mu>0$, the identity specializes to
\begin{equation}
D^2f_{\rm NH}(\bar I_1)[H,H]
=\mu J^{-2/3}Q_F(H,H)
-\frac{\mu I_1}{3J^{5/3}}D^2J(F)[H,H].
\label{eq:nh-completion}
\end{equation}
The first term is a positive pointwise reference form. The second is the
induced-pressure determinant curvature, whose integrated behavior depends
on the declared geometry, cubature and boundary conditions. At $F=I_3$,
$Q_{I_3}(H,L)=H{:}L+(\tr H)(\tr L)/3$, giving the unit-shear form used
in \Cref{lem:pressure-template}.

\section*{CRediT author statement}
Yanlin Liu originated the central theoretical ideas and mathematical
formulations and led the theoretical development.

\noindent\textbf{Yanlin Liu:} Conceptualization (central theoretical ideas);
Methodology (mathematical formulation and method development);
Formal analysis (theoretical analysis and derivations);
Software (research-specific development, framework extensions and performance
optimization); Investigation; Validation; Writing---original draft;
Writing---review and editing.

\noindent\textbf{Chao Huang:} Software (development of the foundational
simulation framework); Conceptualization (introduction of the underlying
numerical problem and the near-incompressibility regime).

\noindent\textbf{Kaixiang Yao:} Conceptualization (introduction of the
motivating Tet10 problem and its robotics application context);
Writing---review and editing.

\noindent\textbf{Yao Shen:} Supervision; Resources (provision of computational
resources); Conceptualization (broad problem introduction and research context).

\section*{Data and code availability}
Exact certificates, verification scripts, controlled inputs, and retained
numerical outputs are available from the authors upon reasonable request.
A public reproducibility repository is planned for a subsequent version.
The private simulation platform and its production solvers are outside the
planned release.

\section*{Declaration of generative AI use}
OpenAI Codex was used to assist with experimental coding, preliminary
mathematical review, and preliminary drafting. Anthropic Claude was
used for draft reviewing and editing. Yanlin Liu provided the central
theoretical ideas and initial mathematical formulations.
Responsibility for the claims, citations, computational results,
and final manuscript rests with the named authors.

\bibliographystyle{unsrtnat}
\bibliography{references}

\end{document}